\documentclass[11pt, reqno]{amsart}
\usepackage{amsmath, amssymb, amsthm, graphicx,marvosym}
\usepackage{soul,bbm}
\usepackage{cancel}
\usepackage{cite}
\usepackage[nobysame]{amsrefs}[11pts]
\usepackage{bm}
\usepackage{color}
\usepackage{geometry}
\allowdisplaybreaks
\newtheorem{theorem}{Theorem}[section]
\newtheorem{definition}{Definition}[section]
\newtheorem{lemma}{Lemma}[section]
\newtheorem{remark}{Remark}[section]

\numberwithin{equation}{section}
\numberwithin{figure}{section}

\makeatletter \@addtoreset{equation}{section} \makeatother

\newcommand{\R}{\mathbb{R}}

\newcommand{\eps}{\varepsilon}
\newcommand{\dive}{\mathrm{div}}

\newcommand{\pa}{\partial}
\newcommand{\pt}{\partial_t}
\newcommand{\vp}{\varphi}

\newcommand{\al}{\alpha}
\newcommand{\D}{\nabla}
\newcommand{\hu}{\hat{u}}

\newcommand{\Id}{\mathbb{I}_d}
\newcommand{\cuti}[1]{{\textbf{\underline{#1}}}}

\newcommand{\Y}{\mathcal{Y}}
\newcommand{\tn}{\tilde{n}}
\newcommand{\tw}{\tilde{w}}
\newcommand{\tvp}{\tilde{\vp}}

\newcommand{\supp}{\text{supp}}
\newcommand{\mk}{\mathcal{K}}

\newcommand{\rmd}{{\rm{d}}}
\newcommand{\gp}{\mathcal{G}}

\usepackage[colorlinks,
linkcolor=black,
anchorcolor=black,
citecolor=black
]{hyperref}
\allowdisplaybreaks
\begin{document}

\title[Euler-Poisson with vacuum]{Global-in-time convergence from bipolar Euler-Poisson equations to unipolar ones}
% allowing Far Field Vacuum}

\author{Zhongmin Qian}
\address[Z. Qian]{Mathematical Institute, University of Oxford,	Oxford OX2 6GG, UK
%The Institute of Mathematical Sciences, Chinese University of Hong Kong, Hong Kong.
}
\email{\tt zhongmin.qian@maths.ox.ac.uk}

\author{Liang Zhao}
\address[L. Zhao]{Department of Mathematics, National University of Singapore, Singapore 119076, Singapore.}
\email{\tt lz.sh.pde@gmail.com}

\author{Shengguo Zhu}
\address[S.  Zhu]{School of Mathematical Sciences, CMA-Shanghai and MOE-LSC,  Shanghai Jiao Tong University, Shanghai 200240, P.R. China.}
\email{\tt zhushengguo@sjtu.edu.cn}


\date{\today}

\begin{abstract}
In this paper, the Cauchy problem for the multi-dimensional (M-D) bipolar Euler-Poisson equations with far field vacuum is considered. Based on physical observations and  some elaborate analysis of this system's intrinsic  structures, for a class of smooth initial data that are of small  scaled  density but possibly large  mean velocity,  we give one rigorous   global-in-time convergence proof for regular solutions   from M-D bipolar  Euler-Poisson equations to M-D unipolar Euler-Poisson equations through the vanishing electron-to-ion mass-ratio limit. Here the initial scaled  density is required to vanish in the far field, and the spectrum of the Jacobian matrix of the initial mean velocity stays uniformly away from the nonpositive real axis. In order to deal with singular  limits of this kind, the global-in-time uniform tame estimates of regular solutions to M-D bipolar  Euler-Poisson equations with respect to the mass ratio are established, based on which  the  corresponding error estimates in smooth function spaces   between the two systems considered   are given.  To this end, we establish global \textit{a priori} estimates for solutions with compactly supported initial densities, uniformly in the mass ratio and the initial support radii. Compactness then yields global regular solutions for general initial densities.
\end{abstract}

\maketitle

%\tableofcontents

\vspace{2mm}

\section{Introduction}

The d-dimensional scaled bipolar Euler-Poisson system reads (\cites{Chen1984,Markowich1990,Sitenko1995}):
\begin{equation}\label{origin1}
\begin{cases} \pt \rho_\nu  +\dive(\rho_\nu u_\nu)=0,\\[6pt]
		m_\nu\pt(\rho_\nu u_\nu)+m_\nu\dive (\rho_\nu u_\nu\otimes u_\nu)+\nabla p_\nu(\rho_\nu)= q_\nu\rho_\nu\nabla\vp,\\[6pt]
		\Delta \vp=\rho_i-\rho_e, \end{cases}
\end{equation}
which plays an important role in describing movements of charged fluids in semiconductors or plasma when the magnetic field is weak. Here, $x=(x_1,...,x_d)\in \R^d$, $t\geq 0$ are the space and time variables, respectively; $\nu=i,e$ with $i$ standing for ions and $e$ for electrons;   $\rho_\nu\geq 0$ denotes the scaled  density, $u_\nu=(u_\nu^{(1)}, \cdots, u_\nu^{(d)})^\top$ the mean velocity  of the particle, $\vp$ the scaled electric potential, and $p_\nu(\rho_\nu)$ the pressure. The physical parameters $(q_e,q_i)=(-1,1)$ and $m_\nu>0$ stand for the charge and mass of a single particle; the ratio $m_e/m_i$ is small in cases of physical interest. As usual, we denote the mass ratio by 
$$m_e/m_i=\eps^2$$ with   $\eps\in (0,1]$. Moreover,  for
polytropic fluids, the constitutive relations are given by
\begin{equation}\label{pressure}
	p_\nu=A_\nu\rho_\nu^{\gamma_\nu}, \quad A_\nu>0,\quad  \gamma_\nu> 1,
\end{equation}
where $A_\nu$ are entropy constants and $\gamma_\nu$ are adiabatic exponents.  In this paper, for  the Cauchy problem (\ref{origin1})-(\ref{pressure}) with the following  initial data and far field behavior:
\begin{align}
	\label{oini} (\rho_\nu,u_\nu)(0,x)=(\rho_\nu^0(x)\geq 0,u_\nu^0(x)) \qquad\text{for}&\,\, \,\,x\in \R^d,\\[6pt]
	\label{ofar} (\rho_\nu,\D\vp)(t,x)\to 0 \quad \text{as}\,\, |x|\to \infty \qquad\text{for}& \,\, \,\,t\geq 0,
\end{align}
we will establish  the uniform global  existence of  solutions in smooth function spaces  with respect to $\eps$, which allows us to show  rigorously  the global-in-time convergence and  error estimates  for these solutions  in the vanishing mass-ratio (VMR for short) limit: 
\begin{equation}\label{infmass}
	\sqrt{m_e/m_i}=\eps\to 0,
\end{equation}
from  the  classical M-D bipolar Euler-Poisson equations to unipolar ones for electrons:
\begin{equation}\label{Euler0}
\begin{cases} \pt \rho_e  +\dive(\rho_e u_e)=0,\\[6pt]
		\pt(\rho_e u_e)+\dive (\rho_e u_e\otimes u_e)+\nabla p_e(\rho_e)=-\rho_e\D\vp,\\[6pt]
		\Delta\vp=b(t,x)-\rho_e, \end{cases}
\end{equation}
where $b(t,x)$ is called the doping profile, which represents the background density.

For fixed  $m_\nu=1$, when the initial data are away from the vacuum,  there is a rich literature on the well-posedness  and behavior of solutions to the bipolar Euler-Poisson equations \eqref{origin1}. The local well-posedness of classical solutions  follows from the standard symmetric hyperbolic structure, cf. \cites{Chen2002, Kato1975, Lax1973,Majda1984}. A typical phenomenon in such Euler-type systems is the development of finite-time singularities from smooth initial data. See, for example, Wang-Chen \cite{Wang1998} and Yuen \cite{Yuen2011}.  However, such an approach to the local well-posedness  fails  in the presence of the vacuum due to the degeneracies of the time evolution. Generally vacuum will appear in the far field under some physical requirements such as finite total mass and finite momentum in  $ \mathbb{R}^d$.  For the isentropic flow, one of the  main issues in the presence of vacuum is   to understand the behavior of the  velocity field near the vacuum. In 1986, when the initial density vanishes in some open domain or decays to zero in the far field, via introducing the local sound speed
$$
c=\sqrt{p'(\rho)}=\sqrt{A\gamma} \rho^{\frac{\gamma-1}{2}}
$$
to provide  a new  symmetrization scheme,   Makino-Ukai-Kawashima \cite{Makino1986} gave   the  local existence of  the unique regular solution in an inhomogeneous Sobolev space  to the  M-D compressible Euler equations, and  the  global well-posedness of smooth solutions   with small data  was established  by Grassin \cite{Grassin1998} and Serre \cite{Serre1997} in some homogeneous Sobolev spaces via extracting a dispersive effect after some invariant transformation (see also  Blanc-Danchin-Ducomet-Ne\v{c}asov\'{a} \cite{Blanc2021} and Danchin-Ducomet \cite{Danchin2021} for \eqref{Euler0}). We also refer readers to \cites{Ali2003,Athanasiou2023,Chen2021,Chen2017,Lin1,GerardVaret2013,Germain2013, Gu2016, Guo2011,Guo2014a, Strauss1,Hadzic 2019,Hsiao2003,Li2012,Lin2,Strauss2,Xu2015,Zheng2019,Huang2012,Hsiao2000} for other related progress.

Due to the complex hyperbolic-elliptic coupled  structure of the bipolar Euler-Poisson equations \eqref{origin1}, both the mathematical analysis and numerical simulations of their solutions face great challenges, in which high computing power, special algorithms and strong techniques are needed, especially for M-D cases \cites{Jiang1995,Nisar2016}. Therefore,  physicists and mathematicians try to approximate the bipolar Euler-Poisson equations with some simpler models, which should be physical, less restrictive and easy to carry out. In plasma physics \cites{Chen1984,Sitenko1995}, the ion mass is much larger than the electron mass, so that $m_e/m_i=\eps^2\ll1$. For example, in a proton--electron plasma, $m_e/m_i\approx5.45\times10^{-4}$.  It is then a natural question to establish rigorous analysis in simplifying the bipolar Euler-Poisson equations \eqref{origin1} by the VMR limit \eqref{infmass} on different scales to obtain the corresponding simplified unipolar Euler-Poisson equations. On the ion scale, the electron inertia is negligible compared with the ion dynamics. Hence the VMR limit \eqref{infmass} takes the form of what is conventionally referred to as the zero-electron-mass limit. We refer readers to Al\`i-Chen-J\"{u}ngel-Peng \cite{Ali2010}, Acheritogaray-Degond-Frouvelle-Liu \cite{Acheritogaray2011}, Guo-Pausader \cite{Guo2011}, Xu-Zhang \cite{Xu2013} and references cited therein for more details. In the present paper, we study the same VMR limit on the electron scale, where the ions are so heavy that they respond only weakly to pressure gradients and electrostatic forces. Introduce the following electron-time scaling
\[
t=\sqrt{m_e}\tilde t,\qquad \tilde \rho_\nu(\tilde t,x)=\rho_\nu(t,x),\quad \tilde u_\nu(\tilde t,x)=\sqrt{m_e} u_\nu(t,x), \quad \tilde \vp(\tilde t,x)=\vp(t,x),
\]
and still denote $(\tilde t,\tilde\rho_\nu,\tilde u_\nu,\tilde\vp)$ as $(t,\rho_\nu,u_\nu,\vp)$ for simplicity. Then one has
\begin{equation}\label{main-afterscaling}
\begin{cases} \pt \rho_e  +\dive(\rho_e u_e)=0,\\[6pt]
		\pt(\rho_e u_e)+\dive (\rho_e u_e\otimes u_e)+\nabla p_e(\rho_e)=-\rho_e\nabla\vp,\\[6pt]
		\pt \rho_i +\dive(\rho_i u_i)=0,\\[6pt]
		\pt(\rho_i u_i)+\dive (\rho_i u_i\otimes u_i)+\eps^2\nabla p_i(\rho_i)=\eps^2\rho_i\nabla\vp,\\[6pt]
		\Delta \vp=\rho_i-\rho_e. \end{cases}
\end{equation}
Denote by  $(\bar{\rho}_\nu,\bar{u}_\nu,\bar{\vp})$ the limit of $(\rho_\nu,u_\nu,\vp)$ and formally let $\eps\to 0$ in \eqref{main-afterscaling}$_3$--\eqref{main-afterscaling}$_4$. Then
\begin{equation}\label{assm-0}
	\pt\bar{\rho}_i+\dive(\bar{\rho}_i\bar{u}_i)=0,\quad \pt \bar{u}_i+(\bar{u}_i\cdot\D)\bar{u}_i=0.
\end{equation}
Further assuming the initial data of $u_i$ tend to zero as $\eps\to 0$, one obtains $\bar{u}_i=0$ and thus $\bar{\rho}_i$ is a function of $x$, which we regard as the background density and denote as $b(x)$. Substituting the solved $\bar{\rho}_i$ into the limit equations for $(\rho_e,u_e,\vp)$ will lead to a decoupled system for electrons  \eqref{Euler0} (See also \cites{Xi2024,Zhao2021}). The physical assumption for these kinds of simplifications is that ions move more slowly than electrons, so that they can be regarded as non-moving, and $\bar{\rho}_i(x)$ becomes a time-independent background density. However, the key idea of the simplification through the VMR limit for bipolar systems is to decouple the limit equations for ions and electrons. Consequently, one only needs to ensure the global-in-time well-posedness of \eqref{assm-0}  (more specifically, the Burgers equations for $\bar{u}_i$) to carry out subsequent analysis of this limit rather than requiring  $\bar{u}_i=0$, which is so strong that it eliminates the possibility that the doping profile $b$ may depend on the time. It is then quite natural to consider the  Cauchy problem of the Burgers equations with rarefaction initial data $u_i^0$.  Hence $u_i^0$ is required to satisfy a spectral condition on $\D u_i^0$. 

In order to carry out the global-in-time analysis of the VMR limit on the electron scale, the first task is to identify a  proper class of initial data and the corresponding  mechanism for propagating regularity in time in  some energy space to establish the desired uniform  global-in-time existence of regular solutions with respect to $\eps$. Some essential difficulties occur accordingly. On the one hand, the spectral condition on $\D u_i^0$ makes $u_i$ and $\D u_i$ non-integrable over $\R^d$.  On the other hand, compared with the Euler equations studied in Serre \cite{Serre1997} and Grassin \cite{Grassin1998}, system \eqref{origin1} contains the additional  term $\nabla\vp$, which creates additional difficulties. First, in order to deal with the degeneracy in the time evolution caused by the appearance of vacuum, one needs to apply the Makino-Ukai-Kawashima  symmetrization scheme introduced in  \cite{Makino1986} to   \eqref{origin1}.  However, this scheme will  result in the destruction of the dissipative structure of $\D\vp$ in \eqref{origin1} and the loss of the proper entropy inequalities for the error system. Besides, the conservative structure and the associated physical energy cancellations are not directly visible in the symmetrized variables, which  are  necessary for  obtaining  the zeroth-order estimates on the error variables. Second, although it seems that $\D\vp$ can be bounded by densities through the Poisson equation, this term is non-local and has strong coupling effects, which imply that in the corresponding weighted energy estimates, the power of  $\eps$ in the weight function  of  $\D\vp$ in both ion and electron momentum equations should be exactly the same. Due to this, we have to consider $(\rho_\nu,\sqrt{m_\nu}u_\nu)$ as variables instead of $(\frac{1}{\sqrt{m_\nu}}\rho_\nu, u_\nu)$, which results in the vanishing of $u_i^0$ after the limit is applied, leading to a contradiction with the spectral condition on $\D u_i^0$. Last but not least, due to the Poisson structure, $\vp$ should be some harmonic function when the densities $\rho_\nu$ degenerate, not necessarily $0$. This implies that in the vacuum region, the time evolution of the velocities does not simply degenerate into the Burgers equations, which makes it hard for us to describe the dynamics of velocities. In order to conquer these obstacles, we introduce the perturbation variables: $$w_e=u_e-\hu_e\quad \text{and} \quad \eps^{-1}w_i=\eps^{-1}(u_i-\hu_i)$$ with $\hu_\nu$ being the solutions to the following Cauchy problem of the Burgers equations:
\begin{equation}\label{burgers}
	\pt \hu_\nu+(\hu_\nu\cdot\D)\hu_\nu =0 \quad \text{in} \ \ (0,T]\times\R^d;\quad \hu_\nu(0,x)=u_\nu^0(x) \quad\text{for} \quad x\in\R^d.
\end{equation}
The boundedness of $\eps^{-1}w_i$ in some energy spaces fixes the paradox between the vanishing condition and the dispersive condition on $u_i^0$, which, in addition, provides a time-dependent damping mechanism necessary for global-in-time analysis of solutions. Based on this, a series of approximate problems having truncated convection operators and compactly supported initial data is carefully designed. For these approximate problems, we first establish local estimates uniform in the convection truncation parameter. After passing to the limit in this parameter, we derive global estimates uniform in \(\varepsilon\) and the size of the initial density supports. Furthermore, our reformulation naturally admits some dissipative structure  for error systems, which  makes it possible to obtain the global-in-time convergence towards \eqref{Euler0} and get the global-in-time error estimates between \eqref{origin1} and \eqref{Euler0}.

Throughout this paper, we always assume $m_e=1$ and $m_i=\eps^{-2}$ for simplicity. We also adopt the following simplified notations,
\begin{equation*}\begin{split}
		& L^p=L^p(\R^d),\quad H^r=H^r(\R^d),\quad |f|_p=\|f\|_{L^p(\R^d)},\\[3pt]
		&\|f\|_r=\|f\|_{H^r(\R^d)},\quad  \|u\|_{\dot{H}^r}:=|\D^r u|_2,\quad \Gamma=L^\infty\cap \dot{H}^1\cap\dot{H}^s, \\[3pt]
		&\Xi=\{f\in L^1_{\text{loc}}(\mathbb{R}^d): |\nabla f|_\infty+\|\nabla^2 f\|_{s-1}<\infty\}, \\[3pt]
		&  \|f(t,x)\|_{T,\Xi}=\|\nabla f(t,x)\|_{L^\infty([0,T]\times\mathbb{R}^d)}+\|\nabla^2 f(t,x)\|_{L^\infty([0,T]; H^{s-1}(\mathbb{R}^d))},\\[3pt]
		&  \|f\|_{\Xi}=|\nabla f|_\infty+\|\nabla^2 f\|_{s-1},\quad \|f\|_\Gamma=|f|_\infty+\|f\|_{\dot{H}^1}+\|f\|_{\dot{H}^s},\\[3pt]
		&\|f\|_{\Gamma\cap L^q}=\|f\|_\Gamma+|f|_q, \quad \|(f,g)\|_X=\|f\|_{X}+\|g\|_{X}.
\end{split}
\end{equation*}
\begin{remark}
In all weak and weak-* convergence statements, the limits are taken in each factor separately, and $L^\infty$ factors carry the weak-* topology. To fix the spatial constants in $\Xi$, we additionally require
\[
	f_k(0) \longrightarrow f(0), \qquad \|f_k(\cdot,0)-f(\cdot,0)\|_{L^\infty(0,T;\mathbb{R}^d)} \longrightarrow 0,
\]
for spatial and time-dependent convergence, respectively. Here $T>0$ is arbitrary and finite, and point values at the origin are taken using spatially continuous representatives.
\end{remark}
We first introduce a proper class of solutions called regular solutions to the Cauchy problem \eqref{origin1} with \eqref{oini}-\eqref{ofar}.

\begin{definition}\label{regular-solution} Let $T>0$ be a constant. $(\rho_\nu, u_\nu, \vp)(t,x)$  is called a regular solution to the  Cauchy problem \eqref{origin1} with \eqref{oini}-\eqref{ofar} in  $[0,T]\times \R^d$, if $(\rho_\nu, u_\nu, \vp)(t,x)$ solves this problem in the sense of distributions,  and
\begin{itemize}
\item[$(\rm A)$] $\rho_\nu\geq 0$,\, $\rho_\nu^{\frac{\gamma_\nu-1}{2}}\in C([0,T]; H_{\rm{loc}}^{s'})\cap L^\infty([0,T]; \Gamma\cap L^q)$;

\vspace{3mm}

\item[$(\rm B)$] $(u_\nu-\hu_\nu,\D\vp)\in C([0,T];H_{\rm{loc}}^{s'})\cap L^\infty([0,T];\Gamma)$;

\vspace{3mm}

\item[$(\rm C)$] $m_\nu(\pt u_\nu+(u_\nu\cdot\D)u_\nu)=q_\nu\D\vp$  \ \text{as} \  $\rho_\nu=0$,
\end{itemize}
where $s'\in[0,s)$ is an arbitrary constant and $q$ satisfies \eqref{res1}.
\end{definition}

Now we are ready to show  the main results in this  paper. For simplicity, we denote 
$$
\underline{\gamma}=\min\{\gamma_i,\gamma_e\},\quad \bar{\gamma}=\max\{\gamma_i,\gamma_e\}.
$$
The first theorem  concerns the global existence and uniform estimates with respect to $\eps$.

\begin{theorem}\label{globalthm} Let  the parameters $(s,d,q,\gamma_\nu)$ be such that $s,d\geq 3$ are  integers, and
\begin{equation}\label{res1}
		q>\frac{d}{d-1},\qquad \dfrac{d}{2}+1<s< \frac{2}{\bar{\gamma}-1}+\frac{1}{2}, \qquad 1<\gamma_\nu<1+\min\Big\{\frac{2}{3},\frac{4}{d+1},\dfrac{4}{q},\dfrac{2d}{d+q}\Big\},
\end{equation}
and there is no upper bound for $s$ if $\frac{2}{\gamma_\nu-1}$ are integers. If the   initial data $( \rho_\nu^0, u_\nu^0)$ satisfy
\begin{itemize}
\item[$(\rm A_1)$] $u_\nu^0\in \Xi$, and there exists a constant  $\kappa>0$ independent of $\eps$ such that
$$
		\text{Dist}\big(\text{Sp}( \nabla u_\nu^0(x)), \mathbb{R}_{-} \big)\geq \kappa \quad \text{for\  all}\quad  x\in \mathbb{R}^d,
$$

\item[$(\rm A_2)$] $\rho_\nu^0\geq 0$\  and  \ $\Big\| (\rho_\nu^0)^{\frac{\gamma_\nu-1}{2}}\Big\|_\Gamma+\Big| (\rho_\nu^0)^{\frac{\gamma_\nu-1}{2}}\Big|_q\leq B_\nu^0$,
\end{itemize}
where $B_\nu^0>0$ are some constants depending on $\gamma_i, \gamma_e,  A_i, A_e, \kappa,q,s,d,\|u_i^0\|_{\Xi},\|u_e^0\|_{\Xi}, |u_e^0(0)|$ and $|u_i^0(0)|$ but independent of $\eps$, then  the  Cauchy problem  \eqref{origin1} with \eqref{oini}-\eqref{ofar} has a   unique  regular solution $(\rho_\nu, u_\nu,\vp)$ in  $[0,T]\times \mathbb{R}^d$ for any $0<T<\infty$. The uniqueness of $\vp$ is understood up to the addition of a function depending only on time.
Moreover, for every $t\geq0$, $\sigma\in[1,s]$ and $\nu=i,e$, the following estimates hold:
\begin{equation*}
\begin{split}
\big|(\rho_\nu^{\frac{\gamma_\nu-1}{2}},\sqrt{m_\nu}(u_\nu-\hu_\nu))(t)\big|_\infty\leq& C(1+t)^{-\min\{1,\frac{d(\gamma_\nu-1)}{2}\}},\\
\big|\rho_\nu^{\frac{\gamma_\nu-1}{2}}(t)\big|_q\leq& C(1+t)^{\frac{d}{q}-\min\{1,\frac{d(\gamma_\nu-1)}{2}\}},\\
\big\|(\rho_\nu^{\frac{\gamma_\nu-1}{2}},\sqrt{m_\nu}(u_\nu-\hu_\nu))(t)\big\|_{\dot H^\sigma}\leq& C(1+t)^{\frac{d}{2}-\sigma-\min\{1,\frac{d(\gamma_\nu-1)}{2}\}},
\end{split}
\end{equation*}
and
\begin{equation*}
(1+t)^{\min\{d,\frac{2}{\bar\gamma-1}\}-1}\Bigl(|\D\vp(t)|_\infty+(1+t)^{-\frac{d}{q}}|\D\vp(t)|_q+(1+t)^{\sigma-\frac{d}{2}}\|\D\vp(t)\|_{\dot H^\sigma}\Bigr)\leq C.
\end{equation*}
The constant $C$ depends only on the fixed parameters and initial velocity bounds listed above, and is independent of $T$ and $\eps$.
\end{theorem}

The second  theorem concerns the global-in-time convergence in the VMR limit on the electron scale from \eqref{origin1} to \eqref{Euler0} as $\eps \to 0$.

\begin{theorem}\label{thm2.2} Let $(\rho_\nu,u_\nu,\vp)$ be the global-in-time regular solution to the  Cauchy problem  \eqref{origin1} with \eqref{oini}-\eqref{ofar} obtained in \textup{Theorem \ref{globalthm}}. If the initial data satisfy
\begin{align}
		(\rho_\nu^0)^{\frac{\gamma_\nu-1}{2}}\rightharpoonup (\bar{\rho}_\nu^0)^{\frac{\gamma_\nu-1}{2}}\quad &\text{weakly in}\,\,\, L^q\cap \Gamma,\nonumber\\
		u_\nu^0\rightharpoonup \bar{u}_\nu^0\quad &\text{weakly in}\,\,\, \Xi,\nonumber
\end{align}
then there exist functions $\bar{\rho}_\nu$, $\bar{w}_\nu$, $\bar{\vp}$ and $\bar{\hu}_\nu$ satisfying $\bar{w}_i=0$, $$\bar{\rho}_\nu^{\frac{\gamma_\nu-1}{2}}\in L^\infty([0,T];\Gamma\cap L^q),\quad \bar{w}_e\in L^\infty([0,T];\Gamma),\quad  \bar{\hu}_\nu\in L^\infty([0,T];\Xi),$$ such that as $\eps\to 0$, up to subsequences,
\begin{align*}
		(u_e-\hu_e,\D\vp)\rightharpoonup (\bar{w}_e,\D\bar{\vp})\quad &\text{weakly-*}\,\,in \,\,\,L^\infty([0,T];\Gamma),\\
		u_i-\hu_i\longrightarrow 0\quad &\text{strongly}\,\,in \,\,\,L^\infty([0,T];\Gamma),\\
		\rho_\nu^{\frac{\gamma_\nu-1}{2}}\rightharpoonup \bar{\rho}_\nu^{\frac{\gamma_\nu-1}{2}} \quad &\text{weakly-*}\,\, in\,\,\, L^\infty([0,T];\Gamma\cap L^q),\\
		\hu_\nu\rightharpoonup\bar{\hu}_\nu\quad &\text{weakly-*}\,\, in\,\,\, L^\infty([0,T];\Xi),
\end{align*}
where $\bar{\hu}_\nu$ are solutions to the problem \eqref{burgers} with initial data $\bar{u}_\nu^0$ in the sense of distributions. Let $\bar{u}_\nu=\bar{w}_\nu+\bar{\hu}_\nu$. Then $(\bar{\rho}_e,\bar{u}_e,\bar{\vp})$ is the unique global regular solution to the Cauchy problem of \eqref{Euler0} with $b(t,x)=\bar{\rho}_i$  and the following initial data and far field behavior:
\begin{eqnarray}
		\label{oinibar} (\bar{\rho}_e,\bar{u}_e)(0,x)=(\bar{\rho}_e^0(x)\geq 0,\bar{u}_e^0(x)), &\text{for}&\,\, x\in \R^d,\\
		\label{ofarbar} (\bar{\rho}_e,\D\bar{\vp})(t,x)\to 0 \quad \text{as}\,\, |x|\to \infty  &\text{for}& \,\, t\geq 0.
\end{eqnarray}
Here regular solutions are understood in the sense of Definition~\ref{regular-solution}, adapted to \eqref{Euler0} with $b$ fixed and $\nu=e$, $m_e=1$, $q_e=-1$, including the velocity equation on vacuum in $(\mathrm C)$. The uniqueness of $\bar\vp$ is understood up to the addition of a function depending only on time.
%following compressible Euler equations:
%\begin{equation*}
%\begin{cases}
%\pt \bar{\rho}+\dive(\bar{\rho}\bar{u})=0,\\
%\pt(\bar{\rho}\bar{u})+\dive(\bar{\rho}\bar{u}\otimes\bar{u})+\D p(\bar{\rho})=0.
%\end{cases}
%\end{equation*}
\end{theorem}

The last theorem  concerns the global-in-time  error estimates in some  homogeneous Sobolev spaces between  \eqref{origin1} and \eqref{Euler0}.

\begin{theorem}\label{thm2.3} Let the conditions in \textup{Theorems \ref{globalthm}-\ref{thm2.2}} hold. Let $(\rho_\nu,u_\nu,\vp)$ be the solution to the problem \eqref{origin1} with \eqref{oini}-\eqref{ofar} obtained in \textup{Theorem \ref{globalthm}}, and $(\bar{\rho}_e,\bar{u}_e,\bar{\vp})$ be that to \eqref{Euler0} with \eqref{oinibar}-\eqref{ofarbar} obtained in \textup{Theorem \ref{thm2.2}}. Assume further that $u_\nu^0$ are independent of $\eps$ and $((\bar{\rho}_\nu^0)^{\frac{\gamma_\nu-1}{2}},\D u_i^0)\in\dot{H}^{s+1}$.
If the initial data $\rho_\nu^0$ admit the following convergence:
\[
	|(\rho_\nu^0)^{\frac{\gamma_\nu-1}{2}}-(\bar{\rho}_\nu^0)^{\frac{\gamma_\nu-1}{2}}|_\infty+\|(\rho_\nu^0)^{\frac{\gamma_\nu-1}{2}}-(\bar{\rho}_\nu^0)^{\frac{\gamma_\nu-1}{2}}\|_{\dot{H}^ 1\cap\dot{H}^s}\leq \eps^r,
\]
for some positive $r$ independent of $\eps$, then it holds that
\begin{align*}
		\|(\rho_\nu^{\frac{\gamma_\nu-1}{2}}-\bar{\rho}_\nu^{\frac{\gamma_\nu-1}{2}}, u_\nu-\bar{u}_\nu)(t)\|_{\Gamma}\leq&C_0\eps^{\min\{r,1\}}, \quad t\in[0,T],
\end{align*}
where the constant $C_0>0$ may depend on $T$ but is independent of $\eps$.
\end{theorem}

\begin{remark} The restriction \eqref{res1} on $\gamma_\nu,s$ and $q$ in \textup{Theorem \ref{globalthm}}  is mainly used in the analysis  on  the electric potential $\vp$. When  vacuum appears, due to the high non-linearity of the pressure, one encounters some difficulties in the derivation of the $L^2$ estimate for $\D \vp$ through the Poisson equation, which requires  high integrability conditions of  $\rho_\nu$.
\end{remark}

\begin{remark}The conditions (A$_1$)-(A$_2$) in \textup{Theorem \ref{globalthm}} identify a class of admissible initial data that provides unique solvability of  the  Cauchy problem  \eqref{origin1} with \eqref{oini}-\eqref{ofar}. Such initial data contain the following examples:
$$\rho_\nu^0(x)=\lambda_1e^{-|x|^2}, \quad u_\nu^0(x)=Ax+\mathfrak{b}+\lambda_2 f(x)$$
where $A$ is a $d\times d$ constant matrix whose eigenvalues are all greater than $2\kappa$, $\lambda_1>0$ and $\lambda_2>0$ are sufficiently small constants, $f\in\Xi$ and $\mathfrak{b}\in\R^d$ is a constant vector.
\end{remark}

%\begin{remark}
%	Without loss of generality, we can assume that $u_\nu^0(0)=0$, which can be achieved by  the following Galilean transformation:
%	\begin{equation*}
%		\begin{split}
%			t'=&t,\quad x_\nu'=x-u_\nu^0(0)t,\quad
%			\rho_\nu'(t',x_\nu')=\rho_\nu(t,x),\quad
%			\widehat{u}_\nu'(t',x_\nu')= \widehat{u}_\nu(t,x)-u_\nu^0(0).
%		\end{split}
%	\end{equation*}
%\end{remark}

%\begin{remark} It is worth emphasizing that only $\rho_\nu^0$ are required to be sufficiently small. This smallness is independent of other norms of the solutions. At the same time, it is worth emphasizing that no smallness is required on $u_\nu^0$.
%\end{remark}

The rest of the paper is organized as follows. In \S 2, we will   give some basic settings of our problem, and  provide a brief introduction of the main  strategies in this paper. \S 3 is devoted to  the proof of local-in-time  existence of the classical solutions for  a series of carefully designed approximate problems  with  truncated convection operators and  compactly supported initial data.  In \S 4, we first pass to the limit $N\to\infty$ on a common local time interval and then extend the resulting solutions to global-in-time ones by establishing estimates uniform in $\eps$ and the size of the initial supports. In \S 5, we extend this result to general initial data. In addition, the global convergence analysis in the VMR limit from bipolar Euler-Poisson equations \eqref{origin1} to unipolar Euler-Poisson equations \eqref{Euler0} for general initial data can be found at the end of \S 5. \S 6 is devoted to the proof of the global error estimates. Finally, we give an appendix to  show some basic lemmas  used in this paper.

\section{Reformulation and Main Strategies}

In this section, we will give   some basic settings of the problem studied, and then provide  a brief introduction of proof strategies. 

\subsection{Reformulation and formal limit system} For $\gamma_\nu>1$, denote
\begin{equation*}
	n_\nu=\sqrt{\dfrac{4A_\nu\gamma_\nu}{(\gamma_\nu-1)^2}}\rho_\nu^{(\gamma_\nu-1)/2},\quad w_\nu=u_\nu-\hu_\nu=(w_\nu^{(1)},\cdots,w_\nu^{(d)})^\top,
\end{equation*}
where $\hu_\nu=(\hu_\nu^{(1)},\cdots,\hu_\nu^{(d)})^\top$ are the solutions to \eqref{burgers}.  Then from  \eqref{origin1} and \eqref{burgers},  the  Cauchy problem  \eqref{origin1} with \eqref{oini}-\eqref{ofar} is of the form:
\begin{equation}\label{sys-MUK}
\begin{cases} \pt n_e+(w_e+\hu_e)\cdot\nabla n_e+\dfrac{\gamma_e-1}{2}n_e\dive w_e=-\dfrac{\gamma_e-1}{2} n_e\dive \hu_e,\\
		\pt w_e+((w_e+\hu_e)\cdot\nabla)w_e+\dfrac{\gamma_e-1}{2}n_e\nabla n_e=-\nabla\vp-(w_e\cdot\nabla)\hu_e,\\
		\pt n_i+(w_i+\hu_i)\cdot\nabla n_i+\dfrac{\gamma_i-1}{2}n_i\dive w_i=-\dfrac{\gamma_i-1}{2} n_i\dive \hu_i,\\
		\eps^{-2}\pt w_i+\eps^{-2}((w_i+\hu_i)\cdot\nabla)w_i+\dfrac{\gamma_i-1}{2}n_i\nabla n_i=\nabla\vp-\eps^{-2}(w_i\cdot\nabla)\hu_i,\\
		\Delta\vp=C_i n_i^{2/(\gamma_i-1)}-C_e n_e^{2/(\gamma_e-1)},\\
		t=0:\,\,(n_\nu,w_\nu)(0,x):=(n_\nu^0(x),0):=\Big(\sqrt{\dfrac{4A_\nu\gamma_\nu}{(\gamma_\nu-1)^2}}(\rho_\nu^0)^{(\gamma_\nu-1)/2},0\Big), \end{cases}
\end{equation}
where $C_\nu=\Big(\dfrac{(\gamma_\nu-1)^2}{4A_\nu\gamma_\nu}\Big)^{1/(\gamma_\nu-1)}$. If we denote $(\beta_e,\beta_i)=(0,1), \ W_\nu=(n_\nu,\eps^{-\beta_\nu}w_\nu^\top)^\top$, then system \eqref{sys-MUK} can be written as:
\begin{equation}\label{sys-symm}
\begin{cases} \pt W_\nu+ \displaystyle\sum_{j=1}^d A_\nu^j(W_\nu;\hu_\nu,\eps)\pa_{x_j}W_\nu=Q_\nu^1(\D\vp)-Q_\nu^2(W_\nu;\D\hu_\nu),\\
		\Delta\vp=C_i n_i^{2/(\gamma_i-1)}-C_e n_e^{2/(\gamma_e-1)},\\
		(n_\nu,w_\nu)(0,x)=(n_\nu^0(x),0), \end{cases}
\end{equation}
where   $\{\xi_j\}_{j=1}^d$ forms a canonical basis of $\R^d$ and
\begin{equation}\label{defQ}
\begin{split}
		A_\nu^j(W_\nu;\hu_\nu,\eps)=&\left( \begin{matrix} w_\nu^{(j)}+\hu_\nu^{(j)}&\eps^{\beta_\nu}\dfrac{\gamma_\nu-1}{2}n_\nu\xi_j^\top\\
			\eps^{\beta_\nu}\dfrac{\gamma_\nu-1}{2}n_\nu\xi_j&(w_\nu^{(j)}+\hu_\nu^{(j)})\Id \end{matrix} \right),\\
		Q_\nu^1(\D\vp)=&\left(\begin{matrix}0\\
			\eps^{\beta_\nu}q_\nu\D\vp\end{matrix}\right), \quad Q_\nu^2(W_\nu;\D\hu_\nu)=\left(\begin{matrix}\dfrac{\gamma_\nu-1}{2}n_\nu\dive\hu_\nu\\[2mm]
			\eps^{-\beta_\nu}(w_\nu\cdot\D)\hu_\nu\end{matrix}\right).
\end{split}
\end{equation}
It is clear that $W_\nu$ satisfy symmetric hyperbolic systems and $\vp$ satisfies an elliptic equation.

Based on the reformulation \eqref{sys-MUK}, we are ready to show the  VMR limit $\eps\to 0$ formally. Denote the formal limits of $(n_\nu,w_\nu,\vp, \hu_\nu)$ and $(n_\nu^0, u_\nu^0)$ as $(\bar{n}_\nu,\bar{w}_\nu,\bar{\vp},\bar{\hu}_\nu)$ and $(\bar{n}_\nu^0,\bar{u}_\nu^0)$ respectively. The limits of \eqref{sys-MUK}$_4$ and \eqref{burgers} are
\[
\begin{cases} \pt \bar{w}_i+((\bar{w}_i+\bar{\hu}_i)\cdot\D)\bar{w}_i+(\bar{w}_i\cdot\D)\bar{\hu}_i=0,\\
	\pt\bar{\hu}_i+(\bar{\hu}_i\cdot\D)\bar{\hu}_i=0,\\
	t=0:\,\,(\bar{w}_i,\bar{\hu}_i)(0,x)=(0,\bar{u}_i^0). \end{cases}
\] 
This implies that $u_i$ converges to $\bar{\hu}_i$ when $\eps\to 0$. The formal limit of \eqref{sys-MUK}$_3$ yields
\begin{equation}\label{sys-lim-i}
	\pt\bar{n}_i+\bar{\hu}_i\cdot\D\bar{n}_i+\frac{\gamma_i-1}{2}\bar{n}_i\dive\bar{\hu}_i=0,\quad \bar{n}_i(0,x)=\bar{n}_i^0,
\end{equation}
which admits a unique solution $\bar{n}_i(t,x)$. The limits  for \eqref{sys-MUK}$_1$, \eqref{sys-MUK}$_2$ and \eqref{sys-MUK}$_5$ are 
\begin{equation}\label{sys-lim}
\begin{cases} \pt\bar{n}_e+(\bar{w}_e+\bar{\hu}_e)\cdot\D\bar{n}_e+\dfrac{\gamma_e-1}{2}\bar{n}_e\dive\bar{w}_e=-\dfrac{\gamma_e-1}{2}\bar{n}_e\dive\bar{\hu}_e,\\[2mm]
		\pt\bar{w}_e+((\bar{w}_e+\bar{\hu}_e)\cdot\D)\bar{w}_e+\dfrac{\gamma_e-1}{2}\bar{n}_e\D\bar{n}_e=-\D\bar{\vp}-(\bar{w}_e\cdot\D)\bar{\hu}_e,\\
		\Delta\bar{\vp}=C_i (\bar{n}_i(t,x))^{2/(\gamma_i-1)}-C_e \bar{n}_e^{2/(\gamma_e-1)},\\
		(\bar{n}_e,\bar{w}_e)(0,x)=(\bar{n}_e^0(x),0), \end{cases}
\end{equation}
which is a decoupled system for $(\bar{n}_e,\bar{w}_e,\bar{\vp})$ since $\bar{n}_i$ is solved through \eqref{sys-lim-i}. Besides, the limit system for \eqref{burgers} with $\nu=e$ is
\begin{equation}\label{eqeqeq}
	\pt\bar{\hu}_e+(\bar{\hu}_e\cdot\D)\bar{\hu}_e=0, \qquad \bar{\hu}_e(0,x)=\bar{u}_e^0(x).
\end{equation}
Let $\bar{u}_e=\bar{w}_e+\bar{\hu}_e$. It follows from the equations in  \eqref{sys-lim} and \eqref{eqeqeq} that
\begin{equation*}
\begin{cases} \pt \bar{n}_e+\bar{u}_e\cdot\nabla \bar{n}_e+\dfrac{\gamma_e-1}{2}\bar{n}_e\dive \bar{u}_e=0,\\
		\pt\bar{u}_e+(\bar{u}_e\cdot\D)\bar{u}_e+\dfrac{\gamma_e-1}{2}\bar{n}_e\D\bar{n}_e=-\D\bar\vp,\\
		\Delta\bar{\vp}=C_i (\bar{n}_i(t,x))^{2/(\gamma_i-1)}-C_e \bar{n}_e^{2/(\gamma_e-1)}, \end{cases}
\end{equation*}
which recovers system \eqref{Euler0} by defining $ \bar{\rho}_e=\Big(\bar{n}_e\sqrt{\frac{(\gamma_e-1)^2}{4A_e\gamma_e}}\Big)^{\frac{2}{\gamma_e-1}}$.

\subsection{Main strategies} Now we sketch the main strategy for the proof. The study of the global convergence in the VMR limit from the system in \eqref{sys-symm} to \eqref{Euler0} is based on the uniform global existence of regular solutions to \eqref{sys-symm}. However, owing to  the non-integrability of $\nabla \hu_\nu$ and the unboundedness of $\hu_\nu$, the convection terms are not guaranteed to belong to $L^\infty(\R^d)$, and  $\hu_\nu^{(j)} w_\nu\cdot w_\nu$ is not guaranteed to belong to $L^1(\R^d)$ for all $1\leq j\leq d$. Then one can not take advantage of the symmetric hyperbolic structure of \eqref{sys-symm}, since the following integration-by-parts identity cannot be taken for granted:
\[
\int_{\R^d} \eps^{-2\beta_\nu}w_\nu\cdot(\hu_\nu\cdot\D) w_\nu\rmd x= -\frac{1}{2}\int_{\R^d}\dive \hu_\nu |\eps^{-\beta_\nu}w_\nu|^2\rmd x.
\]
This, together with the low-frequency difficulty of the Poisson equation, motivates the introduction of the following truncated approximate problem of \eqref{sys-symm} ({\bf{TEP}}):
\begin{equation}\label{TEP}
\begin{cases} \pt W_\nu+ \displaystyle\sum_{j=1}^d A_\nu^j(W_\nu;\hu_\nu^N,\eps)\pa_{x_j}W_\nu=Q_\nu^1(\D\vp)-Q_\nu^2(W_\nu;\D\hu_\nu),\\
		\Delta\vp=C_i n_i^{2/(\gamma_i-1)}-C_e n_e^{2/(\gamma_e-1)},\\
		(n_\nu,w_\nu)(0,x)=(n_\nu^0(x),0), \end{cases}
\end{equation}
where $\hu_\nu^N=\hu_\nu F(|x|/N)$ for $N\geq 1$, and $F(x)$ is the regular bump function satisfying
\begin{equation}\label{bump}
	F(x)\in C_c^\infty(\R), \qquad 0\leq F(x)\leq 1, \qquad F(x)=\begin{cases} 1,\quad  |x|\leq 1,\\
		0, \quad |x|\geq 2. \end{cases}
\end{equation}
For simplicity, let $F^N:=F(|x|/N)$. Moreover, if  $n_\nu^0$ is compactly supported, it is not hard to show  that  $\D\vp\in H^s$ if $n_\nu\in H^s$ at least on a finite time interval. Besides, the truncation $\hu_\nu^N$ in $A_\nu^j(W_\nu;\hu_\nu^N,\eps)$ ensures that $\hu_\nu^N\in L^\infty(\R^d)$ for finite $N$.

Based on the above observations, the study of the global convergence  in the VMR limit for  \eqref{sys-symm} can be divided into the following two steps:
\begin{itemize}
%\item the local existence of the unique classical solution to ({\bf{TEP}}) with compactly supported initial density for fixed  $\eps$, which can be obtained by some standard iterations for symmetric hyperbolic systems (\S3).
\item  the global uniform estimates for solutions to \eqref{sys-symm} with respect to $\eps$ and the size of the support for $n_\nu^0$ (\S 4). For this purpose,  the major difficulty is caused by the appearance of vacuum and the absence of the damping mechanism.  The key idea here  relies on an elaborate argument on extracting a dispersive effect after some invariant transformation  and  energy methods based on suitable choices of time weights. Mathematically speaking, the dispersive conditions on $u_\nu^0$ enable us to establish an energy ODE inequality,
% \begin{equation}
% \dfrac{\rmd}{\rmd t}Y_s+\dfrac{a}{1+t}Y_s\leq \dfrac{CY_s}{(1+t)^2}+CY_s^2+C(1+t)^{\frac{2}{\bar\gamma-1}}Y_s^{\frac{2}{\bar\gamma-1}},\quad \bar\gamma=\max\{\gamma_i,\gamma_e\},
% \end{equation}
which allows us to establish uniform estimates of solutions to \eqref{sys-symm} in proper energy spaces. Based on these uniform estimates, we establish the global existence of regular solutions with compactly supported initial densities in \S 4.
\item the global existence of the regular solution to \eqref{origin1} with general initial data as well as the global convergence for the VMR limit on the electron scale (\S 5). The proof is based on  a cut-off technique, where the initial data $n_\nu^0$ are multiplied by a sequence of cut-off functions  so that they become compactly supported (\S5.1). Then a series of approximate solutions is obtained and the unique regular solution  to \eqref{sys-symm} for the general data is just the limit of these approximate solutions. Here, due to the Poisson equation, it is hard to obtain uniform estimates which are independent of the support size for $n_\nu^0$ in the classical $H^s$ spaces. We close the energy estimates for $n_\nu$ in $L^\infty([0,T];L^\infty\cap L^q\cap\dot{H}^1\cap\dot{H}^s)$ with $q$ satisfying \eqref{res1} and for $(\eps^{-\beta_\nu}w_\nu,\D\vp)$ in $L^\infty([0,T];L^\infty\cap\dot{H}^1\cap\dot{H}^s)$ (\S5.2). Moreover, based on uniform estimates with respect to $\eps$, one can establish the global convergence of the VMR limit by using classical compactness theories (\S5.3).
\end{itemize}

Finally, for the global-in-time error estimates between \eqref{origin1} and \eqref{Euler0}, we first establish the corresponding error estimates for solutions with compactly supported initial densities. A similar ODE for the error variables is then established, yielding estimates independent of the size of the initial supports. By weak convergence and lower semi-continuity of norms in Sobolev spaces, the desired error estimates for general initial data follow by passing to the limit $\mk\to\infty$ in the initial density approximation (\S 6).

%%%%%%%%%%%%%%%%%%%%%%%%%%%%%%%%%%%%%%%%%%%%%%%%%%%%%%%%%%%%%%%%%%%%%%%%%%%%%%%%%%%%%%%%%
\section{Local existence for truncated system with compactly supported density}\label{s3}

This section is devoted to  establishing the local-in-time  existence of the classical solutions to \eqref{TEP} with compactly supported initial data. Notice that system \eqref{TEP} is symmetric hyperbolic for all $\eps$ and $A_\nu$. For simplicity, in the rest of this section, we always  assume that  $\eps=A_\nu=1$ and $\beta_\nu=0$. The other cases can be dealt with  via exactly the   same arguments used in this section.  

The main result of this section can be stated as follows:

\begin{theorem}\label{comthmloc} Let $\hu_\nu$ be the solution to \eqref{burgers} defined in $[0,T^0]\times \R^d$ and $s,d$ be integers with $(d,s,\gamma_\nu)$ satisfying
\[
	s,\ d\geq 3, \qquad \frac{d}{2}+1<s<\frac{2}{\bar{\gamma}-1}+\frac{1}{2}, \qquad \bar{\gamma}<1+ \min\Big\{\frac{4}{5},\frac{4}{d+1}\Big\},\quad \underline{\gamma}>1,
\]
and there is no upper bound for $s$ if $\frac{2}{\gamma_\nu-1}$ are integers. 	If $(\rho_\nu^0,u_\nu^0)$ satisfies the  assumptions $(\rm A_1)$-$(\rm A_2)$ in \textup{Theorem \ref{globalthm}}, and $\text{supp}_x \rho_\nu^0(x)\subset B_{R_\nu}$ for some finite $R_\nu>0$ with $B_{R_\nu}$ being the ball centered at the origin with radius $R_\nu$, then there exist a time $T_{\mathrm{loc}}\in (0,T^0]$ independent of $N$ and a unique classical solution $(n_\nu,w_\nu,\vp)$ in $[0,T_{\mathrm{loc}}]\times\R^d$ for every given $N\geq 1$ to \eqref{TEP} satisfying
\begin{equation*}
		(n_\nu,w_\nu)\in C([0,T_{\mathrm{loc}}];H^s)\cap C^1([0,T_{\mathrm{loc}}];H^{s-1}), \quad \D\vp\in C([0,T_{\mathrm{loc}}];H^s).
\end{equation*}
\end{theorem}
%\begin{remark}
%	In Theorem \ref{3.1}, $(n,w,\varphi)$  in $[0,T^*]\times \mathbb{R}^3$ is called a strong solution   to the Cauchy problem \eqref{TEP}, if it satisfies \eqref{TEP} in the sense of distributions,  and satisfies the equations in \eqref{TEP} for a.e. $(t,x)\in (0,T^*]\times \mathbb{R}^3$.
%\end{remark}

\begin{remark}\label{local-restart}
Fix a finite horizon $T^0>0$ and $t_0\in[0,T^0)$. For each fixed $\eps\in(0,1]$, the proof of Theorem~\ref{comthmloc} applies to \eqref{TEP} starting at $t_0$, with the prescribed background $\hu_\nu|_{[t_0,T^0]}$ and initial data
\[
(n_\nu,w_\nu)(t_0)=(n_{\nu,0},w_{\nu,0})\in H^s,\qquad n_{\nu,0}\geq0,\qquad \operatorname{supp}n_{\nu,0}\subset B_{\mathcal R_0}.
\]
Replace \eqref{LTEPini} by these data, start the iteration from $(n_{\nu,0},w_{\nu,0})$, and define the data bound
\[
M_0:=1+\sum_{\nu=i,e}\left(\|(n_{\nu,0},\eps^{-\beta_\nu}w_{\nu,0})\|_s+\sup_{0\leq t\leq T^0}\|\hu_\nu(t)\|_\Xi\right).
\]
No smallness of these data is needed. The resulting solution exists on $[t_0,t_0+\tau]$, where $\tau=\min\{\tau_*,T^0-t_0\}$ and $0<\tau_*\leq1$ can be chosen independently of $N$, depending only on $T^0$, the fixed parameters, and upper bounds for $M_0$, $\mathcal R_0$ and
\[
\max_{\nu=i,e}\sup_{\substack{t_0\leq t\leq T^0\\ |x|\leq\mathcal R_0+1}}|\hu_\nu(t,x)|.
\]
Passing to the limit $N\to\infty$ as in \S\ref{s4.1} gives a solution of \eqref{sys-symm} on the same local interval.
\end{remark}

The rest of this section is devoted to the proof of Theorem \ref{comthmloc}.

\subsection{Linearization and uniform estimates}\label{3.1} In order to show the local existence for \eqref{TEP}, we  first need to consider its linear approximate problem. Fix a finite time $T^0>0$, and let $(\tilde{\pi}_\nu,\tilde{v}_\nu,\D\tilde{\Phi})$ be known functions satisfying
\begin{equation}\label{cond1}
\begin{split}
		& (\tilde{\pi}_\nu,\tilde{v}_\nu)(0,x)=(n_\nu^0(x),0),\quad  \D\tilde{\Phi}\in C([0,T^0];H^s),\\
		&(\tilde{\pi}_\nu,\tilde{v}_\nu)\in C([0,T^0];H^s)\cap C^1([0,T^0];H^{s-1}).
\end{split}
\end{equation}

Let $C_S$ be the embedding constant in the embedding $H^s\hookrightarrow L^\infty$ and $\mathbb{V}>0$ be a constant to be specified below (see \eqref{def-VV}), independently of $N$ and the iteration index. Set $\mathcal R=\max_{\nu=i,e}R_\nu$ and $\mathfrak c(t):=\mathcal R+\mathbb{V}t$. We require
\begin{equation}\label{supp-pi}
	\operatorname{supp}_x\widetilde\pi_\nu(t,\cdot) \subseteq\overline B_{\mathfrak c(t)}, \qquad 0\le t\le T^0.
\end{equation}
The constant $\mathbb{V}$ is independent of $N$ and the iteration index. Denote $W_\nu^N=(n_\nu^N,(w_\nu^N)^\top)^\top$ and $\tilde{W}_\nu=(\tilde{\pi}_\nu,\tilde{v}_\nu^\top)^\top$, and  we introduce  the  linear problem:
\begin{equation}\label{LTEP}
\begin{cases} \pt W_\nu^N+ \displaystyle \sum_{j=1}^d A_\nu^j(\tilde{W}_\nu;\hu_\nu^N,1)\pa_{x_j}W_\nu^N=Q_\nu^1(\D\tilde{\Phi})-Q_\nu^2(W_\nu^N;\D\hu_\nu),\\
		\Delta\vp^N=C_i\gp_{\frac{2}{\gamma_i-1}}(\tilde{\pi}_i)-C_e \gp_{\frac{2}{\gamma_e-1}}(\tilde{\pi}_e),\\
\end{cases}
\end{equation}
with the initial data
\begin{equation}\label{LTEPini}
	(n_\nu^N,w_\nu^N)(0,x)=(n_\nu^0(x),0).
\end{equation}
The functions $A_\nu^j(\cdot;\cdot,\cdot), Q_\nu^1(\,\cdot\,)$ and $Q_\nu^2(\,\cdot\,;\D\hu_\nu)$ are defined in \eqref{defQ}, and 
\[
\gp_p(z):= \begin{cases} z^{p}, & p\in\mathbb N,\\[2mm]
	z|z|^{p-1}, & p\notin\mathbb N, \end{cases} \qquad p\in\mathbb R, \quad p>1.
\]
\begin{remark}\label{signed-power-estimates} For noninteger $p=2/(\gamma_\nu-1)$, our assumptions give $p>2$ and $s-1<p-\frac12$, so the required estimates for $\gp_p$ and $\gp_p'$ follow from Lemma~\ref{comm0} with exponent $p-1$ and the product estimate in Lemma~\ref{comm1}.
\end{remark}

\begin{remark} We do not require $\widetilde{\pi}_\nu$ to be nonnegative, since the iteration scheme introduced below does not guarantee that the iterates $n_\nu^m$ remain nonnegative. To avoid taking noninteger powers of potentially negative quantities, the function $\gp$ is introduced. We will establish the nonnegativity of the limit $n_\nu$ when  the original nonlinear system is recovered.
\end{remark}

The global  well-posedness  of classical solutions to  the Cauchy problem \eqref{LTEP}--\eqref{LTEPini} at least  for $1\leq N<\infty$ can be  stated as follows.

\begin{lemma}\label{exisloc}Let the conditions in \textup{Theorem \ref{comthmloc}} hold. Then there exists a unique classical solution $(n_\nu^{N},w_\nu^{N},\vp^{N})$ in $[0,T^0]\times \R^d$ to \eqref{LTEP}--\eqref{LTEPini} satisfying
\[
	(n_\nu^{N},w_\nu^{N})\in C([0,T^0]; H^s)\cap C^1([0,T^0];H^{s-1}), \qquad \D\vp^N\in C([0,T^0]; H^s).
\]
\end{lemma}
\begin{proof} Notice that $Q_\nu^1(\D\tilde{\Phi}) $ and $Q_\nu^2(W_\nu^N;\D\hu_\nu)$ belong to $H^s$ if one assumes $W_\nu^N\in H^s$. Consequently, by classical arguments (cf. \cites{Kato1975,Lax1973}),  the existence of the classical solution $(n_\nu^{N},w_\nu^{N})$ is obtained, which satisfies the following properties:
\[
	(n_\nu^{N},w_\nu^{N})\in C([0,T^0];H^s)\cap C^1([0,T^0];H^{s-1}).
\]

By \eqref{supp-pi}, the prescribed densities are supported in the fixed ball $\overline B_{C_*}$ on $[0,T^0]$, where $C_*:=\mathfrak c(T^0)$. Under the assumptions of Theorem~\ref{comthmloc}, Lemma~\ref{Poisson} and the composition estimates discussed in Remark~\ref{signed-power-estimates} give
\begin{equation*}
		\|\D\vp^{N}\|_s\leq \sum_{\nu=i,e}C_*^0\big\|\gp_{\frac{2}{\gamma_\nu-1}}(\tilde{\pi}_\nu)\big\|_{s-1}\leq \sum_{\nu=i,e}C_*^0|\tilde{\pi}_\nu|_\infty^{\frac{3-\gamma_\nu}{\gamma_\nu-1}}\|\tilde{\pi}_\nu\|_{s-1},
\end{equation*}
in which $C_*^0>0$ is a generic constant that depends solely on $(C_*,s,d,\gamma_i,\gamma_e)$.  It follows from  the requirements on $\tilde{\pi}_\nu$  that  $\D\vp^{N}\in L^\infty([0,T^0];H^{s})$. Recall that the assumptions of Theorem~\ref{comthmloc} include
\begin{equation}\label{scon}
		s>\frac{d}{2}+1\quad \text{and} \quad s,d\geq 3.
\end{equation} 
For integer powers, the composition estimates follow from polynomial product estimates and impose no upper bound on $s$. In addition, let $t^1$ and $t^2$ be two times in $[0,T^0]$. Then
\[
	\Delta(\vp^N(t^1,x)-\vp^N(t^2,x))=\sum_{\nu=i,e}C_\nu q_\nu(\gp_{\frac{2}{\gamma_\nu-1}}(\tilde{\pi}_\nu(t^1,x))-\gp_{\frac{2}{\gamma_\nu-1}}(\tilde{\pi}_\nu(t^2,x))).
\]
By Lemma \ref{Poisson}, it holds that 
\begin{align}\label{midd}
		&\|\D(\vp^N(t^1,x)-\vp^N(t^2,x))\|_{s}\notag\\
		\leq& C_*^0\sum_{\nu=i,e}\|\gp_{\frac{2}{\gamma_\nu-1}}(\tilde{\pi}_\nu(t^1,\cdot))-\gp_{\frac{2}{\gamma_\nu-1}}(\tilde{\pi}_\nu(t^2,\cdot))\|_{s-1}\\
		\leq& C_*^0\sum_{\nu=i,e}|t^1-t^2|\sup_{\tau\in[0,T^0]}\|\gp'_{\frac{2}{\gamma_\nu-1}}(\tilde\pi_\nu(\tau,\cdot))\pt\tilde{\pi}_\nu(\tau,\cdot)\|_{s-1}. \notag
\end{align}
By Lemma \ref{comm0}, if $\gamma_\nu\leq 2$ and $0\leq s-1< {\displaystyle\min_{\nu=i,e}}\{\frac{2}{\gamma_\nu-1}\}-\frac{1}{2}$,
\begin{equation}\label{middd}
		\big\|\gp'_{\frac{2}{\gamma_\nu-1}}(\tilde\pi_\nu(\tau,\cdot))\big\|_{s-1}\leq C_*^0 |\tilde{\pi}_\nu|_{\infty}^{\frac{2}{\gamma_\nu-1}-2}\|\tilde{\pi}_\nu\|_{s-1}\leq C_*^0\|\tilde{\pi}_\nu\|_{s-1}^{\frac{2}{\gamma_\nu-1}-1}.
\end{equation}
Similarly, because of \eqref{scon}, one needs the following restriction on $\bar\gamma$:
\[
	\bar{\gamma}<1+\min\Big\{\frac{4}{5},\frac{4}{d+1}\Big\}.
\]
Substituting \eqref{middd} into \eqref{midd} and noticing that $\pt\tilde{\pi}_\nu\in C([0,T^0];H^{s-1})$, one obtains
\begin{equation}\label{mid}
		\|\D(\vp^N(t^1,x)-\vp^N(t^2,x))\|_{s}\to 0, \quad \text{as}\,\,\,t^1\to t^2,
\end{equation}
which implies $\D\vp^N\in C([0,T^0];H^s)$. This ends the proof.
\end{proof}

Next, we give some \textit{a priori} estimates for solutions $(n_\nu^{N},w_\nu^{N},\D\vp^N)$ in $H^s$, which are independent of $N$ but dependent on $R_\nu$.  For this purpose, we fix a positive constant $c_0$ large enough, which is independent of $N$, such that
\begin{equation}\label{linini}
\begin{aligned}
		&1+\|n_\nu^0\|_s+\|u_\nu^0\|_{\Xi}\leq c_0, \ \  \sup_{0\leq t\leq T^*}\left(\|\tilde{\pi}_\nu(t)\|_s+\|\tilde{v}_\nu(t)\|_s\right)\leq c_1, \\
		&\sup_{0\leq t\leq T^*}\|\D\tilde{\Phi}(t)\|_s\leq c_2,\qquad \operatorname{supp}_x\widetilde\pi_\nu(t,\cdot) \subseteq\overline B_{\mathfrak c(t)}, \quad \text{for}\ \ t\in [0,T^*],
\end{aligned}
\end{equation}
for some constants $c_1$ and $c_2$ independent of $N$ satisfying
\[
c_2>c_1>c_0>1,
\]
and a time $T^*\in(0,T^0)$, which  will be determined later (see \eqref{T*1}) and depends on $R_\nu, c_0$, $T^0$, $|u_\nu^0(0)|$ and fixed constants $(s,d,q,\gamma_i,\gamma_e,\kappa)$ but is independent of $N$. For simplicity, the superscript $N$ is not expressed explicitly in the rest of this section.

In the rest of this section, $C\geq 1$ will denote a generic positive constant depending only on fixed constants $(A_\nu,\gamma_\nu,\kappa,T^0,R_\nu,s,d,q,|u_\nu^0(0)|)$, but independent of $N$, which may differ from line to line. It follows from \eqref{fujia1}--\eqref{wuqiong}, ({\bf IH(j)}$_1$) and \eqref{fujiajia0} that
\begin{equation*}\begin{split}
		\|\hu_\nu\|_{\tilde T_0,\Xi}\leq C c^4_0, \quad \text{for}\,\,\tilde T_0=\min\{T^0,c_1^{-4}\},
\end{split}
\end{equation*}
and based on this, one obtains
\[
|\hu_\nu(t,x)|\leq C(1+c_0^4)e^{Cc_0^4t}(1+|x|)\leq Cc_0^4(1+|x|), \quad \text{for}\quad 0<t\leq \tilde T_0=\min\{T^0,c_1^{-4}\}.
\]
Set $T_0=\min\{\tilde T_0,\mathbb{V}^{-1}\}$, where $\mathbb{V}$ is defined as
\begin{equation}\label{def-VV}
	\mathbb{V}:=1+C_Sc_1+\max_{\nu=i,e}\sup_{\substack{0\le t\le T^0\\
			|x|\le \mathcal R+1}}|\widehat u_\nu(t,x)|.
\end{equation}
On $|x|=\mathfrak c(t)$, since $\mathfrak c(t)\le \mathcal R+1$,
\[
\frac{x}{|x|}\cdot (\widetilde v_\nu+\widehat u_\nu^N)\le |\widetilde v_\nu+\widehat u_\nu^N| \le \mathbb{V}-1<\mathfrak c'(t).
\]
Thus a forward characteristic cannot leave the moving ball. Fix $|x|>\mathfrak c(t)$, and let $Y(\sigma)$ be the backward characteristic of $\widetilde v_\nu+\widehat u_\nu^N$ through $(t,x)$. Then $|Y(\sigma)|>\mathfrak c(\sigma)$ for $0\le\sigma\le t$. Consequently $\widetilde\pi_\nu(\sigma,Y(\sigma))=0$, and
\[
\frac{\rmd}{\rmd\sigma}n_\nu^N(\sigma,Y(\sigma)) =-\frac{\gamma_\nu-1}{2}\operatorname{div}\widehat u_\nu(\sigma,Y(\sigma)) n_\nu^N(\sigma,Y(\sigma)).
\]
Its initial value is zero because $|Y(0)|>\max_{\nu=i,e}R_\nu$. Therefore
\begin{equation}\label{supp-n}
	\operatorname{supp}_x n_\nu^N(t,\cdot) \subseteq\overline B_{\mathfrak c(t)}, \qquad 0\le t\le T^*.
\end{equation}

Let $(n_\nu,w_\nu,\vp)$ be the unique classical solution to \eqref{LTEP} in $[0,T^0]\times \R^d$ obtained in Lemma \ref{exisloc}. One has the following uniform estimates on $(n_\nu,w_\nu)$.
\begin{lemma}\label{locuni} It holds
\begin{equation*} 
		\|(n_\nu,w_\nu)(t)\|_{s}^2\leq Cc_0^2, \qquad {\rm{for}}\,\,\, 0\leq t\leq T_1=\min\{T_0,c_2^{-4}\}.
\end{equation*}

\end{lemma}
\begin{proof} First, applying $\D^l$ to \eqref{LTEP}$_1$ ($0\leq l\leq s$), multiplying the resulting equations by $\D^l W_\nu$ and integrating over $\mathbb{R}^d$ yield that
\begin{align*}
		\dfrac{1}{2}\dfrac{\rmd}{\rmd t}|\D^l W_\nu|_2^2=&-\int_{\R^d}(\D^lW_\nu)^\top \sum_{j=1}^d\Big(\D^l\Big( A_\nu^j(\tilde{W}_\nu;\hu_\nu^N)\pa_{x_j}W_\nu\Big)- A_\nu^j(\tilde{W}_\nu;\hu_\nu^N)\D^l\pa_{x_j}W_\nu\Big)\rmd x\nonumber\\
		&+\frac{1}{2}\sum_{j=1}^d\int_{\R^d}(\D^l W_\nu)^\top\pa_{x_j}A_\nu^j(\tilde{W}_\nu;\hu_\nu^N,1)\D^l W_\nu\rmd x\\
		&+q_\nu\int_{\R^d}(\D^{l}\D\tilde{\Phi})^\top\D^l w_\nu\rmd x-\int_{\R^d}(\D^lW_\nu)^\top\D^l Q_\nu^2(W_\nu;\D\hu_\nu)\rmd x:= \sum_{j=1}^4\mathcal{I}_j^l,\nonumber
\end{align*}
with the natural correspondence of $\{\mathcal{I}_j^l\}_{j=1}^4$. By \eqref{galilean}, one obtains
\begin{align*}
		|\D\hu^N|_\infty\leq& C|\D \hu|_\infty|F^N|_\infty+C|\hu\D F^N|_\infty\\
		\leq&Cc_0^4+\frac{Cc_0^4}{N}(1+|x|)|F'(|x|/N)|\leq Cc_0^4,\qquad N\leq |x|\leq 2N.
\end{align*}
It follows that $\mathcal{I}_2^l$ and $\mathcal{I}_3^l$ can be estimated as follows
\begin{equation}\label{I12loc}
\begin{split}
			|\mathcal{I}_2^l|\leq& C(|\D \tilde{\pi}_\nu|_\infty+|\D \tilde{v}_\nu|_\infty+|\D \hu^N|_\infty)|\D^l W_\nu|_2^2\leq C(c_0^4+c_1)\|W_\nu\|_s^2, \\
			|\mathcal{I}_3^l|\leq& C|\D^l W_\nu|_2^2+ C\|\D \tilde{\Phi}\|_s^2\leq C\|W_\nu\|_s^2+Cc_2^2.
\end{split}
\end{equation}

For $\mathcal{I}_1^l$, notice that $\mathcal{I}_1^0=0$, and for the case $l=1$:
\[
	|\mathcal{I}_1^1|\leq C(|\D \tilde{\pi}_\nu|_\infty+|\D \tilde{v}_\nu|_\infty+|\D \hu_\nu^N|_\infty)|\D W_\nu|_2^2\leq C(c_0^4+c_1)\|W_\nu\|_s^2.
\]
For $2\leq l\leq s$, $\mathcal{I}_1^l$ is a sum of terms of the form
\[
	L^{(\sigma)}:=\int_{\R^d}(\D^lW_\nu)^\top\D^\sigma A_\nu^j(\tilde{W}_\nu;\hu^N_\nu,1)\D^{l-\sigma }\pa_{x_j}W_\nu\rmd x, \qquad 1\leq \sigma\leq l, \qquad 1\leq j\leq d.
\]
Then it is clear that for $\sigma=1$:
\begin{equation}\label{midd3}
		\Big|L^{(1)}\Big|\leq C|\D^l W_\nu|_2^2|\D(\tilde{\pi}_\nu,\tilde{v}_\nu,\hu_\nu^N)|_\infty\leq C(c_0^4+c_1)\|W_\nu\|_s^2.
\end{equation}
For $2\leq \sigma\leq l$, one needs to consider different cases. Let
\[
	B_\nu^j(\tilde{W}_\nu)=\left( \begin{matrix} \tilde{v}_\nu^{(j)}&\dfrac{\gamma_\nu-1}{2}\tilde{\pi}_\nu\xi_j^\top\\
		\dfrac{\gamma_\nu-1}{2}\tilde{\pi}_\nu\xi_j&\tilde{v}_\nu^{(j)}\Id \end{matrix} \right),
\]
where $\tilde{v}_\nu=(\tilde{v}_\nu^{(1)},\cdots, \tilde{v}_\nu^{(d)})^\top$. Then for $2\leq \sigma\leq l$, one has
\begin{align}\label{midd4}
		|L^{(\sigma)}|\leq&C\Big|\int_{\R^d}(\D^l W_\nu)^\top\D^\sigma B_\nu^j(\tilde{W}_\nu)\D^{l-\sigma }\pa_{x_j}W_\nu\rmd x\Big|\nonumber\\
		&+C\Big|\int_{\R^d}\D^l n_\nu\D^\sigma(\hu_\nu^N)^{(j)}\D^{l-\sigma }\pa_{x_j}n_\nu\rmd x\Big|\\
		&+C\Big|\int_{\R^d}(\D^l w_\nu)^\top\D^\sigma(\hu_\nu^N)^{(j)}\D^{l-\sigma }\pa_{x_j}w_\nu\rmd x\Big|:=L_{\tilde{W}_\nu}^{(\sigma)}+L_{n_\nu}^{(\sigma)}+L_{w_\nu}^{(\sigma)},\nonumber
\end{align}
with the natural correspondence of $L_{\tilde{W}_\nu}^{(\sigma)},L_{n_\nu}^{(\sigma)}$ and $L_{w_\nu}^{(\sigma)}$. First, by  Lemma \ref{ga-ni}, one has
\begin{align}\label{midd5}
		|L_{\tilde{W}_\nu}^{(\sigma)}|\leq& C|\D^l W_\nu|_2|\D^\sigma\tilde{W}_\nu|_{\frac{2(l-1)}{\sigma-1}}|\D^{l+1-\sigma }W_\nu|_{\frac{2(l-1)}{l-\sigma}}\nonumber\\
		\leq& C|\D^l W_\nu|_2|\D\tilde{W}_\nu|_\infty^{\frac{l-\sigma}{l-1}}|\D^l\tilde{W}_\nu|_2^{\frac{\sigma-1}{l-1}}|\D W_\nu|_\infty^{\frac{\sigma-1}{l-1}}|\D^lW_\nu|_2^{\frac{l-\sigma}{l-1}}\\
		\leq & C\|\tilde{W}_\nu\|_s\|W_\nu\|_s^2\leq Cc_1\|W_\nu\|_s^2.\nonumber
\end{align}
Next, by Lemma \ref{hu}, one obtains
\begin{equation*}
\begin{split}
			|\D^2 \hu_\nu^N|_\infty\leq & C|\D^2\hu_\nu|_\infty+C|\D\hu_\nu|_\infty|\D F^N|_\infty+C|\hu_\nu\D^2 F^N|_\infty\\
			\leq & C|\D^2\hu_\nu|_\infty+\frac{C}{N}|\D\hu_\nu|_\infty+Cc_0^4(1+|x|)\frac{1}{N^2}\leq Cc_0^4.
\end{split}
\end{equation*}
It follows that for $\sigma=2$:
\begin{equation}\label{midd6}
		\Big|L_{n_\nu}^{(2)}+L_{w_\nu}^{(2)}\Big|\leq C|\D^2 \hu_\nu^N|_\infty\| W_\nu\|_s^2\leq Cc_0^4\| W_\nu\|_s^2.
\end{equation}
For cases $3\leq \sigma\leq l$, one has
\begin{align}
		\Big|L_{n_\nu}^{(\sigma)}\Big|\leq& \Big|\int_{\R^d}\D^l n_\nu\sum_{m'=3}^\sigma m_{\sigma m'}\D^{m'} \hu_\nu \D^{\sigma-m'} F^N\D^{l-\sigma }\pa_{x_j}n_\nu\rmd x\Big|\nonumber\\
		&+\Big|\int_{\R^d}\D^l n_\nu m_{\sigma 2}\D^{2} \hu_\nu \D^{\sigma-2} F^N\D^{l-\sigma }\pa_{x_j}n_\nu\rmd x\Big|\nonumber\\
		&+\Big|\int_{\R^d}\D^l n_\nu(\sigma\D\hu_\nu\D^{\sigma-1}F^N\D^{l-\sigma }\pa_{x_j}n_\nu+\hu_\nu\D^\sigma F^N\D^{l-\sigma }\pa_{x_j}n_\nu)\rmd x\Big|\nonumber\\
		\leq& C\sum_{m'=3}^\sigma|\D^l n_\nu|_2|\D^{l+1-\sigma }n_\nu|_{\frac{2(l-1)}{l-\sigma+1}}|\D^{m'} \hu_\nu|_{\frac{2(l-1)}{\sigma-2}}+C|\D^ln_\nu|_2|\D^2\hu_\nu|_\infty|\D^{l-\sigma}\pa_{x_j}n_\nu|_2\nonumber\\
		&\label{midd7}+C|\D^l n_\nu|_2\Big(|\D\hu_\nu|_\infty+c_0^4\Big)|\D^{l+1-\sigma} n_\nu|_2\\
		\leq&C|\D^l n_\nu|_2\Big(\sum_{m'=3}^\sigma|n_\nu|_\infty^{\frac{\sigma-2}{l-1}}|\D^{l-1}n_\nu|_2^{\frac{l-\sigma+1}{l-1}}|\D^{m'} \hu_\nu|_{2}^{\frac{\sigma-m'}{l-m'+1}}|\D^{m'} \hu_\nu|_{\frac{2(l-1)}{m'-2}}^{\frac{l-\sigma+1}{l-m'+1}}+ c_0^4\|n_\nu\|_s\Big)\nonumber\\
		\leq& C\|n_\nu\|_s^2\Big(\sum_{m'=3}^\sigma|\D^{m'} \hu_\nu|_{2}^{\frac{\sigma-m'}{l-m'+1}}|\D^{l+1} \hu_\nu|_2^{\frac{(m'-2)(l-\sigma+1)}{(l-1)(l-m'+1)}}|\D^2\hu_\nu|_\infty^{\frac{l-\sigma+1}{l-1}}+c_0^4\Big)\nonumber\\
		\leq & Cc_0^4\|n_\nu\|_s^2,\nonumber
\end{align}
with $m_{\sigma m'}$ the binomial coefficients. Similar estimates can be obtained for $L_{w_\nu}^{(\sigma)}$. Then  \eqref{midd7} implies  that for $3\leq \sigma\leq l$,
\begin{equation}\label{midd9}
		\Big|L_{n_\nu}^{(\sigma)}+L_{w_\nu}^{(\sigma)}\Big|\leq Cc_0^4\| W_\nu\|_s^2.
\end{equation}
Substituting \eqref{midd5}, \eqref{midd6} and \eqref{midd9} into \eqref{midd4}, one has
\[
	|L^{(\sigma)}|\leq C(c_0^4+c_1)\| W_\nu\|_s^2, \quad \text{for}\,\,\, 2\leq \sigma\leq l,
\]
which, along with \eqref{midd3}, yields
\begin{equation}\label{I3loc}
		|\mathcal{I}_1^l|\leq C(c_0^4+c_1)\|W_\nu\|_s^2.
\end{equation}

For $\mathcal{I}_4^l$, one has for $l=0$ or $l=1$:
\begin{equation}\label{midd1}
\begin{split}
			|\mathcal{I}_4^0|\leq& |W_\nu|_2^2|\D\hu_\nu|_\infty\leq Cc_0^4|W_\nu|_2^2,\\
			|\mathcal{I}_4^1|\leq& |\D W_\nu|_2(|\D\hu_\nu|_\infty|\D W_\nu|_2+|\D^2\hu_\nu|_\infty|W_\nu|_2)\leq Cc_0^4\|W_\nu\|_{s}^2,
\end{split}
\end{equation}
and for cases $2\leq l\leq s$, it holds
\begin{align}\label{midd2}
		|\mathcal{I}_4^l|\leq& C\sum_{\sigma=l-1}^l|\D^l W_\nu|_2 |\D^{\sigma+1}\hu_\nu|_2|\D^{l-\sigma}W_\nu|_\infty+C\sum_{\sigma=2}^{l-2}|\D^l W_\nu|_2|\D^{\sigma+1}\hu_\nu\D^{l-\sigma}W_\nu|_2\nonumber\\
		&+C\sum_{\sigma=0}^1|\D^{\sigma+1}\hu_\nu|_\infty|\D^{l-\sigma} W_\nu|_2|\D^{l}W_\nu|_2\nonumber\\
		\leq& C\sum_{\sigma=2}^{l-2}|\D^l W_\nu|_2|\D^{\sigma+1}\hu_\nu|_{\frac{2(l-1)}{\sigma-1}}|\D^{l-\sigma}W_\nu|_{\frac{2(l-1)}{l-\sigma}}+Cc_0^4\|W_\nu\|_s^2\\
		\leq& C\sum_{\sigma=2}^{l-2}|\D^l W_\nu|_2 |\D^2\hu_\nu|_\infty^{\frac{l-\sigma}{l-1}}|\D^{l+1}\hu_\nu|_2^{\frac{\sigma-1}{l-1}}|W_\nu|_\infty^{\frac{\sigma-1}{l-1}}|\D^{l-1}W_\nu|_2^{\frac{l-\sigma}{l-1}}+Cc_0^4\|W_\nu\|_s^2\nonumber\\
		\leq &Cc_0^4\|W_\nu\|_s^2.\nonumber
\end{align}
Estimates \eqref{midd1}--\eqref{midd2} yield
\begin{equation}\label{I4loc}
		|\mathcal{I}_4^l|\leq Cc_0^4\|W_\nu\|_s^2.
\end{equation}

Due to estimates \eqref{I12loc}, \eqref{I3loc} and \eqref{I4loc}, one concludes
\[
	\dfrac{\rm d}{\rmd t}\|W_\nu\|_s^2\leq C(c_0^4+c_1)\|W_\nu\|_s^2+Cc_2^2,
\]
which, along with Gr\"{o}nwall's inequality, implies
\begin{equation}\label{fineslin}
		\|W_\nu(t)\|_s^2\leq e^{Cc_2^4t}(\|W_\nu(0)\|_s^2+Cc_2^2t)\leq Cc_0^2,
\end{equation}
for $0\leq t\leq T_1=\min\{T_0, c_2^{-4}\}$. This ends the proof.
\end{proof}

We then give uniform estimates on $\D\vp$.

\begin{lemma}\label{nablavp} The solution $\vp$ to \eqref{LTEP}$_2$ has the following estimates
\begin{equation}\label{locunieq3}
		\|\D\vp(t)\|_{s}\leq Cc_1^{\frac{2}{\underline{\gamma}-1}}, \qquad {\rm{for}}\,\,\, 0\leq t\leq T_1.
\end{equation}
\end{lemma}
\begin{proof}	By \eqref{linini} and $\mathbb{V}T_1\le1$, for $0\le t\le T_1$ one has
\[
	\operatorname{supp}_x\widetilde\pi_\nu(t,\cdot) \subseteq\overline B_{\mathfrak c(t)} \subseteq\overline B_{\mathcal{R}+1},
\]
where $\mathcal{R}=\max_{\nu=i,e}R_\nu$. Hence Lemma \ref{Poisson} and the composition estimates give
\[
\begin{aligned}
		\|\nabla\vp(t)\|_s &\le C\bigl(1+(\mathcal{R}+1)^{\frac{d+2}{2}}\bigr) \sum_{\nu=i,e} \left\|\gp_{\frac{2}{\gamma_\nu-1}} (\widetilde\pi_\nu(t))\right\|_{s-1}\le C_\mathcal{R}c_1^{\frac{2}{\underline{\gamma}-1}}.
\end{aligned}
\]
Here $C_\mathcal{R}$ depends on the fixed outer radius $\mathcal{R}+1$ and the fixed parameters, but is independent of $\mathbb{V},c_1,c_2,N$ and the iteration index. This proves \eqref{locunieq3}.
\end{proof}

It follows from the estimates  \eqref{fineslin}-\eqref{locunieq3}  that
\[
\|W_\nu(t)\|_s^2\leq Cc_0^2, \qquad \|\D\vp(t)\|_{s}\leq Cc_1^{\frac{2}{\underline{\gamma}-1}},
\]
for $0\leq t\leq T_1=\min\{T_0, c_2^{-4}\}$. Then one can define $c_1$, $c_2$ and  $T^*$ by
\begin{equation}\label{T*1}
	c_1=C^{\frac{1}{2}}c_0, \qquad c_2=Cc_1^{\frac{2}{\underline{\gamma}-1}}=C^{\frac{\underline{\gamma}}{\underline{\gamma}-1}}c_0^{\frac{2}{\underline{\gamma}-1}}, \qquad T^*=\min\{T_0,c_2^{-4}\}.
\end{equation}
One then obtains that for $0\leq t\leq T^*$,
\begin{equation}\label{locunieq2}
	\|W_\nu(t)\|_s^2\leq c_1^2, \qquad \|\D\vp(t)\|_s^2\leq c_2^2.
\end{equation}

In summary, for constants $c_1, c_2$ and the time $T^*$ defined in \eqref{T*1} as well as the given functions $(\tilde{\pi}_\nu,\tilde{v}_\nu,\tilde{\Phi})$ satisfying \eqref{cond1} and \eqref{linini}, there exists a unique solution $(n_\nu^N,w_\nu^N,\vp^N)$ satisfying \eqref{supp-n} and \eqref{locunieq2}. Here it should be noted that the definitions of $(c_1,c_2,T^*)$ are all independent of $N$, but may be dependent on $R_\nu$.

\subsection{Strong convergence from  linearized problems to the nonlinear one} The proof is based on the classical iteration scheme and the existence results for the linearized problem obtained in \S\ref{3.1}. In this subsection, $C$ may also depend on $c_0$, and thus $c_1$ and $c_2$, but is independent of $N$ and the iteration index. 

In the following iteration, Lemma~\ref{exisloc} is applied on $[0,T^*]$, with its assumptions on the prescribed functions imposed on this interval.

We first introduce our iteration scheme. Let $(\tilde{\pi}_\nu,\tilde{v}_\nu,\tilde{\Phi})=(n_\nu^0,w_\nu^0=0,0)$ and $(n_\nu^1,w_\nu^1)$ be the classical solution to \eqref{LTEP}--\eqref{LTEPini}. Since $(n_\nu^0)^{\frac{2}{\gamma_\nu-1}}\in H^{s-1}\cap L_{\text{loc}}^1$,
\[
\vp^1=-\dfrac{1}{d(d-2)\omega_d}\int_{\R^d}\dfrac{C_i(n_i^0)^{\frac{2}{\gamma_i-1}}-C_e(n_e^0)^{\frac{2}{\gamma_e-1}}}{|x-y|^{d-2}}\rmd y<\infty,
\]
which is exactly the solution to the Poisson equation \[
\Delta \vp^1=C_i(n_i^0)^{\frac{2}{\gamma_i-1}}-C_e(n_e^0)^{\frac{2}{\gamma_e-1}},
\] 
and thus $\D\vp^1\in L^\infty([0,T];H^s)$ according to Lemma \ref{Poisson}. Then we construct approximate solutions $(W_\nu^{m+1}:=(n_\nu^{m+1},(w_\nu^{m+1})^\top)^\top,\vp^{m+1})$ inductively as follows: Assume that $(W_\nu^m:=(n_\nu^{m},(w_\nu^{m})^\top)^\top,\vp^{m})$ has been defined for $m\geq 1$, and let $(W_\nu^{m+1}, \vp^{m+1})$ be the unique solution to \eqref{LTEP}--\eqref{LTEPini} with $(\tilde{\pi}_\nu,\tilde{v}_\nu,\tilde{\Phi})$ replaced by $(n_\nu^m,w_\nu^m,\vp^{m})$ as follows
\begin{equation}\label{lintruniter}
\begin{cases} \pt W_\nu^{m+1}+ \displaystyle \sum_{j=1}^d A_\nu^j(W_\nu^m;\hu_\nu^N,1)\pa_{x_j}W_\nu^{m+1}=Q_\nu^1(\D \vp^m)-Q_\nu^2(W_\nu^{m+1};\D\hu_\nu),\\
		\Delta\vp^{m+1}=C_i \gp_{\frac{2}{\gamma_i-1}}(n_i^m)-C_e \gp_{\frac{2}{\gamma_e-1}}(n_e^m),\\
\end{cases}
\end{equation}
with the initial data
\begin{equation}\label{lintruniterini}
	(n_\nu^{m+1},w_\nu^{m+1})(0,x)=(n_\nu^0(x),0).
\end{equation}

For simpler interpretation, let $\hat{\Omega}$ be the following closed ball for $(f,g,h)(t,x)$ such that
\[
\sup_{0\leq t\leq T^*}(\|f(t)\|_s+\|g(t)\|_s)\leq c_1,\quad \sup_{0\leq t\leq T^*}\|\D h\|_{s}\leq c_2, \quad \operatorname{supp}_x f(t,\cdot) \subseteq\overline B_{\mathfrak c(t)}, \ \  0\le t\le T^*.
\]
Since $c_1>c_0$ and $w_\nu^0=0$, it holds
\[
\|n_\nu^0\|_{s}+\|w_\nu^0\|_s\leq c_1, \qquad \|\D\vp^0\|_s=0\leq c_2.
\]
Besides, by an argument similar to that in Lemma \ref{nablavp}, one has
\[
\|\D \vp^1(t)\|_s\leq c_2, \qquad \text{for}\,\,\,t\in(0,T^*].
\]
This yields that $(n_\nu^0,w_\nu^0,\vp^0)\in\hat{\Omega}$ as well as $(n_\nu^1,w_\nu^1,\vp^1)\in\hat{\Omega}$.  The support calculation in the proof of \eqref{supp-n} shows that
\[
\operatorname{supp}_x n_\nu^{m+1}(t,\cdot) \subseteq\overline B_{\mathfrak c(t)} \quad\text{whenever}\quad \operatorname{supp}_x n_\nu^m(t,\cdot) \subseteq\overline B_{\mathfrak c(t)}.
\]
The original initial iterate satisfies this condition since $\operatorname{supp}n_\nu^0\subseteq\overline B_{\mathcal R} \subseteq\overline B_{\mathfrak c(t)}$. By the argument used at the end of \S\ref{3.1}, if $(n_\nu^m,w_\nu^m,\vp^m)\in \hat{\Omega}$, one has $(n_\nu^{m+1},w_\nu^{m+1}, \vp^{m+1})\in\hat{\Omega}$, which shows that the map $\mathcal{T}:(n_\nu^m,w_\nu^m,\vp^m)\to (n_\nu^{m+1},w_\nu^{m+1}, \vp^{m+1})$ defined by \eqref{lintruniter}--\eqref{lintruniterini} maps $\hat{\Omega}$ into itself.

We now prove the convergence of the whole sequence $(n_\nu^m,w_\nu^m,\vp^m)$ to a limit $(n_\nu,w_\nu,\vp)$ in some strong sense. Define $\bar{W}_\nu^{m+1}:=(\bar{n}_\nu^{m+1},(\bar{w}_\nu^{m+1})^\top)^\top$ and $\bar{\vp}^{m+1}$ by
\[
\bar{n}_\nu^{m+1}=n_\nu^{m+1}-n_\nu^m, \qquad \bar{w}_\nu^{m+1}=w_\nu^{m+1}-w_\nu^m, \qquad \bar{\vp}^{m+1}=\vp^{m+1}-\vp^m.
\]
Notice that $A_\nu^j(\cdot;\widehat u_\nu^N,1)$ is affine, while $Q_\nu^1(\cdot)$ and $Q_\nu^2(\cdot;\nabla\widehat u_\nu)$ are linear. Then it follows from \eqref{lintruniter}--\eqref{lintruniterini} that $\bar{W}_\nu^{m+1}$ and $\bar{\vp}^{m+1}$ satisfy
\begin{equation}\label{chait}
\begin{cases} \quad \pt \bar{W}_\nu^{m+1} +\displaystyle \sum_{j=1}^d A_\nu^j(W_\nu^m;\hu_\nu^N,1)\pa_{x_j}\bar{W}_\nu^{m+1}\\
		=Q_\nu^1(\D\bar{\vp}^m)-Q_\nu^2(\bar{W}_\nu^{m+1};\D\hu_\nu)-\displaystyle \sum_{j=1}^d A_\nu^j(\bar{W}_\nu^m;0,1)\pa_{x_j}W_\nu^{m},\\
		\quad \Delta\bar{\vp}^{m+1}=\sum_{\nu=i,e}q_\nu C_\nu (\gp_{\frac{2}{\gamma_\nu-1}}(n_\nu^{m})-\gp_{\frac{2}{\gamma_\nu-1}}(n_\nu^{m-1})), \end{cases}
\end{equation}
with the initial data
\begin{equation*}
	(\bar{n}_\nu^{m+1},\bar{w}_\nu^{m+1})(0,x)=(0,0).
\end{equation*}
Multiplying \eqref{chait}$_1$ by $2\bar{W}_\nu^{m+1}$ and integrating over $\mathbb{R}^d$ yield that
\begin{align}\label{midd10}
	\dfrac{\rmd}{\rmd t}|\bar{W}_\nu^{m+1}|_2^2\leq& C(|\D W_\nu^m|_\infty+|\D\hu_\nu^N|_\infty+|\D\hu_\nu|_\infty)|\bar{W}_\nu^{m+1}|_2^2\nonumber\\
	&+C|\bar{w}_\nu^{m+1}|_2|\D\bar{\vp}^m|_2 +C|\bar{W}_\nu^{m+1}|_2|\D W_\nu^m|_\infty|\bar{W}_\nu^{m}|_2\\
	\leq& C|\bar{W}_\nu^{m+1}|_2^2+C(|\D\bar{\vp}^m|_2^2+|\bar{W}_\nu^{m}|_2^2).\nonumber
\end{align}
Besides, by Lemma \ref{Poisson}, the Poisson equation \eqref{chait}$_2$ shows that for $m\geq 2$,
\begin{align}\label{naka000}
\begin{split}
		|\D\bar{\vp}^{m}|_2\leq& C\sum_{\nu=i,e}|(\gp_{\frac{2}{\gamma_\nu-1}}(n_\nu^{m-1})-\gp_{\frac{2}{\gamma_\nu-1}}(n_\nu^{m-2}))|_{\frac{2d}{d+2}}\\
		\leq& C\sum_{\nu=i,e}|\bar{n}_\nu^{m-1}|_2\int_0^1|\gp_{\frac{2}{\gamma_\nu-1}}'(n_\nu^{m-2}+s\bar{n}_\nu^{m-1})|_d\rmd s\leq C\sum_{\nu=i,e}|\bar{n}_\nu^{m-1}|_2.
\end{split}
\end{align}
Hence, substituting \eqref{naka000} into \eqref{midd10} implies that for $m\geq 2$,
\[
\sum_{\nu=i,e}\dfrac{\rmd}{\rmd t}|\bar{W}_\nu^{m+1}|_2^2\leq C\sum_{\nu=i,e}\left(|\bar{W}_\nu^{m+1}|_2^2+|\bar{W}_\nu^{m}|_2^2+|\bar{W}_\nu^{m-1}|_2^2\right).
\]
By  Gr\"{o}nwall's inequality, it holds that
\[
\sum_{\nu=i,e}|\bar{W}_\nu^{m+1}(t)|_2^2\leq  Ct e^{Ct}\sup_{s\in[0,t]}\sum_{\nu=i,e}(|\bar{W}_\nu^{m}(s)|_2^2+|\bar{W}_\nu^{m-1}(s)|_2^2).
\]
Choose $T^{**}\in(0,T^*)$ small enough such that
\[
CT^{**}e^{CT^{**}}<\dfrac{1}{4}.
\]
Applying Lemma \ref{lem:two-step-convergence} to $\{(W_i^m,W_e^m)\}$ in $L^\infty([0,T^{**}];L^2(\R^d;\R^{2d+2}))$, one obtains from the preceding estimate and \eqref{locunieq2} that
\[
\sum_{m=1}^\infty\left(\sup_{t\in[0,T^{**}]}\sum_{\nu=i,e}|\bar{W}_\nu^m(t)|_2^2\right)^{1/2}<\infty.
\]
Moreover, \eqref{naka000} yields
\[
\sum_{m=1}^\infty\sup_{t\in[0,T^{**}]}|\D\bar{\vp}^m(t)|_2 \leq \sup_{t\in[0,T^{**}]}|\D\bar{\vp}^1(t)|_2 +C\sum_{m=1}^\infty\left(\sup_{t\in[0,T^{**}]}\sum_{\nu=i,e}|\bar{W}_\nu^m(t)|_2^2\right)^{1/2}<\infty.
\]
Here the first term is finite by \eqref{locunieq2} and $\vp^0=0$. Thus $\{(n_\nu^m,w_\nu^m,\D\vp^m)\}$ is Cauchy in $L^\infty([0,T^{**}];L^2)$. Interpolating with the uniform $H^s$ bound \eqref{locunieq2}, one obtains
\begin{equation}\label{converloc}
	(n_\nu^m,w_\nu^m,\D\vp^m)\to(n_\nu,w_\nu,\D\vp)\quad\text{in}\,\,\,L^\infty([0,T^{**}];H^{s-1}).
\end{equation}
Due to \eqref{locunieq2} and the weak lower semicontinuity of Sobolev norms, $(n_\nu,w_\nu,\vp)$ still satisfies \eqref{locunieq2}. Now \eqref{converloc} implies that $(n_\nu,w_\nu,\vp)$ satisfies \eqref{TEP}$_1$, together with:
\[
\Delta \vp=C_i \gp_{\frac{2}{\gamma_i-1}}(n_i)-C_e \gp_{\frac{2}{\gamma_e-1}}(n_e),
\]
in the sense of distributions. By the characteristic formula for the limiting density equation \eqref{TEP}$_1$ and $n_\nu^0\geq 0$, one has $n_\nu\geq 0$. Hence $\gp_\frac{2}{\gamma_\nu-1}(n_\nu)=n_\nu^\frac{2}{\gamma_\nu-1}$ and $(n_\nu,w_\nu,\vp)$ satisfies \eqref{TEP} in the sense of distributions. So the existence of a strong solution is proved.

We then prove the uniqueness of solutions. For $t\in[0,T^{**}]$, $n_\nu$ is still compactly supported. Let $(n_{\nu,1},w_{\nu,1},\vp_1)$ and $(n_{\nu,2},w_{\nu,2},\vp_2)$ be two solutions to \eqref{TEP} satisfying the uniform \textit{a priori} estimate \eqref{locunieq2}. Set
\[
\tilde{n}_\nu=n_{\nu,1}-n_{\nu,2}, \qquad \tilde{w}_\nu=w_{\nu,1}-w_{\nu,2}, \qquad \tilde{\vp}=\vp_1-\vp_2.
\]
Then similar to \eqref{chait}, $(\tilde{n}_\nu,\tilde{w}_\nu,\tilde{\vp}) $ satisfies
\begin{equation*}
\begin{cases} \quad \pt \tilde{n}_\nu+(w_{\nu,1}+\hu_\nu^N)\cdot\nabla \tilde{n}_\nu+\dfrac{\gamma_\nu-1}{2}n_{\nu,1}\dive \tilde{w}_\nu\\
		=-\tilde{w}_\nu\cdot\D n_{\nu,2}-\dfrac{\gamma_\nu-1}{2}\tilde{n}_\nu\dive  w_{\nu,2}-\dfrac{\gamma_\nu-1}{2}\tilde{n}_\nu\dive \hu_\nu,\\[2mm]
		\quad\pt \tilde{w}_\nu+((w_{\nu,1}+\hu_\nu^N)\cdot\D)\tilde{w}_\nu+\dfrac{\gamma_\nu-1}{2}n_{\nu,1}\nabla \tilde{n}_\nu\\
		=-(\tilde{w}_\nu\cdot\D)w_{\nu,2}-(\tilde{w}_\nu\cdot\nabla)\hu_\nu-\dfrac{\gamma_\nu-1}{2}\tilde{n}_\nu\nabla n_{\nu,2}+q_\nu\D\tilde{\vp},\\[2mm]
		\quad\Delta\tilde{\vp}=\sum_{\nu=i,e}q_\nu C_\nu (n_{\nu,1}^{2/(\gamma_\nu-1)}-n_{\nu,2}^{2/(\gamma_\nu-1)}), \end{cases}
\end{equation*}
with the initial data
\[
(\tilde{n}_\nu,\tilde{w}_\nu)(0,x)=(0,0).
\]
Using arguments similar to those for system \eqref{chait}, it holds that for $t\in[0,T^{**}]$
\[
\dfrac{\rmd}{\rmd t}\sum_{\nu=i,e}|(\tilde{n}_\nu,\tilde{w}_\nu)(t)|_2^2\leq C\sum_{\nu=i,e}|(\tilde{n}_\nu,\tilde{w}_\nu)(t)|_2^2,
\]
which yields $\tilde{n}_\nu=\tilde{w}_\nu=0$. Then the uniqueness is obtained.

We now prove the time continuity of solutions. Notice that for every $N\geq 1$, 
\[
(n,w,\D\vp)\in L^\infty([0,T^{**}];H^s), \qquad (\pt n,\pt w)\in L^\infty([0,T^{**}];H^{s-1}).
\]
Besides, taking the time derivative of both sides of  \eqref{TEP}$_2$, one obtains
\[
\Delta\pt\vp=\sum_{\nu=i,e}q_\nu C_\nu\frac{2}{\gamma_\nu-1}n_\nu^{\frac{3-\gamma_\nu}{\gamma_\nu-1}}\pt n_\nu,
\]
which shows 
$$\pt\D\vp\in L^\infty([0,T^{**}];H^{s-1}).$$
This yields for $s'\in[0,s)$,
\[
(n_\nu,w_\nu,\D\vp)\in C([0,T^{**}];H^{s'}\cap\text{weak-}H^s).
\]
Besides, using the same arguments as in the proof of Lemma \ref{locuni}, one obtains 
\[
\limsup_{t\to 0^+}\|(n_\nu,w_\nu)(t)\|_s\leq \|(n_\nu,w_\nu)(0,\cdot)\|_s=\|n_\nu^0\|_s, 
\]
which shows that $(n_\nu,w_\nu)$ are right continuous at $t=0$ in $H^s$. The time reversibility of equations in \eqref{TEP} yields $(n_\nu,w_\nu)\in C([0,T^{**}];H^s)$ and thus one obtains $(n_\nu,w_\nu)\in C^1([0,T^{**}];H^{s-1})$ by \eqref{TEP}$_1$.  In addition, using the same argument as in \eqref{mid}, one obtains $\D\vp\in C([0,T^{**}];H^s)$. Set $T_{\mathrm{loc}}=T^{**}$. Hence the proof of Theorem \ref{comthmloc} is complete. \hfill $\square$

\section{Global-in-time well-posedness for compactly supported initial data}

In this section, we establish global existence, uniqueness and uniform estimates for the Cauchy problem \eqref{sys-symm} with compactly supported initial densities. We first pass to the limit $N\to\infty$ to obtain a local solution of the untruncated system. We then derive estimates independent of $\eps$ and the initial support radii and use them to obtain global existence. Throughout this section, $C\geq1$ denotes a generic constant depending only on $(A_\nu,\gamma_\nu,\kappa,s,d,q,\allowbreak\|u_i^0\|_\Xi,\|u_e^0\|_\Xi,\allowbreak|u_i^0(0)|,|u_e^0(0)|)$ and independent of $T$, $N$, $\eps$ and $R_\nu$, which may change from line to line. We use $C_R$ for a generic constant that may additionally depend on the initial support radii. The finite-time $H^s$ bounds may depend on $T$ and $R_\nu$.

The main result is as follows.

\begin{theorem}\label{thm4.1}
Let $(s,d,q,\gamma_\nu)$ satisfy the parameter conditions of \textup{Theorem~\ref{globalthm}}. Assume that the initial data $(\rho_\nu^0,u_\nu^0)$ satisfy assumptions $(\mathrm{A}_1)$--$(\mathrm{A}_2)$ in \textup{Theorem~\ref{globalthm}}, with $\operatorname{supp}_x\rho_\nu^0\subset B_{R_\nu}$ for some $R_\nu>0$, $\nu=i,e$. Then the Cauchy problem \eqref{sys-symm} admits a unique global regular solution $(n_\nu,\eps^{-\beta_\nu}w_\nu,\vp)$ on $[0,\infty)\times\R^d$ satisfying
\begin{equation}\label{regu4.1}
\begin{aligned}
&(n_\nu,\eps^{-\beta_\nu}w_\nu)\in C([0,T];H^s)\cap C^1([0,T];H_{\mathrm{loc}}^{s-1}),\qquad \D\vp\in C([0,T];H^s),
\end{aligned}
\end{equation}
for every $T>0$. Moreover, for every $t\in[0,T]$ and $\sigma\in[1,s]$, the following estimates hold:
\begin{equation}\label{uni1}
\begin{split}
|(n_\nu,\eps^{-\beta_\nu}w_\nu)(t)|_\infty\leq& C(1+t)^{-\min\{1,\frac{d(\gamma_\nu-1)}{2}\}},\\
|n_\nu(t)|_q\leq& C(1+t)^{\frac{d}{q}-\min\{1,\frac{d(\gamma_\nu-1)}{2}\}},\\
\|(n_\nu,\eps^{-\beta_\nu}w_\nu)(t)\|_{\dot H^\sigma}\leq& C(1+t)^{\frac{d}{2}-\sigma-\min\{1,\frac{d(\gamma_\nu-1)}{2}\}},
\end{split}
\end{equation}
and
\begin{equation}\label{univp2}
(1+t)^{\frac{2\bar b}{\bar\gamma-1}-1}\Bigl(|\D\vp(t)|_\infty+(1+t)^{-\frac{d}{q}}|\D\vp(t)|_q+(1+t)^{\sigma-\frac{d}{2}}\|\D\vp(t)\|_{\dot H^\sigma}\Bigr)\leq C,
\end{equation}
where $\bar b=\max_{\nu=i,e}\min\{d(\gamma_\nu-1)/2,1\}$.
\end{theorem}

\subsection{Passing to the limit \texorpdfstring{$N\to\infty$}{}}\label{s4.1}

Let $(n_\nu^N,w_\nu^N,\vp^N)$ be the local classical solution to \eqref{TEP} on $[0,T_{\mathrm{loc}}]\times\R^d$ given by Theorem~\ref{comthmloc}. By \eqref{locunieq2},
\[
\sup_{0\leq t\leq T_{\mathrm{loc}}}\|(n_\nu^N,\eps^{-\beta_\nu}w_\nu^N,\D\vp^N)(t)\|_s\leq C_R,
\]
where both $T_{\mathrm{loc}}$ and $C_R$ may depend on the support radii $R_\nu$, but are independent of $N$. Thus, up to a subsequence,
\[
(n_\nu^N,\eps^{-\beta_\nu}w_\nu^N,\D\vp^N)\stackrel{*}{\rightharpoonup}(n_\nu,\eps^{-\beta_\nu}w_\nu,\D\vp)\quad\text{weakly-* in }L^\infty([0,T_{\mathrm{loc}}];H^s).
\]
For each $R_0>0$, \eqref{TEP}$_1$ also gives an $N$-independent bound for $\pt(n_\nu^N,\eps^{-\beta_\nu}w_\nu^N)$ in $L^\infty([0,T_{\mathrm{loc}}];H^{s-1}(B_{R_0}))$, which, along with  Lemma~\ref{aubin} and a diagonal argument, yields, along a further subsequence,
\[
(n_\nu^N,\eps^{-\beta_\nu}w_\nu^N)\to(n_\nu,\eps^{-\beta_\nu}w_\nu)\quad\text{strongly in }C([0,T_{\mathrm{loc}}];H_{\mathrm{loc}}^{s-1}).
\]
Since $\hu_\nu^N=\hu_\nu$ on $B_{R_0}$ for $N\geq R_0$, these convergences suffice to pass to the limit in \eqref{TEP} in the sense of distributions, yielding a solution to \eqref{sys-symm} with the prescribed initial data.

The nonnegativity of $n_\nu^N$ and the support bound \eqref{supp-n} also pass to the limit. Hence 
\[
n_\nu\geq0,\qquad \operatorname{supp}_x n_\nu(t,\cdot)\subseteq\overline B_{\mathfrak c(t)},\qquad 0\leq t\leq T_{\mathrm{loc}}.
\]
By \eqref{nablau}, the map $x\mapsto\hu_\nu(t,x)-x/(1+t)$ is Lipschitz with constant bounded by $CM/(1+t)^2$, where $M$ is as in Lemma \ref{hu} and $C$ depends only on $d$. Together with the identity $\hu_\nu(t,t u_\nu^0(0))=u_\nu^0(0)$, this gives
\[
\left|\hu_\nu(t,x)-\frac{x+u_\nu^0(0)}{1+t}\right|
\leq \frac{CM}{(1+t)^2}\left|x-t u_\nu^0(0)\right|.
\]
Consequently,
\[
|\hu_\nu(t,x)|\leq \frac{C\bigl(|u_\nu^0(0)|+|x|\bigr)}{1+t},\qquad t\geq0,
\]
where $C$ is independent of $T$. We next show the following lemma concerning the estimates for the convection terms.

\begin{lemma}\label{transport-ibp} Let $T>0$, and let $\widehat u$ be the solution to \eqref{burgers} given in \textup{Lemma \ref{hu}}. Suppose that $V\in C_c^\infty([0,T]\times \R^d)$. Then, for every $t\in[0,T]$, one has
\[
	\int_{\mathbb R^d} V\cdot(\widehat u\cdot\nabla V)\rmd x=-\frac12\int_{\mathbb R^d}(\operatorname{div}\widehat u)|V|^2\rmd x.
\]
\end{lemma}

\begin{proof}
Fix $t\in[0,T]$ and let $F^N$ be as in \eqref{bump}. Integration by parts gives
\[
	\int_{\R^d}F^N\dive(\widehat u|V|^2)\rmd x=-\int_{\R^d}(\widehat u\cdot\D F^N)|V|^2\rmd x.
\]
By the preceding bound for $\widehat u$,
\[
\begin{aligned}
		&\left|\int_{\R^d}(\widehat u\cdot\D F^N)|V|^2\rmd x\right|\leq \frac{C\bigl(1+|u^0(0)|/N\bigr)}{1+t}\int_{|x|\ge N}|V(t,x)|^2\rmd x\longrightarrow0.
\end{aligned}
\]
Since $V(t,\cdot)\in C_c^\infty$ and $\widehat u(t,\cdot)\in C^1$, dominated convergence yields $\int_{\R^d}\dive(\widehat u|V|^2)\rmd x=0$. Expanding the divergence proves the identity.
\end{proof}

We next establish strong time continuity on any interval $[0,T]$ of existence on which $(W_\nu,\D\vp)\in L^\infty([0,T];H^s)$, where $W_\nu=(n_\nu,\eps^{-\beta_\nu}w_\nu^\top)^\top$. The equation gives $\pt W_\nu\in L^\infty([0,T];H^{s-1}_{\mathrm{loc}})$. Together with the $H^s$ bound, this gives a weakly continuous representative $W_\nu\in C_{\mathrm w}([0,T];H^s)$. For $|\alpha|\in\{0,s\}$, write $A_\nu^j=A_\nu^j(W_\nu;\hu_\nu,\eps)$ and set
\[
\begin{aligned}
V_{\nu,\alpha}&=\pa^\alpha W_\nu,\qquad S_{\nu,\alpha}&=\pa^\alpha\bigl(Q_\nu^1(\D\vp)-Q_\nu^2(W_\nu;\D\hu_\nu)\bigr)
-\sum_{j=1}^d[\pa^\alpha,A_\nu^j]\pa_{x_j}W_\nu.
\end{aligned}
\]
Lemmas~\ref{comm1}, \ref{comm2} and \ref{hu} imply $S_{\nu,\alpha}\in L^\infty([0,T];L^2)$ and $\D A_\nu^j\in L^\infty([0,T]\times\R^d)$. Let $F^L=F(|x|/L)$ be the cutoff in \eqref{bump}. Spatially mollifying the equation for $V_{\nu,\alpha}$, testing with its mollified value times $F^L$, and integrating over any $[t_1,t_2]\subseteq[0,T]$ therefore gives the localized energy identity. For fixed $L$, the mollifier commutators tend to zero by the spatial Lipschitz bounds for $A_\nu^j$. The cutoff boundary terms then vanish as $L\to\infty$, since the preceding bound for $\hu_\nu$ gives, for $L\geq1$,
\[
\begin{aligned}
&\int_0^T\int_{\R^d}|\D F^L|(|\hu_\nu|+|W_\nu|)|V_{\nu,\alpha}|^2\rmd x\rmd t\\
\leq&\ C\int_0^T(1+|W_\nu(t)|_\infty)
\int_{|x|\geq L}|V_{\nu,\alpha}(t,x)|^2\rmd x\rmd t\longrightarrow0.
\end{aligned}
\]
Retaining the endpoint values of the weakly continuous representative in these limits, one obtains, for every $0\leq t_1\leq t_2\leq T$,
\[
\begin{aligned}
&\frac12\bigl(|V_{\nu,\alpha}(t_2)|_2^2-|V_{\nu,\alpha}(t_1)|_2^2\bigr)=\int_{t_1}^{t_2}\left(\frac12\sum_{j=1}^d
\langle V_{\nu,\alpha},(\pa_{x_j}A_\nu^j)V_{\nu,\alpha}\rangle_{L^2}
+\langle V_{\nu,\alpha},S_{\nu,\alpha}\rangle_{L^2}\right)\rmd t.
\end{aligned}
\]
The integrand is in $L^1(0,T)$. Summing over $|\alpha|\in\{0,s\}$ proves continuity of a norm equivalent to the $H^s$ norm. Weak continuity then yields
\[
(n_\nu,\eps^{-\beta_\nu}w_\nu)\in C([0,T];H^s).
\]
Moreover, the characteristic formula for the density and the linear growth of $\hu_\nu+w_\nu$ show that the density support remains uniformly bounded on $[0,T]$. Thus $H^s\hookrightarrow L^\infty$ also gives $n_\nu\in C([0,T];L^q)$ for every $1\leq q<\infty$. These conclusions use only bounds on the interval of existence and do not require the global estimates below.

We prove local uniqueness by the $L^2$ difference energy argument of \S\ref{s3}, with weight $\eps^{-\beta_\nu}$ on the velocity differences. Lemma~\ref{transport-ibp} and the preceding spatial localization justify integration by parts for the linearly growing $\hu_\nu$; the cutoff errors integrated in time vanish as $L\to\infty$ since the differences belong to $L^\infty([0,T];L^2)$. Using the same Poisson estimate and Gr\"{o}nwall's inequality, one obtains uniqueness of $(n_\nu,w_\nu,\D\vp)$ in the local $H^s$ class.

\subsection{Uniform global  \textit{a priori} estimates} Fix any finite interval $[0,T]$ on which a solution of \eqref{sys-symm} with the prescribed initial data exists in the local $H^s$ class. Substituting \eqref{nablau} into \eqref{sys-symm}, one has
\begin{equation}\label{TEPeps}
\begin{cases} \quad\pt n_e+(w_e+\hu_e)\cdot\nabla n_e+\dfrac{\gamma_e-1}{2} n_e\dive w_e\\[2mm]
		=-\dfrac{\gamma_e-1}{2} n_e\Big(\dfrac{d}{1+t}+\dfrac{\text{Tr}K_e}{(1+t)^2}\Big),\\[2.2mm]
		\quad \pt w_e+((w_e+\hu_e)\cdot\nabla)w_e+\dfrac{\gamma_e-1}{2}n_e\nabla n_e\\
		=-\Big(\dfrac{w_e}{1+t}+\dfrac{w_e\cdot K_e}{(1+t)^2}\Big)-\nabla\vp,\\[2mm]
		\quad\pt n_i+(w_i+\hu_i)\cdot\nabla n_i+\dfrac{\gamma_i-1}{2} n_i\dive w_i\\[2mm]
		=-\dfrac{\gamma_i-1}{2} n_i\Big(\dfrac{d}{1+t}+\dfrac{\text{Tr}K_i}{(1+t)^2}\Big),\\[2.2mm]
		\quad \eps^{-2}\pt w_i+\eps^{-2}((w_i+\hu_i)\cdot\nabla)w_i+\dfrac{\gamma_i-1}{2}n_i\D n_i\\
		=-\Big(\dfrac{\eps^{-2}w_i}{1+t}+\dfrac{\eps^{-2}w_i\cdot K_i}{(1+t)^2}\Big)+\nabla\vp,\\[2mm]
		\quad \Delta\vp=C_i n_i^{2/(\gamma_i-1)}-C_e n_e^{2/(\gamma_e-1)}, \end{cases}
\end{equation}
with $\text{Tr}K_\nu$ being the trace of $K_\nu$. The initial data can be given as
\begin{equation*}
	(n_\nu,w_\nu)(0,x)=(n_\nu^0(x),0).
\end{equation*}

On this interval, we derive time-weighted \textit{a priori} estimates in the spaces
\[
n_\nu\in L^\infty([0,T]; \Gamma\cap L^q), \qquad (\eps^{-\beta_\nu}w_\nu,\D\vp)\in L^\infty([0,T]; \Gamma),
\]
with constants independent of $T$, $\eps$ and the initial support radii $R_\nu$, where $(\beta_e,\beta_i)=(0,1)$. These estimates will be used in \S\ref{s4.4} to continue the local solution globally. For simplicity, we introduce $W_\nu=(n_\nu, (\eps^{-\beta_\nu}w_\nu)^\top)^\top$ and the following
\begin{equation*}
\begin{split}
		\dot{X}_{\nu,l}:=&\|W_\nu\|_{\dot{H}^l},\quad\quad\quad\,\, b_\nu:=\min\Big\{\dfrac{d(\gamma_\nu-1)}{2},1\Big\},\quad b=\min\{b_i,b_e\},\\
		\,\, c_{d,\gamma_\nu,l,\nu}:=&-\dfrac{d}{2}+l+b_\nu,\quad c_{d,l}=\min\{c_{d,\gamma_i,l,i},c_{d,\gamma_e,l,e}\}, \quad \text{for}\, 1\leq l\leq s.
\end{split}
\end{equation*}

One first has the $L^\infty$ estimates on $W_\nu$.
\begin{lemma}\label{chp4-1} There exist constants  $c_{31}^\nu>0 (\nu=i,e)$ depending only on the generic constant $C$, such that
\begin{align}\label{emnlinfty}
		|n_\nu(t)|_\infty\leq& \frac{e^{\frac{c_{31}^\nu t}{1+t}}|n_\nu^0|_\infty}{(1+t)^{b_\nu}}+\frac{C}{(1+t)^{b_\nu}}\int_0^t(1+\tau)^{b_\nu}|n_\nu(\tau)|_\infty\dot{X}_{\nu,1}^{\frac{2s-d-2}{2s-2}}\dot{X}_{\nu,s}^{\frac{d}{2s-2}}\rmd\tau,\\
		\label{emwlinfty} |\eps^{-\beta_\nu}w_\nu(t)|_\infty\leq&\frac{C}{(1+t)^{b_\nu}}\int_0^t(1+\tau)^{b_\nu}\Big(\mathcal{I}(\tau)+| n_\nu(\tau)|_\infty\dot{X}_{\nu,1}^{\frac{2s-d-2}{2s-2}}\dot{X}_{\nu,s}^{\frac{d}{2s-2}}\Big)\rmd\tau,
\end{align}
where 
\[
	\mathcal{I}(\tau)=\Big(\sum_{\nu=i,e}|n_\nu|_q^{\frac{q}{2}\frac{2s-d}{2s-2}}|n_\nu|_\infty^{\big(\frac{2}{\gamma_\nu-1}-\frac{q}{2}\big)\frac{2s-d}{2s-2}}\Big)\Big(\sum_{\nu=i,e}|n_\nu|_\infty^{\frac{3-\gamma_\nu}{\gamma_\nu-1}\frac{d-2}{2s-2}}\dot{X}_{\nu,1}^{\frac{1}{s-1}\frac{d-2}{2s-2}}\dot{X}_{\nu,s}^{\frac{s-2}{s-1}\frac{d-2}{2s-2}}\Big).
\]
\end{lemma}
\begin{proof} Notice that \eqref{TEPeps} implies that there exist two constants $c_{31}^\nu>0 (\nu=i,e)$ depending only on $C$, such that
\begin{equation}\label{naka4-2}
		\pt n_\nu+(w_\nu+\hu_\nu)\cdot\nabla n_\nu+\dfrac{b_\nu}{1+t} n_\nu\leq C|n_\nu|_\infty|\dive w_\nu|_\infty+\frac{c_{31}^\nu n_\nu}{(1+t)^2}.
\end{equation}
Multiplying \eqref{naka4-2}  by $(1+t)^{b_\nu}e^{-\frac{c_{31}^\nu t}{1+t}}$  yields
\begin{align*}
		\pt\big((1+t)^{b_\nu} e^{-\frac{c_{31}^\nu t}{1+t}}n_\nu \big)+(w_\nu+\hu_\nu)\cdot\nabla \big((1+t)^{b_\nu}e^{-\frac{c_{31}^\nu t}{1+t}} n_\nu\big)	\leq|(1+t)^{b_\nu}e^{-\frac{c_{31}^\nu t}{1+t}} n_\nu|_\infty|\dive w_\nu|_\infty.
\end{align*}
Then one obtains
\begin{align*}
		|(1+t)^{b_\nu}e^{-\frac{c_{31}^\nu t}{1+t}}n_\nu(t)|_\infty\leq  C|n_\nu^0|_\infty+C\int_0^t|(1+\tau)^{b_\nu}n_\nu(\tau)|_\infty|\dive w_\nu(\tau)|_\infty\rmd\tau,
\end{align*}
which, along with the following basic interpolation,
\begin{equation}\label{embed00}
		|\D W_\nu|_\infty\leq C\dot{X}_{\nu,1}^{\frac{2s-d-2}{2s-2}}\dot{X}_{\nu,s}^{\frac{d}{2s-2}},
\end{equation}
yields the desired estimate \eqref{emnlinfty}. Similarly, \eqref{TEPeps}$_2$ and \eqref{TEPeps}$_4$ imply
\begin{align}\label{nakaa}
		|\eps^{-\beta_\nu}(1+t)^{b_\nu}w_\nu(t)|_\infty\leq C\int_0^t\Big((1+\tau)^{b_\nu}\left(| n_\nu(\tau)|_\infty|\D n_\nu(\tau)|_\infty+|\D\vp(\tau)|_\infty\right)\Big)\rmd\tau.
\end{align}

We then bound $|\D \vp|_\infty$. For $s>\dfrac{d}{2}+1$, by the Poisson equation,
\begin{align*}
\begin{split}
			|\D\vp|_\infty\leq& C\|\D\vp\|_{\dot{H}^1}^{\frac{2s-d}{2s-2}}\|\D\vp\|_{\dot{H}^s}^{\frac{d-2}{2s-2}}\\
			\leq& C\Big(\sum_{\nu=i,e}\Big|n_\nu^{2/(\gamma_\nu-1)}\Big|_2^{\frac{2s-d}{2s-2}}\Big)\Big(\sum_{\nu=i,e}\Big\|n_\nu^{2/(\gamma_\nu-1)}\Big\|_{\dot{H}^{s-1}}^{\frac{d-2}{2s-2}}\Big).
\end{split}
\end{align*}
By interpolation, for $\frac{2}{\gamma_\nu-1}>\frac{q}{2}$, one obtains
\begin{align*}
		|n_\nu^{2/(\gamma_\nu-1)}|_2= \Big(\int_{\R^d}|n_\nu|^{\frac{4}{\gamma_\nu-1}}\rmd x\Big)^{\frac{1}{2}} \leq C|n_\nu|_\infty^{\frac{2}{\gamma_\nu-1}-\frac{q}{2}}|n_\nu|_q^{\frac{q}{2}},
\end{align*}
and by Lemma \ref{comm0}, for $\gamma_\nu\leq 2$ and $0\leq s-1<\frac{2}{\gamma_\nu-1}+\frac{1}{2}$, one has
\begin{align}\label{Hs}
		\|n_\nu^{2/(\gamma_\nu-1)}\|_{\dot{H}^{s-1}}\leq C|n_\nu|_\infty^{\frac{3-\gamma_\nu}{\gamma_\nu-1}}\|n_\nu\|_{\dot{H}^{s-1}}\leq C|n_\nu|_\infty^{\frac{3-\gamma_\nu}{\gamma_\nu-1}}\dot{X}_{\nu,1}^{\frac{1}{s-1}}\dot{X}_{\nu,s}^{\frac{s-2}{s-1}}.
\end{align}
Of course, there is no upper bound for $s$ if $\frac{2}{\gamma_\nu-1}$ is an integer. Similarly, since $s>\frac{d}{2}+1 $ and $s\geq 3$, one needs the following restriction on $\gamma_\nu$:
\[
	\bar{\gamma}<1+\min\Big\{\frac{4}{3},\dfrac{4}{d-1}\Big\}.
\]
Combining estimates \eqref{nakaa}-\eqref{Hs} yields \eqref{emwlinfty}.
\end{proof}

The next lemma studies the $L^q$ estimate for $n_\nu$.
\begin{lemma}  For $q$ satisfying \eqref{res1}, one has
\begin{align}\label{mid88}
		\dfrac{\rmd}{\rmd t}|n_\nu|_q+\Big(\dfrac{d(\gamma_\nu-1)}{2}-\dfrac{d}{q}\Big)\dfrac{1}{1+t}|n_\nu|_q\leq C|n_\nu|_q\Big(\dot{X}_{\nu,1}^{\frac{2s-d-2}{2s-2}}\dot{X}_{\nu,s}^{\frac{d}{2s-2}}+\dfrac{1}{(1+t)^2}\Big).
\end{align}
\end{lemma}
\begin{proof} Direct calculations show that for $\eps\in(0,1]$,
\[
	\dfrac{\rmd}{\rmd t}|n_\nu|_q+\Big(\dfrac{d(\gamma_\nu-1)}{2}-\dfrac{d}{q}\Big)\dfrac{1}{1+t}|n_\nu|_q\leq C|n_\nu|_q\Big(|\dive w_\nu|_\infty+\dfrac{1}{(1+t)^2}\Big),
\]
which, along with \eqref{embed00}, yields \eqref{mid88}.
\end{proof}

We then establish the high order uniform estimates of solutions.

\begin{lemma}\label{chp4-3} It holds
\begin{align}
\begin{split}
			\dfrac{\rmd}{\rmd t}\dot{X}_{\nu,1}+\dfrac{c_{d,\gamma_\nu,1,\nu}}{1+t}\dot{X}_{\nu,1}\leq& C\dot{X}_{\nu,1}^{\frac{4s-d-4}{2s-2}}\dot{X}_{\nu,s}^{\frac{d}{2s-2}}+\frac{C}{(1+t)^2}\dot{X}_{\nu,1}\\
			\label{estx2}&+C(1+t)^{\frac{d}{2}-3}|W_\nu|_\infty+C \sum_{\nu=i,e}|n_\nu|_q^{\frac{q}{2}}|n_\nu|_\infty^{\frac{2}{\gamma_\nu-1}-\frac{q}{2}},\\
\end{split}
\end{align}
as well as
\begin{align}
		\label{estxs} &\dfrac{\rmd}{\rmd t}\dot{X}_{\nu,s}+\dfrac{c_{d,\gamma_\nu,s,\nu}}{1+t}\dot{X}_{\nu,s}\\
		&\leq C\dot{X}_{\nu,1}^{\frac{2s-d-2}{2s-2}}\dot{X}_{\nu,s}^{\frac{2s-2+d}{2s-2}}+\frac{C}{(1+t)^2}\dot{X}_{\nu,s}+C(1+t)^{\frac{d}{2}-s-2}|W_\nu|_\infty\nonumber\\
		&\quad +C\sum_{\nu=i,e}|n_\nu|_\infty^{\frac{3-\gamma_\nu}{\gamma_\nu-1}}\dot{X}_{\nu,1}^{\frac{1}{s-1}}\dot{X}_{\nu,s}^{\frac{s-2}{s-1}}+C(1+t)^{\frac{d}{2}-s-1}\dot{X}_{\nu,1}^{\frac{2s-d-2}{2s-2}}\dot{X}_{\nu,s}^{\frac{d}{2s-2}}\nonumber\\
		&\quad +C(1+t)^{-3}\dot{X}_{\nu,1}^{\frac{1}{s-1}}\dot{X}_{\nu,s}^{\frac{s-2}{s-1}}.\nonumber
\end{align}
\end{lemma}
\begin{proof}
System \eqref{TEPeps}$_1$--\eqref{TEPeps}$_4$ can be written as
\begin{equation}\label{symmglobal}
		\pt W_\nu+ \displaystyle\sum_{j=1}^d A_\nu^j(W_\nu;\hu_\nu,\eps)\pa_{x_j}W_\nu+\frac{H_\nu(W_\nu;\eps)}{1+t}=Q_\nu^1(\D\vp)-\frac{Q_\nu^3(W_\nu;\eps)}{(1+t)^2},
\end{equation}
where $A_\nu^j(W_\nu;\hu_\nu,\eps), Q_\nu^1(\D\vp)$ are defined in \eqref{defQ} and
\begin{equation}\label{defH}
\begin{split}
			H_\nu(W_\nu,\eps)=\left(\begin{matrix} \dfrac{d(\gamma_\nu-1)}{2}n_\nu\\
				\eps^{-\beta_\nu}w_\nu \end{matrix}\right), \quad Q_\nu^3(W_\nu;\eps)=\left(\begin{matrix}\dfrac{\gamma_\nu-1}{2} n_\nu\text{Tr}K_\nu\\
				\eps^{-\beta_\nu} w_\nu\cdot K_\nu\end{matrix}\right).
\end{split}
\end{equation}

For $l=1$ or $l=s$, we justify the energy calculation by spatial mollification and testing the differentiated equation \eqref{symmglobal} with $F^L\D^lW_\nu$, where $F^L=F(|x|/L)$ and $F$ is defined in \eqref{bump}. After removing the mollification for fixed $L$, all boundary terms vanish as $L\to\infty$, since the bound for $\hu_\nu$ in \S\ref{s4.1} and $W_\nu\in L^\infty([0,T];H^s)$ give
\[
\begin{aligned}
		&\int_0^T\int_{\R^d}|\D F^L|(|\hu_\nu|+|W_\nu|)|\D^lW_\nu|^2\rmd x\rmd t\\
		&\quad\leq C\int_0^T\left(\frac{1+|u_\nu^0(0)|/L}{1+t}+\frac{|W_\nu(t)|_\infty}{L}\right) \int_{|x|\geq L}|\D^lW_\nu|^2\rmd x\rmd t\longrightarrow0,
\end{aligned}
\]
where $C$ is independent of $T$ and $L$. Thus, in the sense of distributions in time, one obtains
\begin{align}\label{gstart1}
		&\dfrac{1}{2}\dfrac{\rmd}{\rmd t}|\D^l W_\nu|_2^2+\dfrac{b_\nu}{1+t}|\D^l W_\nu|_2^2\notag\\
		\leq& \dfrac{1}{2}\sum_{j=1}^d\int_{\R^d}(\D^l W_\nu)^\top\pa_{x_j}A_\nu^j(W_\nu;\hu_\nu,\eps)\D^l W_\nu\rmd x+q_\nu\eps^{\beta_\nu}\int_{\R^d}(\D^l\D\vp)^\top\D^l (\eps^{-\beta_\nu}w_\nu)\rmd x\\
		&-\int_{\R^d}(\D^l W_\nu)^\top\Big(\D^l\big(\sum_{j=1}^dA_\nu^j(W_\nu;\hu_\nu,\eps)\pa_{x_j} W_\nu\big)-\sum_{j=1}^dA_\nu^j(W_\nu;\hu_\nu,\eps)\D^l\pa_{x_j}W_\nu\Big)\rmd x\notag\\
		&-\dfrac{1}{(1+t)^2}\int_{\R^d}(\D^l W_\nu)^\top \D^l(Q_\nu^3(W_\nu;\eps))\rmd x:=\sum_{j=1}^4I_{\nu,j}^l,\notag
\end{align}
with the natural correspondence of $\{I_{\nu,j}^l\}_{j=1}^4$.

{\cuti{Estimates for $I_{\nu,1}^l$.}} By \eqref{nablau}, one has
\begin{equation}\label{I1}
		|I_{\nu,1}^l|\leq C|\D W_\nu|_\infty\dot{X}_{\nu,l}^2+\dfrac{d}{2}\dfrac{1}{1+t}|\D^lW_\nu|_2^2+\dfrac{C}{(1+t)^2}|\D^lW_\nu|_2^2.
\end{equation}

{\cuti{Estimates for $I_{\nu,2}^l$.}} We first treat the potential term. For the case $l=1$,
\begin{equation}\label{X1vp}
		|\D (\D\vp)|_2\leq C\sum_{\nu=i,e}|n_\nu^{2/(\gamma_\nu-1)}|_2\leq C\sum_{\nu=i,e}|n_\nu|_q^{\frac{q}{2}}|n_\nu|_\infty^{\frac{2}{\gamma_\nu-1}-\frac{q}{2}}.
\end{equation}
For the case $l=s$, one has similarly to \eqref{Hs},
\begin{equation}\label{Xsvp}
		|\D^s (\D\vp)|_2\leq C\sum_{\nu=i,e}\|n_\nu^{2/(\gamma_\nu-1)}\|_{\dot{H}^{s-1}}\leq C\sum_{\nu=i,e}|n_\nu|_\infty^{\frac{3-\gamma_\nu}{\gamma_\nu-1}}\dot{X}_{\nu,1}^{\frac{1}{s-1}}\dot{X}_{\nu,s}^{\frac{s-2}{s-1}}.
\end{equation}
Now the estimates for $I_{\nu,2}^l$ are as follows:
\begin{equation}\label{I2}
\begin{split}
			|I_{\nu,2}^1|	\leq& C\dot{X}_{\nu,1} \sum_{\tilde\nu=i,e}|n_{\tilde\nu}|_q^{\frac{q}{2}}|n_{\tilde\nu}|_\infty^{\frac{2}{\gamma_{\tilde\nu}-1}-\frac{q}{2}},\\
			|I_{\nu,2}^s|	\leq& C\dot{X}_{\nu,s}\sum_{\tilde \nu=i,e}|n_{\tilde\nu}|_\infty^{\frac{3-\gamma_{\tilde\nu}}{\gamma_{\tilde\nu}-1}}\dot{X}_{\tilde\nu,1}^{\frac{1}{s-1}}\dot{X}_{\tilde\nu,s}^{\frac{s-2}{s-1}}.
\end{split}
\end{equation}

{\cuti{Estimates for $I_{\nu,3}^l$.}} Let
\[
	B_\nu^j(W_\nu):=\left( \begin{matrix} w_\nu^{(j)}&\eps^{\beta_\nu}\dfrac{\gamma_\nu-1}{2}n_\nu\xi_j^\top\\
		\eps^{\beta_\nu}\dfrac{\gamma_\nu-1}{2}n_\nu\xi_j&w_\nu^{(j)}\Id \end{matrix} \right),
\]
with $w_\nu=(w_\nu^{(1)},\cdots, w_\nu^{(d)})^\top$. Define
\[
	J_\nu^l:=\D^l\Big(\sum_{j=1}^d A_\nu^j(W_\nu;\hu_\nu,\eps)\pa_{x_j}W_\nu\Big)-\sum_{j=1}^d A_\nu^j(W_\nu;\hu_\nu,\eps)\D^l\pa_{x_j}W_\nu.
\]
For $l\geq 1$, we perform the following calculations componentwise using
\[
	\partial_x^\alpha := \frac{\partial^l} {\partial x_1^{\alpha_1}\cdots\partial x_d^{\alpha_d}}, \qquad |\alpha|=l,
\]
where $\al=(\al_1,\cdots,\al_d)$ is a multi-index of nonnegative integers. Noticing the expansion for $\hu_\nu$ in Lemma \ref{hu}, one has
\[
	J_\nu^l=F_{\nu,1}^l+F_{\nu,2}^l+G_\nu^l,
\]
in which $\{F_{\nu,j}^l\}_{j=1}^2$ and $G_\nu^l$ are defined as
\begin{align*}
		F_{\nu,1}^l:=&\D^l\Big(\sum_{j=1}^d B_\nu^j(W_\nu)\pa_{x_j}W_\nu\Big)-\sum_{j=1}^d B_\nu^j(W_\nu)\D^l\pa_{x_j}W_\nu,\\
		F_{\nu,2}^l:=&\sum_{j=1}^d\sum_{\substack{1\leq i\leq d\\
				\al_i\neq 0}}\al_i\pa_{x_i}(\hu_\nu^{(j)}\mathbb{I}_{d+1})\pa_x^{\al^i}\pa_{x_j} W_\nu,\\
		G_\nu^l:=&J_\nu^l-\sum_{j=1}^2F_{\nu,j}^l,
\end{align*}
where $\al^i$ is a multi-index satisfying $\pa_{x_i}\pa_x^{\al^i}=\pa_x^\al$. By Lemma \ref{comm2}, one has
\begin{align}\label{F1}
\begin{split}
			\Big|\int_{\R^d}(\D^l W_\nu)^\top F_{\nu,1}^l\rmd x\Big|\leq & C|\D^l W_\nu|_2\Big(\|W_\nu\|_{\dot{H}^l}|\D W_\nu|_\infty+|\D W_\nu|_\infty\|\D W_\nu\|_{\dot{H}^{l-1}}\Big)\\
			\leq & C|\D W_\nu|_\infty\dot{X}_{\nu,l}^2.
\end{split}
\end{align}
Next, one obtains from the expansion for $\hu_\nu$ in Lemma \ref{hu} that
\[
	F_{\nu,2}^l= \frac{l}{1+t}\D^l W_\nu+\sum_{\substack{1\leq i\leq d\\
			\al_i\neq 0}}\frac{\al_i (K_\nu)_{i,i}(t,x)}{(1+t)^2}\D^l W_\nu+\sum_{\substack{1\leq i\leq d\\
			\al_i\neq 0}}\sum_{\substack{1\leq j\leq d\\
			j\neq i}}\al_i\frac{(K_\nu)_{i,j}(t,x)}{(1+t)^2}\pa_x^{\al^i}\pa_{x_j}W_\nu,
\]
where $(K_\nu)_{i,j}(t,x)$ is the element of $K_\nu(t,x)$ in the $i$-th row and $j$-th column. Hence
\begin{equation}\label{F2}
		\int_{\R^d}(\D^l W_\nu)^\top F_{\nu,2}^l\rmd x\geq \dfrac{l}{1+t}|\D^l W_\nu|_2^2-\dfrac{CM}{(1+t)^2}|\D^l W_\nu|_2^2.
\end{equation}
Since $G_\nu^1=0$, one concludes from estimates \eqref{F1}--\eqref{F2} for the case $l=1$ that
\begin{equation}\label{I31}
		I_{\nu,3}^1+\dfrac{1}{1+t}\dot{X}_{\nu,1}^2\leq C\Big(|\D W_\nu|_\infty+\dfrac{M}{(1+t)^2}\Big)\dot{X}_{\nu,1}^2.
\end{equation}

It remains to estimate $G_\nu^s$, which consists of products of the form $\D^k\hu_\nu\D^{s+1-k}W_\nu$ for $2\leq k\leq s$. The homogeneous product estimate applied to $\D^2\hu_\nu$ and $\D W_\nu$, together with Lemma \ref{hu}, \eqref{embed00} and interpolation, yields
\begin{equation}\label{Gs}
\begin{aligned}
			|G_\nu^s|_2\leq& C|\D^s\hu_\nu|_2|\D W_\nu|_\infty+C|\D^2\hu_\nu|_\infty|\D^{s-1}W_\nu|_2\\
			\leq& C(1+t)^{\frac{d}{2}-s-1}\dot{X}_{\nu,1}^{\frac{2s-d-2}{2s-2}}\dot{X}_{\nu,s}^{\frac{d}{2s-2}}+C(1+t)^{-3}\dot{X}_{\nu,1}^{\frac{1}{s-1}}\dot{X}_{\nu,s}^{\frac{s-2}{s-1}}.
\end{aligned}
\end{equation}
Consequently, for the case $l=s$, estimates \eqref{F1}--\eqref{F2} and \eqref{Gs} imply that
\begin{equation}\label{I3s}
\begin{aligned}
			I_{\nu,3}^s+\dfrac{s}{1+t}\dot{X}_{\nu,s}^2\leq& C\Big(|\D W_\nu|_\infty+\dfrac{M}{(1+t)^2}\Big)\dot{X}_{\nu,s}^2+C(1+t)^{-3}\dot{X}_{\nu,1}^{\frac{1}{s-1}}\dot{X}_{\nu,s}^{\frac{s-2}{s-1}+1}\\
			&+C(1+t)^{\frac{d}{2}-s-1}\dot{X}_{\nu,1}^{\frac{2s-d-2}{2s-2}}\dot{X}_{\nu,s}^{\frac{d}{2s-2}+1}.
\end{aligned}
\end{equation}

{\cuti{Estimates for $I_{\nu,4}^l$.}} For $l\geq 1$, one obtains by Lemma \ref{comm1},
\begin{align}\label{I4l}
\begin{split}
			|I_{\nu,4}^l|\leq& C(1+t)^{-2}|\D^l W_\nu|_2(\| W_\nu\|_{\dot{H}^l}|K_\nu|_\infty+C|W_\nu|_\infty\|K_\nu\|_{\dot{H}^{l}})\\
			\leq& \dfrac{M}{(1+t)^2}\dot{X}_{\nu,l}^2+C(1+t)^{\frac{d}{2}-l-2}|W_\nu|_\infty\dot{X}_{\nu,l}.
\end{split}
\end{align}
Substituting estimates \eqref{I1}, \eqref{I2}, \eqref{I31}, \eqref{I3s} and \eqref{I4l} into \eqref{gstart1}, one   obtains the desired estimates  \eqref{estx2} and \eqref{estxs}.
\end{proof}

For simplicity, we introduce the following notations:
\begin{equation*}
\begin{split}
		(n_\nu^{\infty},w_\nu^\infty)(t)=&(1+t)^{-1}(|n_\nu|_\infty, |w_\nu|_\infty),\,\, n_\nu^q(t)=(1+t)^{-\frac{d}{q}-1}|n_\nu|_q,\,\, \\
		\dot{Y}_{\nu,\sigma}(t)=&(1+t)^{\sigma-\frac{d}{2}-1}\dot{X}_{\nu,\sigma}(t),\quad 1\leq \sigma\leq s.
\end{split}
\end{equation*}
Let us also set $a_\nu=1+b_\nu\in(1,2]$ and
\[
\tilde{Y}_{\nu,s}=n_\nu^q+\dot{Y}_{\nu,1}+\dot{Y}_{\nu,s}, \quad Y_{\nu,s}=\tilde{Y}_{\nu,s}+n_\nu^\infty+\eps^{-\beta_\nu}w_\nu^\infty.
\]
We then have the following lemma regarding the estimates for $Y_{\nu,s}$.
\begin{lemma}\label{global-smallness} There exists a constant $\delta>0$, which is independent of $(T,\eps,R_\nu)$, such that if $Y_{i,s}(0)+Y_{e,s}(0)\leq \delta$, it holds
\begin{equation}\label{resultnaka}
		Y_{\nu,s}(t)\leq \dfrac{CY_{i,s}(0)+CY_{e,s}(0)}{(1+t)^{1+\min\{1,\frac{d(\gamma_\nu-1)}{2}\}}},\quad t\in[0,T].
\end{equation}
\end{lemma}
\begin{proof} 

From Lemmas \ref{chp4-1}--\ref{chp4-3}, one concludes there exists a constant $c_{32}>0$, such that
\begin{align}
\dfrac{\rmd}{\rmd t}\tilde{Y}_{\nu,s}+\dfrac{a_\nu}{1+t}\tilde{Y}_{\nu,s}\leq& \dfrac{c_{32}\tilde{Y}_{\nu,s}}{(1+t)^2}+\dfrac{C(n_\nu^\infty+\eps^{-\beta_\nu}w_\nu^\infty)}{(1+t)^2}+CY_{\nu,s}^2\nonumber\\
\label{ys1} &+C\sum_{{\tilde\nu}=i,e}(1+t)^{\frac{2}{\gamma_{\tilde\nu}-1}}\tilde{Y}_{{\tilde\nu},s}^{\frac{2}{\gamma_{\tilde\nu}-1}}\\
&+C\sum_{{\tilde\nu}=i,e}(1+t)^{\frac{2}{\gamma_{\tilde\nu}-1}}(n_{\tilde\nu}^\infty)^{\frac{3-\gamma_{\tilde\nu}}{\gamma_{\tilde\nu}-1}}\tilde{Y}_{{\tilde\nu},s},\nonumber\\
\label{n1}n_\nu^\infty(t)\leq& \frac{e^{\frac{c_{31}^\nu t}{1+t}}|n_\nu^0|_\infty}{(1+t)^{a_\nu}}+\frac{C}{(1+t)^{a_\nu}}\int_0^t (1+\tau)^{a_\nu}Y_{\nu,s}^2(\tau)\rmd\tau,\\
\label{w1}\eps^{-\beta_\nu}w_\nu^\infty(t)\leq&\frac{C}{(1+t)^{a_\nu}}\int_0^t\Big((1+\tau)^{a_\nu}\tilde{\mathcal{I}}(\tau)+(1+\tau)^{a_\nu}Y_{\nu,s}^2(\tau)\Big)\rmd\tau,
\end{align}
where
\[
\tilde{\mathcal{I}}(\tau)=\Big(\sum_{{\tilde\nu}=i,e}((1+\tau)Y_{{\tilde\nu},s})^{\frac{2}{\gamma_{\tilde\nu}-1}\frac{2s-d}{2s-2}}\Big)\Big(\sum_{{\tilde\nu}=i,e}((1+\tau)Y_{{\tilde\nu},s})^{\frac{2}{\gamma_{\tilde\nu}-1}\frac{d-2}{2s-2}}\Big).
\]
Now let $c_3:=\max\{c_{31}^i,c_{31}^e,c_{32}\}$ and $Z(t):=\sum_{\nu=i,e}e^{-\frac{c_3t}{1+t}}(1+t)^{a_\nu}Y_{\nu,s}(t)$. By the time continuity established in \S\ref{s4.1}, $Z\in C([0,T])$. Suppose that
\begin{equation*}
			Z(t)\leq c_4Z(0),\qquad \text{for}\,\, t\in[0,T],
\end{equation*}
with $c_4\geq 1$ to be determined in \eqref{c4}. Without loss of generality, we assume at present  $Z(0)<1$. Then \eqref{n1} implies that
\begin{align} \label{n3}
\begin{split}
				e^{-\frac{c_3t}{1+t}}(1+t)^{a_\nu} n_\nu^\infty(t)\leq& C|n_\nu^0|_\infty+C\int_0^t\frac{1}{(1+\tau)^{a_\nu}}Z^2(\tau)\rmd\tau,\\
				\leq& CZ(0)+ Cc_4^2Z^2(0),
\end{split}
\end{align}
since $a_\nu>1$. Next, \eqref{w1} implies that
\begin{align}
\begin{split}
				\eps^{-\beta_\nu}e^{-\frac{c_3t}{1+t}}(1+t)^{a_\nu} w_\nu^\infty(t)\leq& C\int_0^t \Big((1+\tau)^{(1-\bar{a})\frac{2}{\bar{\gamma}-1}+\bar{a}}c_4^{\frac{2}{\underline{\gamma}-1}}Z(0)^{\frac{2}{\bar{\gamma}-1}}\\
				\label{w2}&+ \frac{1}{(1+\tau)^{a_\nu}}c_4^2(Z(0))^2\Big)\rmd\tau,
\end{split}
\end{align}
where $\bar{a}=\max\{a_i,a_e\}$. Notice $\bar{a}=\min\{\frac{d(\bar{\gamma}-1)}{2}+1,2\}$. Since $\gamma_i,\gamma_e<\frac{5}{3}$, one obtains
\[
	\frac{2}{\bar{\gamma}-1}> 3, \quad d(1-\frac{\bar\gamma-1}{2})>2,
\]
which implies
\begin{equation}\label{agamma}
			(1-\bar{a})\frac{2}{\bar{\gamma}-1}+\bar{a}<-1.
\end{equation}
Combining $a_\nu>1$ and \eqref{agamma} with \eqref{w2} yields
\begin{equation}\label{w3}
			\eps^{-\beta_\nu}e^{-\frac{c_3t}{1+t}}(1+t)^{a_\nu} w_\nu^\infty(t)\leq Cc_4^{\frac{2}{\underline{\gamma}-1}}Z^{\frac{2}{\bar{\gamma}-1}}(0)+ Cc_4^2Z^2(0).
\end{equation}

In addition, substituting estimates \eqref{n3} and \eqref{w3} into \eqref{ys1} yields
\begin{equation}\label{naka4-5}
\begin{split}
	&\dfrac{\rmd}{\rmd t}((1+t)^{a_\nu}e^{-\frac{c_3t}{1+t}}\tilde{Y}_{\nu,s})\\
	\leq&\frac{C(1+t)^{a_\nu}e^{-\frac{c_3t}{1+t}}(n_\nu^\infty+\eps^{-\beta_\nu}w_\nu^\infty)}{(1+t)^2}+\frac{Cc_4^2Z^2(0)}{(1+t)^{a_\nu}}\\
	&+C(1+t)^{(1-\bar{a})\frac{2}{\bar{\gamma}-1}+\bar{a}}c_4^{\frac{2}{\underline{\gamma}-1}}Z^{\frac{2}{\bar{\gamma}-1}}(0)\\
	\leq&\frac{CZ(0)+Cc_4^{\frac{2}{\underline{\gamma}-1}}Z^{\frac{2}{\bar{\gamma}-1}}(0)+ Cc_4^2Z^2(0)}{(1+t)^2}+\frac{Cc_4^2Z^2(0)}{(1+t)^{a_\nu}}\\
	&+C(1+t)^{(1-\bar{a})\frac{2}{\bar{\gamma}-1}+\bar{a}}c_4^{\frac{2}{\underline{\gamma}-1}}Z^{\frac{2}{\bar{\gamma}-1}}(0).
\end{split}
\end{equation}
Integrating \eqref{naka4-5} over $[0,t]$ and using $a_\nu>1$ and \eqref{agamma}, one obtains, for some constant $c_5>0$ independent of $T$,
\begin{equation}\label{ys3}
			(1+t)^{a_\nu}e^{-\frac{c_3t}{1+t}}\tilde{Y}_{\nu,s}(t)\leq c_5Z(0)+Cc_4^2Z^2(0)+Cc_4^{\frac{2}{\underline{\gamma}-1}}Z^{\frac{2}{\bar{\gamma}-1}}(0).
\end{equation}
Combining \eqref{n3}, \eqref{w3} and \eqref{ys3} yields that there exists a constant $c_6>0$ such that
\[
		Z(t)\leq c_6Z(0)+Cc_4^2Z^2(0)+Cc_4^{\frac{2}{\underline{\gamma}-1}}Z^{\frac{2}{\bar{\gamma}-1}}(0).
\]
Consequently, choose $c_4$ large enough and then $Z(0)$ small enough, such that
\begin{equation}\label{c4}
			c_6<\frac{c_4}{12}, \qquad \Big(Cc_4+Cc_4^{\frac{3-\underline{\gamma}}{\underline{\gamma}-1}}Z^{\frac{4-2\bar{\gamma}}{\bar{\gamma}-1}}(0)\Big)Z(0)<\frac{1}{3}.
\end{equation}
Then $Z(t)$ is well-defined for $t\in[0,T]$ satisfying $Z(t)<\frac{2}{3}c_4Z(0)$. This implies \eqref{resultnaka}.
\end{proof}

By \eqref{resultnaka} and the preceding Poisson estimates, one obtains \eqref{uni1} and the $L^\infty$ and $\dot H^\sigma$ estimates in \eqref{univp2}. By Lemma~\ref{ga-ni} and \eqref{uni1}, one obtains
\begin{equation}
\begin{aligned}\label{nablaw}
|\D(n_\nu,\eps^{-\beta_\nu}w_\nu)(t)|_\infty
\leq& C\|(n_\nu,\eps^{-\beta_\nu}w_\nu)(t)\|_{\dot H^1}^{\frac{2s-d-2}{2s-2}}
\|(n_\nu,\eps^{-\beta_\nu}w_\nu)(t)\|_{\dot H^s}^{\frac{d}{2s-2}}\notag\\
\leq& C(1+t)^{-1-\min\{1,\frac{d(\gamma_\nu-1)}{2}\}}.
\end{aligned}
\end{equation}
Finally, by \eqref{res1}, one obtains $p=qd/(d+q)\in(1,d)$ and $\frac{2}{\gamma_\nu-1}>\frac{d+q}{d}$. Applying Lemma~\ref{HLS} to the Poisson equation and using \eqref{uni1}, one obtains the $L^q$ estimate in \eqref{univp2}:
\begin{equation}
\begin{split}
|\D\vp|_q\leq C\sum_{\nu=i,e}|n_\nu^{\frac{2}{\gamma_\nu-1}}|_{\frac{qd}{d+q}}\leq C\sum_{\nu=i,e}|n_\nu|_\infty^{\frac{2}{\gamma_\nu-1}-\frac{d+q}{d}}|n_\nu|_q^{\frac{d+q}{d}}\leq  C(1+t)^{\frac{d+q}{q}-(\bar{b})\frac{2}{\bar{\gamma}-1}}.
\end{split}
\end{equation}

For later use, we derive the following $L^2$ estimate for $\eps^{-\beta_\nu}w_\nu$.
\begin{lemma}\label{velocity-L2-lemma}
On every finite existence interval $[0,T]$, $\eps^{-\beta_\nu}w_\nu\in L^\infty([0,T];L^2)$ with a bound depending only on $T$, the initial support radii and the fixed data constants.
\end{lemma}

\begin{proof}
By the density characteristics, the bounds in \S\ref{s4.1} and \eqref{uni1}, and Gr\"{o}nwall's inequality, one obtains a support bound depending only on $T$, the initial support radii and the fixed data constants. By Lemma~\ref{Poisson} and \eqref{X1vp}--\eqref{Xsvp}, one obtains $\D\vp\in L^\infty([0,T];H^s)$.

The spatial localization in \S\ref{s4.1} justifies integration by parts. Testing \eqref{TEPeps}$_2$ and \eqref{TEPeps}$_4$ with $w_\nu$, using \eqref{nablaw} and $0<\eps\leq1$, and applying Young's inequality, one obtains
\[
\frac{\rmd}{\rmd t}|\eps^{-\beta_\nu}w_\nu|_2^2\leq C|\eps^{-\beta_\nu}w_\nu|_2^2+C|n_\nu|_\infty^2|\D n_\nu|_2^2+C|\D\vp|_2^2.
\]
By \eqref{uni1}, the electric-field $L^2$ bound from Lemma~\ref{Poisson} and $w_\nu(0)=0$, one obtains the claimed estimate from Gr\"{o}nwall's inequality.
\end{proof}

\subsection{Uniqueness}\label{s4.3}
Let $(n_{\nu,j},w_{\nu,j},\vp_j)$, $j=1,2$, be solutions with the same initial data and the $H^s$ regularity of Theorem~\ref{thm4.1} on $[0,T]$. Here $C$ may also depend on $T$ and their norms. We first prove equality of their gradients, then exclude spatial constants by finite $L^p$ integrability. Set
\[
\tn_\nu:=n_{\nu,1}-n_{\nu,2}, \,\,\tw_\nu:=w_{\nu,1}-w_{\nu,2}, \,\, \tvp:=\vp_1-\vp_2.
\]
Set $F^k=F(|x|/k)$, $\tn_\nu^k=F^k\tn_\nu$, $\tw_\nu^k=F^k\tw_\nu$, and $\tilde E^k=F^k\D\tvp=:(\tilde{E}_1^k,\cdots,\tilde{E}_d^k)^\top$. On $\supp\D F^k\subset\{k\leq|x|\leq2k\}$, the factor $k^{-1}$ compensates for the linear growth of $\hu_\nu$ in the localized integration by parts.
For the ion differences, one obtains
\begin{equation}\label{chauniqi}
\begin{cases} \pt \tn_i^k+(w_{i,1}+\hu_i)\cdot\nabla \tn_i^k+\dfrac{\gamma_i-1}{2}n_{i,1}\dive \tw_i^k\\
			\quad=-\dfrac{\gamma_i-1}{2}\tn_i^k\frac{d}{1+t} -\dfrac{\gamma_i-1}{2}\tn_i^k\dfrac{\text{Tr} K_i}{(1+t)^2}+(w_{i,1}+\hu_i)\tn_i\cdot\D F^k\\
			\quad\quad\,-\tw_i^k\cdot\D n_{i,2}+\dfrac{\gamma_i-1}{2}n_{i,1}\tilde w_i\cdot\D F^k-\dfrac{\gamma_i-1}{2}\tn_{i}^k\dive w_{i,2}, \\
			\eps^{-2} \pt \tw_i^k+\eps^{-2}((w_{i,1}+\hu_i)\cdot\nabla)\tw_i^k+\dfrac{\gamma_i-1}{2}n_{i,1}\D\tn_i^k\\
			\quad=-\dfrac{\eps^{-2}\tw_i^k}{1+t}-\dfrac{\eps^{-2}\tw_i^k\cdot K_i}{(1+t)^2}+\tilde E^k+\eps^{-2}((w_{i,1}+\hu_i)\cdot\nabla)F^k\tw_i\\
			\quad\quad\,-\eps^{-2}(\tw_i^k\cdot\D)w_{i,2}+\dfrac{\gamma_i-1}{2}(n_{i,1}\tn_i\D F^k-\tn_i^k\D n_{i,2}),\\
			(\tn_i^k,\eps^{-1}\tw_i^k)(0,x)=(0,0), \end{cases}
\end{equation}
For the electron differences, one obtains
\begin{equation}\label{chauniqe}
\begin{cases} \pt \tn_e^k+(w_{e,1}+\hu_e)\cdot\nabla \tn_e^k+\dfrac{\gamma_e-1}{2}n_{e,1}\dive \tw_e^k\\
			\quad=-\dfrac{\gamma_e-1}{2}\tn_e^k\Big(\dfrac{d}{1+t} +\dfrac{\text{Tr} K_e}{(1+t)^2}\Big)+\tn_e(w_{e,1}+\hu_e)\cdot\D F^k\\
			\quad\quad\,+\dfrac{\gamma_e-1}{2}n_{e,1}\tw_e\cdot\D F^k-\tw_e^k\cdot\D n_{e,2}-\dfrac{\gamma_e-1}{2}\tn_{e}^k\dive w_{e,2}, \\

			\pt \tw_e^k+((w_{e,1}+\hu_e)\cdot\nabla)\tw_e^k+\dfrac{\gamma_e-1}{2}n_{e,1}\D\tn_e^k\\
			\quad=-\dfrac{\tw_e^k}{1+t}-\dfrac{\tw_e^k\cdot K_e}{(1+t)^2}-\tilde E^k+((w_{e,1}+\hu_e)\cdot\nabla)F^k\tw_e\\

			\quad\quad\,-(\tw_e^k\cdot\D)w_{e,2}+\dfrac{\gamma_e-1}{2}n_{e,1}\tn_e\D F^k-\dfrac{\gamma_e-1}{2}\tn_e^k\D n_{e,2},\\
			\dive \tilde{E}^k-\D\tvp\cdot\D F^k=\displaystyle\sum_{\nu=i,e}q_\nu C_\nu\dfrac{2}{\gamma_\nu-1}\tn_\nu^k\displaystyle \int_0^1 (n_{\nu,2}+\theta\tn_\nu)^{\frac{2}{\gamma_\nu-1}-1}\rmd\theta,\\
			(\tn_e^k,\tw_e^k)(0,x)=(0,0). \end{cases}
\end{equation}
Denote $\mathcal{W}_{\nu,j}=(n_{\nu,j},\eps^{-\beta_\nu}w_{\nu,j}^\top)^\top\,(j=1,2)$ and $\mathbb{W}_\nu=(\tn_\nu^k,(\eps^{-\beta_\nu}\tw_\nu^k)^\top)^\top$. Then \eqref{chauniqi} and \eqref{chauniqe}$_{1}$--\eqref{chauniqe}$_{2}$ can be rewritten into
\begin{equation}\label{uniquestart}
		\pt\mathbb{W}_\nu+\sum_{j=1}^d A_\nu^j(\mathcal{W}_{\nu,1};\hu_\nu,\eps)\pa_{x_j}\mathbb{W}_\nu+\frac{H_\nu(\mathbb{W}_\nu;\eps)}{1+t}+\frac{Q_\nu^3(\mathbb{W}_\nu;\eps)}{(1+t)^2}=Q_\nu^1(\tilde E^k)+Q_\nu^4+Q_\nu^5,
\end{equation}
Here $A_\nu^j$ and $Q_\nu^1$ are defined in \eqref{defQ}, $H_\nu$ and $Q_\nu^3$ in \eqref{defH}, and
\begin{align*}
		Q_\nu^4=&-\left(\begin{matrix}\tw_\nu^k\cdot\D n_{\nu,2}+\dfrac{\gamma_\nu-1}{2}\tn_\nu^k\dive w_{\nu,2}\\[3mm]
			\eps^{-\beta_\nu}(\tw_\nu^k\cdot\D)w_{\nu,2}+\dfrac{\gamma_\nu-1}{2}\eps^{\beta_\nu}\tn_\nu^k\D n_{\nu,2}\end{matrix}\right),\\
		Q_\nu^5=&\left(\begin{matrix}\tn_\nu (w_{\nu,1}+\hu_\nu)\cdot\D F^k+\dfrac{\gamma_\nu-1}{2}n_{\nu,1}\tw_\nu\cdot\D F^k\\[3mm]
			\eps^{-\beta_\nu}(((w_{\nu,1}+\hu_\nu)\cdot\D F^k)\tw_\nu)+\eps^{\beta_\nu}\dfrac{\gamma_\nu-1}{2}\tn_\nu n_{\nu,1}\D F^k\end{matrix}\right).
\end{align*}
By Sobolev embedding, one obtains $\D\tvp\in L^\infty([0,T];L^{2d/(d-2)})$ and
\[
	|\D\mathbb{W}_\nu|_2\leq |\D F^k|_d|(\tn_\nu,\eps^{-\beta_\nu}\tw_\nu)|_{\frac{2d}{d-2}}+|F^k|_\infty|\D(\tn_\nu,\eps^{-\beta_\nu}\tw_\nu)|_2\leq C|\D(\tn_\nu,\eps^{-\beta_{\nu}}\tw_\nu)|_2.
\]

Applying $\D^1$ to \eqref{uniquestart}, testing against $\D^1\mathbb{W}_\nu$ and integrating over $\R^d$, one obtains
\begin{align*}
		\dfrac{1}{2}\dfrac{\rmd}{\rmd t}|\D^1 \mathbb{W}_\nu|_2^2=& \dfrac{1}{2}\sum_{j=1}^d\int_{\R^d}(\D^1\mathbb{W}_\nu)^\top\pa_{x_j} A_\nu^j(\mathcal{W}_{\nu,1};\hu_\nu,\eps)\D^1\mathbb{W}_\nu\rmd x+q_\nu\int_{\R^d}\D^1\tilde E^k\cdot\D^1\tw_\nu^k\rmd x\nonumber\\
		&-\int_{\R^d}(\D^1\mathbb{W}_\nu)^\top\Big(\D^1\frac{H_\nu(\mathbb{W}_\nu;\eps)}{1+t}+\frac{\D^1 Q_\nu^3(\mathbb{W}_\nu;\eps)}{(1+t)^2}\Big)\rmd x\\
		&+\int_{\R^d}(\D^1\mathbb{W}_\nu)^\top\D^1 Q_\nu^4\rmd x+\int_{\R^d}(\D^1\mathbb{W}_\nu)^\top\D^1 Q_\nu^5\rmd x\\
		&-\sum_{j=1}^d\int_{\R^d}(\D^1\mathbb{W}_\nu)^\top\D^1 A_\nu^j(\mathcal{W}_{\nu,1};\hu_\nu,\eps)\pa_{x_j}\mathbb{W}_\nu\rmd x :=\sum_{j=1}^6\hat{I}_{\nu,j},\nonumber
\end{align*}
with the natural correspondence of $\{\hat I_{\nu,j}\}_{j=1}^6$. By \eqref{res1} and the density $L^q\cap L^\infty$ bounds, one obtains $|n_{\nu,j}^{\frac{2}{\gamma_\nu-1}-1}|_d\leq C$. Using the Sobolev inequality, one then obtains
\begin{align*}
		|\hat{I}_{\nu,1}|\leq& C(|\D\mathcal{W}_{\nu,1}|_\infty+|\D\hu_\nu|_\infty)|\D\mathbb{W}_\nu|_2^2,\notag\\
|\hat{I}_{\nu,2}|\leq& C|\D\mathbb{W}_\nu|_2\Big(\|\D\tvp\|_{L^{\frac{2d}{d-2}}(\R^d\backslash B_k)}\notag\\
&\quad+\sum_{\tilde\nu=i,e}|\tn_{\tilde\nu}^k|_{\frac{2d}{d-2}}\sum_{j=1}^2|n_{\tilde\nu,j}^{\frac{2}{\gamma_{\tilde\nu}-1}-1}|_d\Big)\notag\\
\leq& C\sum_{\tilde\nu=i,e}|\D\mathbb{W}_{\tilde\nu}|_2^2+C|\D\mathbb{W}_\nu|_2\|\D\tvp\|_{L^{\frac{2d}{d-2}}(\R^d\backslash B_k)},\\
		|\hat{I}_{\nu,3}|\leq&  C|\D\mathbb{W}_\nu|_2^2+C|\D\mathbb{W}_\nu|_2(|\D\mathbb{W}_\nu|_2|K_\nu|_\infty+|\mathbb{W}_\nu|_{\frac{2d}{d-2}}|\D K_\nu|_d)\leq C|\D\mathbb{W}_\nu|_2^2,\notag\\
		|\hat{I}_{\nu,4}|\leq& C|\D\mathbb{W}_\nu|_2(|\D^2\mathcal{W}_{\nu,2}|_d|\mathbb{W}_\nu|_{\frac{2d}{d-2}}+|\D\mathbb{W}_\nu|_2|\D\mathcal{W}_{\nu,2}|_\infty)\leq  C|\D\mathbb{W}_\nu|_2^2,\notag\\
		|\hat{I}_{\nu,6}|\leq& C(|\D\mathcal{W}_{\nu,1}|_\infty+|\D\hu_\nu|_\infty)|\D\mathbb{W}_\nu|_2^2.\notag
\end{align*}
For $\hat{I}_{\nu,5}$, we first bound the $\dot{H}^1$ norm of $Q_\nu^5$. Notice that
\begin{align*}
		\|\tn_\nu w_{\nu,1}\D F^k\|_{\dot{H}^1} \leq& \frac{C}{k}\bigl(|w_{\nu,1}|_\infty+|\D w_{\nu,1}|_d\bigr) \|\D\tn_\nu\|_{L^2(\R^d\backslash B_k)}\\
		\leq& \frac{C}{k}\|\D\tn_\nu\|_{L^2(\R^d\backslash B_k)}.
\end{align*}
Similarly, one obtains
\[
\begin{aligned}
		&\|n_{\nu,1} \tw_\nu\cdot\D F^k\|_{\dot{H}^1}+\|\eps^{-\beta_\nu}(w_{\nu,1}\cdot\D) F^k\tw_\nu\|_{\dot{H}^1}+\|\tn_\nu n_{\nu,1}\D F^k\|_{\dot{H}^1}\\
		&\qquad\leq \frac{C}{k}\left(\|\D\tn_\nu\|_{L^2(\R^d\backslash B_k)}+\eps^{-\beta_\nu}\|\D\tw_\nu\|_{L^2(\R^d\backslash B_k)}\right).
\end{aligned}
\]
For the term containing $\hu_\nu$, the bound in \S\ref{s4.1} gives
\begin{align}
		\|\tn_\nu \hu_\nu\cdot\D F^k\|_{\dot{H}^1}\leq C\|\D\tn_\nu\|_{L^2(\R^d\backslash B_k)}+\|\tn_\nu\|_{L^\frac{2d}{d-2}(\R^d\backslash B_k)}\leq C\|\D\tn_\nu\|_{L^2(\R^d\backslash B_k)}, \nonumber
\end{align}
in which one has used the fact that for $k\leq |x|\leq 2k$,
\begin{align*}
		|\hu_\nu \D F^k|_\infty\leq&C|(1+|x|)\D F^k|_\infty\leq C,\\
		|\hu_\nu\D^2 F^k|_d\leq& C|(1+|x|)\D^2 F^k|_d\leq C.
\end{align*}
Similarly one obtains
\[
	\|\eps^{-\beta_\nu}(\hu_\nu\cdot\D) F^k\tw_\nu\|_{\dot{H}^1}\leq C\|\eps^{-\beta_\nu}\D\tw_\nu\|_{L^2(\R^d\backslash B_k)}.
\]
Hence, one obtains
\[
		|\hat{I}_{\nu,5}|\leq C|\D\mathbb{W}_\nu|_2(\|\D\tn_\nu\|_{L^2(\R^d\backslash B_k)}+\eps^{-\beta_\nu}\|\D\tw_\nu\|_{L^2(\R^d\backslash B_k)}).
\]
Combining the preceding estimates, one obtains
\begin{align*}
		\dfrac{\rmd}{\rmd t}|\D\mathbb{W}_\nu|_2^2\leq& C\sum_{\tilde\nu=i,e}|\D\mathbb{W}_{\tilde\nu}|_2^2+C|\D\mathbb{W}_\nu|_2\|\D\tvp\|_{L^\frac{2d}{d-2}(\R^d\backslash B_k)}\nonumber\\
		&+C|\D\mathbb{W}_\nu|_2(\|\D\tn_\nu\|_{L^2(\R^d\backslash B_k)}+\eps^{-\beta_\nu}\|\D\tw_\nu\|_{L^2(\R^d\backslash B_k)}),
\end{align*}
which, along with Gr\"{o}nwall's inequality, implies for $\Y(t)=\sum_{\nu=i,e}|\D\mathbb{W}_\nu(t)|_2^2$,
\begin{equation*}
\begin{split}
			\Y(t) \leq& Ce^{CT}\int_0^Te^{-c\tau}\Y(\tau)^\frac12\Big(\|\D\tvp\|_{L^\frac{2d}{d-2}(\R^d\backslash B_k)}\\
			&+\sum_{\tilde{\nu}=i,e}\big(\|\D\tn_{\tilde{\nu}}\|_{L^2(\R^d\backslash B_k)}+\eps^{-\beta_{\tilde\nu}}\|\D\tw_{\tilde\nu}\|_{L^2(\R^d\backslash B_k)}\big)\Big)\rmd\tau \to 0, \qquad \text{as}\,\,\, k\to\infty.
\end{split}
\end{equation*}
Consequently, $\D\tn_\nu=\D\tw_\nu=0$. Since $\tn_\nu\in L^q$ and $\eps^{-\beta_\nu}\tw_\nu\in L^2$, it follows that $\tn_\nu=\tw_\nu=0$. The Poisson equation and $\D\tvp\in L^2$ then give $\D\tvp=0$. \hfill $\square$

\subsection{Global existence and time continuity}\label{s4.4}
Fix $T>0$ and take $T^0=T+1$ in Remark~\ref{local-restart}. On any existence interval $[0,t_0]\subseteq[0,T]$, \eqref{uni1}, Lemma~\ref{velocity-L2-lemma}, the density support bound and the bound on $\hu_\nu$ in \S\ref{s4.1} bound $\sum_{\nu=i,e}\|W_\nu(t_0)\|_s$, the density support radius $\mathcal R(t_0)$ and
\[
\max_{\nu=i,e}\sup_{\substack{t_0\leq t\leq T+1\\ |x|\leq\mathcal R(t_0)+1}}|\hu_\nu(t,x)|
\]
by a constant depending only on $T$, the initial support radii $R_\nu$ and the fixed data constants. Lemma~\ref{hu} also bounds $\sup_{0\leq t\leq T+1}\|\hu_\nu(t)\|_\Xi$. Remark~\ref{local-restart} therefore gives an existence time $\tau_T\in(0,1]$ independent of $t_0$. Finitely many restarts reach $T$, and the pieces agree on overlaps by \S\ref{s4.3}. Uniqueness follows from \S\ref{s4.3}. Using \S\ref{s4.1}, the Poisson continuity estimates \eqref{midd}--\eqref{mid} and \eqref{sys-symm}, one obtains the time regularity stated in Theorem~\ref{thm4.1}. This completes its proof. \hfill $\square$

%%%%%%%%%%%%%%%%%%%%%%%%%%%%%%%%%%%%%%%%%%%%%%%%%%%%%%%%----S5
\section{Uniform global  estimates and convergence analysis for general data}

In this section, we will first extend the global-in-time well-posedness in Theorem \ref{thm4.1} to the more general case in which $n_\nu^0$ are still small, but are not necessarily compactly supported. Then, based on the global existence of regular solutions, we study the global-in-time convergence in the VMR limit for system \eqref{sys-symm}. In the rest of this section, we denote by $C\geq 1$ a generic constant depending only on fixed constants $(A_\nu,\gamma_\nu,\kappa,s,d,q,\allowbreak\|u_i^0\|_\Xi,\|u_e^0\|_\Xi,\allowbreak|u_i^0(0)|,|u_e^0(0)|)$, but independent of $T$, $\eps$ and any truncation parameters, which may differ from line to line.

In the following subsections, we will prove Theorems \ref{globalthm}--\ref{thm2.2}.

\subsection{Initial approximation}
For initial data $n_\nu^0\in \Gamma\cap L^q$ and $\mk\geq 1$, set $n_{\nu,0}^\mk=n_\nu^0 F^\mk$. For each fixed $\mk$, these data belong to $H^s$ and have compact support. Before constructing global solutions, we verify uniform smallness in $\Gamma\cap L^q$. Since $0\leq F^\mk\leq1$, the $L^\infty$ and $L^q$ norms do not increase. Also, $n_\nu^0\in L^q$ fixes the representative in the homogeneous Sobolev embedding, so that $|n_\nu^0|_{2d/(d-2)}\leq C|\D n_\nu^0|_2$. Thus
\begin{align*}
		\|n_{\nu,0}^\mk\|_{\dot{H}^1}=&|\D n_\nu^0 F^\mk+ n_\nu^0 \D F^\mk|_2\leq C|\D n_\nu^0|_2+|n_\nu^0|_{\frac{2d}{d-2}}|\D F^\mk|_d\leq C|\D n_\nu^0|_2. \nonumber
\end{align*}
For the $\dot{H}^s$ norm, it holds
\[
	\D^s(n_\nu^0 F^\mk)=\sum_{l=1}^s m_{sl}\D^l n_\nu^0\D^{s-l} F^\mk+n_\nu^0\D^s F^\mk,
\]
in which $m_{sl}>0$ are binomial constants. Then,
\begin{align}
		|\D^l n_\nu^0\D^{s-l} F^\mk|_2\leq& C\mk^{l-s}|\D^l n_\nu^0|_2,\qquad 1\leq l\leq s,\nonumber\\
		|n_\nu^0\D^s F^\mk|_2\leq& C|n_\nu^0|_{\frac{2d}{d-2}} |\D^s F^\mk|_d\leq C\mk^{1-s}|\D n_\nu^0|_2.\nonumber
\end{align}
Homogeneous interpolation gives $|\D^l n_\nu^0|_2\leq C(|\D n_\nu^0|_2+|\D^s n_\nu^0|_2)$ for $1\leq l\leq s$. Consequently, there is a constant $C_{\rm cut}\geq1$, depending only on $F,d,s$, such that
\[
\sup_{\mk\geq1}\|n_{\nu,0}^\mk\|_{\Gamma\cap L^q}
\leq C_{\rm cut}\|n_\nu^0\|_{\Gamma\cap L^q}.
\]
Since $w_{\nu,0}^\mk=0$, we choose the constants $B_\nu^0$ in Theorem~\ref{globalthm} sufficiently small that
\begin{align*}
\sup_{\mk\geq1}\sum_{\nu=i,e}Y_{\nu,s}^\mk(0)
=\sup_{\mk\geq1}\sum_{\nu=i,e}\|n_{\nu,0}^\mk\|_{\Gamma\cap L^q}\leq C_{\rm cut}\sum_{\nu=i,e}\sqrt{\frac{4A_\nu\gamma_\nu}{(\gamma_\nu-1)^2}}B_\nu^0\leq\delta,
\end{align*}
where $\delta$ is defined in Lemma~\ref{global-smallness}. This choice is independent of $\mk$ and $\eps$. The compact-support construction in the proof of Theorem~\ref{thm4.1}, with this common threshold, therefore yields global approximate solutions $(n_\nu^{\mk,\eps},\eps^{-\beta_\nu}w_\nu^{\mk,\eps},\vp^{\mk,\eps})$ to \eqref{TEPeps} satisfying \eqref{regu4.1} with
\begin{equation}\label{initialk}
		(n_\nu^{\mk,\eps},\eps^{-\beta_\nu}w_\nu^{\mk,\eps})(0,x)=(n_\nu^0 F^\mk,0).
\end{equation}

\subsection{Global existence, uniqueness and time continuity}
\subsubsection{Passing to  the limit \texorpdfstring{$\mk\to\infty$}{}}
We now pass to the limit  $\mk\to\infty$ in the equations in \eqref{TEPeps} with the initial data \eqref{initialk}. The equations in \eqref{TEPeps} do not contain the parameter $\mk$ and the initial data $n_{\nu,0}^\mk$ is uniformly bounded with respect to $\mk$ in $\Gamma\cap L^q$. Hence, based on the uniform estimates \eqref{uni1} and \eqref{univp2}, in which the generic constant $C$ is independent of the size of the support of $n_{\nu,0}^\mk$, one concludes that the solution sequence $\{n_\nu^{\mk,\eps}\}_{\mk\geq 1}$ is bounded in $ L^\infty([0,T];\Gamma\cap L^q)$, and $\{(\eps^{-\beta_\nu}w_\nu^{\mk,\eps},\D\vp^{\mk,\eps})\}_{\mk\geq 1}$ is bounded in $L^\infty([0,T];\Gamma)$. Consequently, there exists a subsequence of solutions (still denoted by)  $(n_\nu^{\mk,\eps}, \eps^{-\beta_\nu}w_\nu^{\mk,\eps},\vp^{\mk,\eps})$ that converges to a limit $(n_\nu^\eps,\eps^{-\beta_\nu}w_\nu^\eps,\vp^\eps)$ in the following weak-* sense:
\begin{equation}\label{convkwe}
\begin{split}
			n_\nu^{\mk,\eps}\rightharpoonup n_\nu^\eps\quad&\text{weakly*\,\,in}\,\,\, L^\infty([0,T];\Gamma\cap L^q),\\
			(\eps^{-\beta_\nu}w_\nu^{\mk,\eps},\D\vp^{\mk,\eps})\rightharpoonup (\eps^{-\beta_\nu}w_\nu^\eps,\D\vp^{\eps})\quad &\text{weakly*\,\,in}\,\,\, L^\infty([0,T];\Gamma).
\end{split}
\end{equation}
Besides, the weak lower semicontinuity of Sobolev norms implies that $(n_\nu^\eps,\eps^{-\beta_\nu}w_\nu^\eps,\vp^\eps)$ still satisfies \eqref{uni1}, \eqref{univp2} and \eqref{nablaw}.

Next, system \eqref{TEPeps} implies for any finite constant $R_0>0$ and finite time $T>0$,
\begin{equation*}
		\|\pt n_\nu^{\mk,\eps}(t)\|_{H^{s-1}(B_{R_0})}^2+\|\eps^{-\beta_\nu}\pt w_\nu^{\mk,\eps}(t)\|_{H^{s-1}(B_{R_0})}^2\leq C_0(R_0,T),
\end{equation*}
for $0\leq t\leq T$, where the constant $C_0(R_0,T)>0$ depends only on $C$, $R_0$ and $T$. Consequently, up to a further subsequence, one has
\begin{equation}\label{convkst}
		(n_\nu^{\mk,\eps},\eps^{-\beta_\nu}w_\nu^{\mk,\eps})\to (n_\nu^\eps,\eps^{-\beta_\nu}w_\nu^\eps) \quad \text{strongly\,\,in}\,\,\,\,\, C([0,T];H^{s-1}(B_{R_0})).
\end{equation}
Hence, combining the strong convergence in \eqref{convkst}, the weak convergence in \eqref{convkwe} and \eqref{uni1} and \eqref{nablaw} shows that $(n_\nu^\eps,\eps^{-\beta_\nu}w_\nu^\eps,\vp^\eps)$ is a regular solution to \eqref{TEPeps}.

\subsubsection{Uniqueness}\label{s5.2.2} Let $(n_{\nu,1},w_{\nu,1},\vp_1)$ and $(n_{\nu,2},w_{\nu,2},\vp_2)$ be two regular solutions to \eqref{TEPeps}.

We first establish the $L^q$ integrability of the velocity perturbations for these two regular solutions. The Poisson equation and the far-field condition \eqref{ofar} identify $\D\vp_j$ with the Newtonian field. By Lemma \ref{HLS}, with $p=qd/(d+q)\in(1,d)$, and \eqref{res1}, for $j=1,2$,
\begin{align*}
		|\D\vp_j(t)|_q \leq C\sum_{\nu=i,e}|n_{\nu,j}(t)^{\frac{2}{\gamma_\nu-1}}|_{\frac{qd}{d+q}}\leq C\sum_{\nu=i,e} |n_{\nu,j}(t)|_\infty^{\frac{2}{\gamma_\nu-1}-1-q/d} |n_{\nu,j}(t)|_q^{1+q/d}.
\end{align*}
Thus $\D\vp_j\in L^\infty([0,T];L^q)$. Set $u_{\nu,j}:=w_{\nu,j}+\hu_\nu$. These velocities are globally Lipschitz in space and have at most linear growth on $[0,T]$. Using $w_{\nu,j}(0)=0$, integration of the velocity equations along the characteristics of $u_{\nu,j}$, followed by the change of variables under the characteristic flow, gives
\begin{align*}
		|\eps^{-\beta_\nu}w_{\nu,j}(t)|_q \leq \eps^{\beta_\nu}\int_0^t e^{C_T(t-\tau)} \Bigl(|\D\vp_j(\tau)|_q+\frac{\gamma_\nu-1}{2} |n_{\nu,j}(\tau)|_q|\D n_{\nu,j}(\tau)|_\infty\Bigr)\rmd\tau.
\end{align*}
Here $C_T$ controls $|\D\hu_\nu|_\infty+q^{-1}|\dive u_{\nu,j}|_\infty$ on $[0,T]$. Consequently,
\[
	\eps^{-\beta_\nu}w_{\nu,j}\in L^\infty([0,T];L^q), \qquad j=1,2.
\]
In this uniqueness argument, constants may also depend on $T$ and the norms of the two regular solutions on $[0,T]$.

Let
\[
	\tn_\nu:=n_{\nu,1}-n_{\nu,2}, \,\, \tw_\nu:=w_{\nu,1}-w_{\nu,2}, \,\, \tvp:=\vp_1-\vp_2.
\]
Set $F^k=F(|x|/k)$, and
\[
	\tn_\nu^k=F^k\tn_\nu, \quad \tw_\nu^k=F^k\tw_\nu,  \quad \tilde{E}^k=F^k\D\tvp.
\]
Then $(\tn_\nu^k,\tw_\nu^k,\tilde E^k)$ satisfies \eqref{chauniqi}--\eqref{chauniqe}. Denote $\mathcal{W}_{\nu,j}=(n_{\nu,j},\eps^{-\beta_\nu}w_{\nu,j}^\top)^\top\,(j=1,2)$ and $\mathbb{W}_\nu=(\tn_\nu^k,\eps^{-\beta_{\nu}}(\tw_\nu^k)^\top)^\top$. Then $\mathbb{W}_\nu$ satisfies \eqref{uniquestart},  $\D\mathbb{W}_\nu\in L^\infty([0,T];L^2)$ and
\[
	|\D\mathbb{W}_\nu|_2\leq |\D F^k|_d|(\tn_\nu,\eps^{-\beta_\nu}\tw_\nu)|_{\frac{2d}{d-2}}+|F^k|_\infty|\D(\tn_\nu,\eps^{-\beta_\nu}\tw_\nu)|_2\leq C|\D(\tn_\nu,\eps^{-\beta_{\nu}}\tw_\nu)|_2.
\]

Applying $\D^1$ to \eqref{uniquestart}, multiplying by $\D^1 \mathbb{W}_\nu$ and integrating over $\R^d$ yield 
\begin{align*}
		\dfrac{1}{2}\dfrac{\rmd}{\rmd t}|\D^1 \mathbb{W}_\nu|_2^2=& \dfrac{1}{2}\sum_{j=1}^d\int_{\R^d}(\D^1\mathbb{W}_\nu)^\top\pa_{x_j} A_\nu^j(\mathcal{W}_{\nu,1};\hu_\nu,\eps)\D^1\mathbb{W}_\nu\rmd x+q_\nu\int_{\R^d}\D^1\tilde E^k\cdot\D^1\tw_\nu^k\rmd x\nonumber\\
		&-\int_{\R^d}(\D^1\mathbb{W}_\nu)^\top\Big(\D^1\frac{H_\nu(\mathbb{W}_\nu;\eps)}{1+t}+\frac{\D^1 Q_\nu^3(\mathbb{W}_\nu;\eps)}{(1+t)^2}\Big)\rmd x\\
		&+\int_{\R^d}(\D^1\mathbb{W}_\nu)^\top\D^1 Q_\nu^4\rmd x+\int_{\R^d}(\D^1\mathbb{W}_\nu)^\top\D^1 Q_\nu^5\rmd x\\
		&-\sum_{j=1}^d\int_{\R^d}(\D^1\mathbb{W}_\nu)^\top\D^1 A_\nu^j(\mathcal{W}_{\nu,1};\hu_\nu,\eps)\pa_{x_j}\mathbb{W}_\nu\rmd x :=\sum_{j=1}^6\hat{I}_{\nu,j},\nonumber
\end{align*}
with the natural correspondence of $\{\hat{I}_{\nu,j}\}_{j=1}^6$, of which the estimates are as follows:
\begin{align}\label{mid2}
		|\hat{I}_{\nu,1}|\leq& C(|\D\mathcal{W}_{\nu,1}|_\infty+|\D\hu_\nu|_\infty)|\D\mathbb{W}_\nu|_2^2,\notag\\
		|\hat{I}_{\nu,2}|\leq& \underbrace{C|\D\mathbb{W}_\nu|_2\|\D\tvp\|_{L^{\frac{2d}{d-2}}(\R^d\backslash B_k)}}_{:=I_{\nu,k}^1}+C|\D\mathbb{W}_\nu|_2\sum_{\tilde{\nu}=i,e}\Big|\tn_{\tilde{\nu}}^k\int_0^1 (n_{\tilde{\nu},2}+\theta\tn_{\tilde{\nu}})^{\frac{2}{\gamma_{\tilde{\nu}}-1}-1}\rmd\theta\Big|_2\notag\\
		\leq& C|\D\mathbb{W}_\nu|_2\sum_{\tilde{\nu}=i,e}|\tn_{\tilde{\nu}}^k|_{\frac{2d}{d-2}}(|n_{\tilde{\nu},1}^{\frac{2}{\gamma_{\tilde{\nu}}-1}-1}|_d+|n_{\tilde{\nu},2}^{\frac{2}{\gamma_{\tilde{\nu}}-1}-1}|_d)+I_{\nu,k}^1\notag\\
		\leq& C|\D\mathbb{W}_\nu|_2\sum_{\tilde{\nu}=i,e}|\D\tn_{\tilde{\nu}}^k|_2\sum_{j=1}^2|n_{\tilde{\nu},j}|_q^{\frac{q}{d}}|n_{\tilde{\nu},j}|_\infty^{\frac{3-\gamma_{\tilde{\nu}}}{\gamma_{\tilde{\nu}}-1}-\frac{q}{d}}+I_{\nu,k}^1\leq  C\sum_{\tilde{\nu}=i,e}|\D\mathbb{W}_{\tilde\nu}|_2^2+I_{\nu,k}^1,\\
		|\hat{I}_{\nu,3}|\leq&  C|\D\mathbb{W}_\nu|_2^2+C|\D\mathbb{W}_\nu|_2(|\D\mathbb{W}_\nu|_2|K_\nu|_\infty+|\mathbb{W}_\nu|_{\frac{2d}{d-2}}|\D K_\nu|_d)\leq C|\D\mathbb{W}_\nu|_2^2,\notag\\
		|\hat{I}_{\nu,4}|\leq& C|\D\mathbb{W}_\nu|_2(|\D^2\mathcal{W}_{\nu,2}|_d|\mathbb{W}_\nu|_{\frac{2d}{d-2}}+|\D\mathbb{W}_\nu|_2|\D\mathcal{W}_{\nu,2}|_\infty)\leq  C|\D\mathbb{W}_\nu|_2^2,\notag\\
		|\hat{I}_{\nu,6}|\leq& C(|\D\mathcal{W}_{\nu,1}|_\infty+|\D\hu_\nu|_\infty)|\D\mathbb{W}_\nu|_2^2.\notag
\end{align}
For $\hat{I}_{\nu,5}$, we first bound the $\dot{H}^1$ norm of $Q_\nu^5$. Notice that
\begin{align*}
		\|\tn_\nu w_{\nu,1}\D F^k\|_{\dot{H}^1} \leq& \frac{C}{k}\bigl(|w_{\nu,1}|_\infty+|\D w_{\nu,1}|_d\bigr) \|\D\tn_\nu\|_{L^2(\R^d\backslash B_k)}\\
		\leq& \frac{C}{k}\|\D\tn_\nu\|_{L^2(\R^d\backslash B_k)}.
\end{align*}
Similarly, one obtains
\[
\begin{aligned}
		&\|n_{\nu,1} \tw_\nu\cdot\D F^k\|_{\dot{H}^1}+\|\eps^{-\beta_\nu}(w_{\nu,1}\cdot\D) F^k\tw_\nu\|_{\dot{H}^1}+\|\tn_\nu n_{\nu,1}\D F^k\|_{\dot{H}^1}\\
		&\qquad\leq \frac{C}{k}\left(\|\D\tn_\nu\|_{L^2(\R^d\backslash B_k)}+\eps^{-\beta_\nu}\|\D\tw_\nu\|_{L^2(\R^d\backslash B_k)}\right).
\end{aligned}
\]
For the term containing $\hu_\nu$, the bound in \S\ref{s4.1} gives
\begin{align}
		\|\tn_\nu \hu_\nu\cdot\D F^k\|_{\dot{H}^1}\leq C\|\D\tn_\nu\|_{L^2(\R^d\backslash B_k)}+\|\tn_\nu\|_{L^\frac{2d}{d-2}(\R^d\backslash B_k)}\leq C\|\D\tn_\nu\|_{L^2(\R^d\backslash B_k)}, \nonumber
\end{align}
in which one has used the fact that for $k\leq |x|\leq 2k$, 
\begin{align*}
		|\hu_\nu \D F^k|_\infty\leq&C|(1+|x|)\D F^k|_\infty\leq C,\\
		|\hu_\nu\D^2 F^k|_d\leq& C|(1+|x|)\D^2 F^k|_d\leq C.
\end{align*}
Similarly one obtains
\[
	\|\eps^{-\beta_\nu}(\hu_\nu\cdot\D) F^k\tw_\nu\|_{\dot{H}^1}\leq C\|\eps^{-\beta_\nu}\D\tw_\nu\|_{L^2(\R^d\backslash B_k)}.
\]
Hence, one concludes
\begin{equation}\label{mid4}
		|\hat{I}_{\nu,5}|\leq C|\D\mathbb{W}_\nu|_2(\|\D\tn_\nu\|_{L^2(\R^d\backslash B_k)}+\eps^{-\beta_\nu}\|\D\tw_\nu\|_{L^2(\R^d\backslash B_k)}).
\end{equation}
Combining \eqref{mid2}--\eqref{mid4}, one has
\begin{align*}
		\dfrac{\rmd}{\rmd t}|\D\mathbb{W}_\nu|_2^2\leq& C\sum_{\tilde\nu=i,e}|\D\mathbb{W}_{\tilde\nu}|_2^2+C|\D\mathbb{W}_\nu|_2\|\D\tvp\|_{L^\frac{2d}{d-2}(\R^d\backslash B_k)}\nonumber\\
		&+C|\D\mathbb{W}_\nu|_2(\|\D\tn_\nu\|_{L^2(\R^d\backslash B_k)}+\eps^{-\beta_\nu}\|\D\tw_\nu\|_{L^2(\R^d\backslash B_k)}),
\end{align*}
which, along with Gr\"{o}nwall's inequality, implies for $\Y(t)=\sum_{\nu=i,e}|\D\mathbb{W}_\nu(t)|_2^2$, 
\begin{equation*}
\begin{split}
			\Y(t) \leq& Ce^{CT}\int_0^Te^{-c\tau}\Y(\tau)^\frac12\Big(\|\D\tvp\|_{L^\frac{2d}{d-2}(\R^d\backslash B_k)}\\
			&+\sum_{\tilde{\nu}=i,e}\big(\|\D\tn_{\tilde{\nu}}\|_{L^2(\R^d\backslash B_k)}+\eps^{-\beta_{\tilde\nu}}\|\D\tw_{\tilde\nu}\|_{L^2(\R^d\backslash B_k)}\big)\Big)\rmd\tau \to 0, \qquad \text{as}\,\,\, k\to\infty.
\end{split}
\end{equation*}
Consequently, $\D\tn_\nu=\D\tw_\nu\equiv 0$, which implies $\tilde n_\nu$ and $\eps^{-\beta_\nu}\tilde w_\nu$ depend only on time over $[0,T]\times\R^d$. Since $(\tilde n_\nu,\eps^{-\beta_\nu}\tilde w_\nu)\in L^q$, $\tn_\nu=\tw_\nu\equiv 0$. Hence the uniqueness is obtained. \hfill $\square$

\subsubsection{Time continuity}\label{s5.2.3} From \eqref{uni1}, \eqref{univp2} and \eqref{nablaw} and \eqref{TEPeps}, one obtains that
\[
	(n_\nu^\eps,\eps^{-\beta_\nu}w_\nu^\eps,\D\vp^\eps)\in L^\infty([0,T];\Gamma), \qquad (\pt n_\nu^\eps,\eps^{-\beta_\nu}\pt w_\nu^\eps)\in L^\infty([0,T];H_{\text{loc}}^{s-1}).
\]
Besides, one also has $\pt\D\vp^\eps\in L^\infty([0,T];H_{\text{loc}}^{s-1})$. Indeed, we first consider the existing approximations with compactly supported densities. Fix $R\geq1$. Smoothly localizing between $B_{2R}$ and $B_{4R}$ and using the $H^{s-1}$ boundedness of $\D\Delta^{-1}\dive$ near $B_R$, the decay of its differentiated kernel away from $B_R$, and $|u_\nu(t,x)|\leq C_T(1+|x|)$, one obtains
\begin{align*}
		\|\pt\D\vp(t)\|_{H^{s-1}(B_R)} &\leq C_{T,R}\sum_{\nu=i,e}\left( \|\rho_\nu u_\nu(t)\|_{H^{s-1}(B_{4R})} +\int_{|y|>2R}\frac{\rho_\nu(t,y)}{|y|^{d-1}}\rmd y \right)\\
		&\leq C_{T,R}\left(1+\sum_{\nu=i,e} |\rho_\nu(t)|_{\frac{qd}{d+q}}\right) \leq C_{T,R}.
\end{align*}
Here the local term is controlled by the existing product estimates, while H\"older's inequality applies to the integral since $1<qd/(d+q)<d$. The density norms are uniformly bounded by $\rho_\nu=C_\nu n_\nu^{2/(\gamma_\nu-1)}$, \eqref{res1}, and the uniform $L^q\cap L^\infty$ bounds for $n_\nu$. Since the constant is independent of $\mk,\eps$, weak-* compactness and \eqref{convkwe} allow passage to the limit in distributions, yielding $\pt\D\vp^\eps\in L^\infty(0,T;H^{s-1}_{\rm loc})$. Combining these time-derivative bounds with the local $H^s$ bounds from \eqref{uni1}, \eqref{univp2} and \eqref{nablaw}, interpolation yields $	(n_\nu^\eps,\eps^{-\beta_\nu}w_\nu^\eps,\D\vp^\eps)\in C([0,T];H_{\rm loc}^{s'})$ for  $0\leq s'<s$. Hence the proof is completed. \hfill $\square$

\subsection{Global convergence analysis for the VMR limit} We are now ready to prove Theorem \ref{thm2.2}. Suppose the following initial convergences hold
\begin{align*}
		n_\nu^0 \rightharpoonup& \bar{n}_\nu^0 \quad \text{weakly in}\,\,\, L^q\cap \Gamma,\nonumber\\
		u_\nu^0 \rightharpoonup& \bar{u}_\nu^0 \quad \text{weakly in}\,\,\, \Xi.\nonumber
\end{align*}
Based on the uniform estimates \eqref{uni1}, \eqref{univp2} and \eqref{nablaw}, one concludes that the solution sequence $\{n_\nu^\eps\}_{\eps>0} $ is bounded in $L^\infty([0,T];\Gamma\cap L^q)$ and the sequence $\{(\eps^{-\beta_\nu}w_\nu^\eps,\D\vp^\eps)\}_{\eps>0}$ is bounded in $L^\infty([0,T];\Gamma)$. Hence, there exist functions $\bar{n}_\nu$, $\bar{w}_e$ and $\bar{\vp}$ with $\bar{n}_\nu\in L^\infty([0,T];\Gamma\cap L^q)$ and $(\bar{w}_e,\D\bar{\vp})\in L^\infty([0,T];\Gamma)$ such that as $\eps\to 0$, up to subsequences
\begin{align}\label{wee001}
\begin{split}
			n_\nu^\eps \rightharpoonup& \bar{n}_\nu \quad \text{weakly-* in} \,\,\,L^\infty([0,T];\Gamma\cap L^q),\\
			(w_e^\eps,\D\vp^\eps)\rightharpoonup& (\bar{w}_e,\D\bar{\vp}) \quad \text{weakly-* in} \,\,\,L^\infty([0,T];\Gamma),\\
			w_i^\eps\to& 0 \quad\,\,\,\, \text{strongly\,\,\,in} \,\,\,L^\infty([0,T];\Gamma).
\end{split}
\end{align}
The initial convergence convention and the uniform gradient bounds give uniform local bounds for the initial velocities. By Lemmas \ref{aubin} and \ref{hu} and \eqref{burgers}, up to a subsequence,
\[
	\hat{u}_\nu^\varepsilon \longrightarrow \bar{\hat{u}}_\nu \quad\text{strongly in } C([0,T];H^{s-1}_{\mathrm{loc}}(\mathbb{R}^d)),
\]
for every finite $T>0$, where $\bar{\hat{u}}_\nu$ solves \eqref{burgers} with initial data $\bar{u}_\nu^0$. Since $s-1>d/2$, this also gives uniform convergence at the spatial origin on $[0,T]$. Together with the uniform derivative bounds, this yields
\begin{equation}\label{wee02}
		\hu_\nu\rightharpoonup \bar{\hu}_\nu \quad \text{weakly-* in} \,\,\,L^\infty([0,T];\Xi).
\end{equation}
Let
\begin{equation*}
		\Xi^s=\{f\in L^1_{\text{loc}}(\mathbb{R}^d): |\nabla f|_\infty+\|\nabla^2 f\|_{s-2}<\infty\}.
\end{equation*}
Then $\bar{w}_e,\bar{\hu}_\nu\in\Xi^s$. For simplicity, let $\bar{w}_i:=0$ and  $\bar{u}_\nu=\bar{w}_\nu+\bar{\hu}_\nu.$ It is easy to check that $\bar{u}_\nu\in\Xi^s$.

Now we study the strong convergence of solutions. Notice that it is clear that $u_i^\eps$ converges to $\bar{\hu}_i$, which is the solution to Burgers equations with initial data $\bar u_i^0$. Next, by \eqref{TEPeps}, $\{\pt n_i^\eps\}_{\eps>0}$ is bounded in $L^\infty([0,T];H_{{\rm{loc}}}^{s-1})$, which implies that $\{n_i^\eps\}_{\eps>0}$ is relatively compact in $C([0,T];H_{{\rm{loc}}}^{s-1})$. Consequently, passing to the limit $\eps\to 0$ in \eqref{TEPeps}$_3$ in the sense of distributions, one obtains that $\bar{n}_i$ satisfies
\begin{equation}\label{ion-n-limit}
		\pt\bar{n}_i+\bar{\hu}_i\cdot\D\bar{n}_i+\frac{\gamma_i-1}{2}\bar{n}_i\dive\bar{\hu}_i=0,\quad \bar{n}_i(0,x)=\bar{n}_i^0,
\end{equation}
which shows that the ion system converges to the following problem
\begin{equation*}
\begin{cases} \pt\bar{n}_i+\bar{\hu}_i\cdot\D\bar{n}_i+\frac{\gamma_i-1}{2}\bar{n}_i\dive\bar{\hu}_i=0,\\
			\pt \bar{\hu}_i+(\bar{\hu}_i\cdot\D)\bar{\hu}_i=0,\\
			t=0:(\bar{n}_i,\bar{\hu}_i)=(\bar{n}_i^0,\bar{u}_i^0)(x). \end{cases}
\end{equation*}

We then pass to the limit in the electron system together with the Poisson equation. From system \eqref{TEPeps} and \eqref{burgers}, it follows that for any $T>0$,
\[(\pt n_e, \pt w_e,\pt\hu_e)\in L^\infty([0,T];H_{\text{loc}}^{s-1}),\]
which implies $\{n_e,w_e,\hu_e\}_{\eps>0}$ are relatively compact in $C([0,T];H_{\text{loc}}^{s-1})$, which, along with the weak convergences \eqref{wee001} and \eqref{wee02}, implies that $(\bar{n}_e,\bar{w}_e,\bar{\vp})$ satisfies
\begin{equation}\label{electron-limit}
\begin{cases} \pt\bar{n}_e+(\bar{w}_e+\bar{\hu}_e)\cdot\D\bar{n}_e+\dfrac{\gamma_e-1}{2}\bar{n}_e\dive\bar{w}_e=-\dfrac{\gamma_e-1}{2}\bar{n}_e\dive\bar{\hu}_e,\\[2mm]
			\pt\bar{w}_e+((\bar{w}_e+\bar{\hu}_e)\cdot\D)\bar{w}_e+\dfrac{\gamma_e-1}{2}\bar{n}_e\cdot\D\bar{n}_e=-\D\bar{\vp}-(\bar{w}_e\cdot\D)\bar{\hu}_e,\\
			\Delta\bar{\vp}=C_i (\bar{n}_i(t,x))^{2/(\gamma_i-1)}-C_e \bar{n}_e^{2/(\gamma_e-1)},\\
			(\bar{n}_e,\bar{w}_e)(0,x)=(\bar{n}_e^0(x),0), \end{cases}
\end{equation}
which is a decoupled system for $(\bar{n}_e,\bar{w}_e,\bar{\vp})$ since $\bar{n}_i$ can be solved through \eqref{ion-n-limit}. Besides, the limit system for the electron system in \eqref{burgers} is
\begin{equation}\label{zs00}
		\pt\bar{\hu}_e+(\bar{\hu}_e\cdot\D)\bar{\hu}_e=0, \qquad \bar{\hu}_e(0,x)=\bar{u}_e^0(x).
\end{equation}
It follows from  \eqref{electron-limit} and the equation in \eqref{zs00} that $(\bar{n}_e,\bar{u}_e,\bar{\vp})$ satisfies
\begin{equation*}
\begin{cases} \pt \bar{n}_e+\bar{u}_e\cdot\nabla \bar{n}_e+\dfrac{\gamma_e-1}{2}\bar{n}_e\dive \bar{u}_e=0,\\
			\pt\bar{u}_e+(\bar{u}_e\cdot\D)\bar{u}_e+\dfrac{\gamma_e-1}{2}\bar{n}_e\D\bar{n}_e=-\D\bar\vp,\\
			\Delta\bar{\vp}=C_i (\bar{n}_i(t,x))^{2/(\gamma_i-1)}-C_e \bar{n}_e^{2/(\gamma_e-1)},\\
			(\bar{n}_e,\bar{u}_e)(0,x)=(\bar{n}_e^0,\bar{u}_e^0)(x), \end{cases}
\end{equation*}
which recovers system \eqref{Euler0} by defining $ \bar{\rho}_e=\Big(\bar{n}_e\sqrt{\frac{(\gamma_e-1)^2}{4A_e\gamma_e}}\Big)^{\frac{2}{\gamma_e-1}}$ and the doping profile $b(t,x)=\bar\rho_i(t,x)=C_i(\bar{n}_i(t,x))^{2/(\gamma_i-1)}$. In addition, by the  weak lower semicontinuity of Sobolev norms, $(\bar{n}_e,\bar{w}_e, \D\bar{\vp})$ still satisfies \eqref{uni1}, \eqref{univp2} and \eqref{nablaw}. By \eqref{uni1}--\eqref{univp2} and the spatial regularity above, $\bar n_e$ and $\D\bar\vp$ belong to $L^q$ and have bounded gradients, so \eqref{ofarbar} holds. The local compactness established above and the uniform electric-field time-derivative bound in \S\ref{s5.2.3}, together with interpolation, yield
\[
(\bar n_e,\bar w_e,\D\bar\vp)\in C([0,T];H_{\mathrm{loc}}^{s'}),\qquad 0\leq s'<s.
\]
The ion Burgers flow and transport equation uniquely determine $b$. Since this common background cancels in the difference of the electron Poisson equations, uniqueness of $(\bar n_\nu,\bar w_e,\D\bar\vp)$ in the class of regular solutions specified in Theorem~\ref{thm2.2} follows from the localized energy argument in \S\ref{s5.2.2}. \qed

\section{Global error estimates}
In this section, the global error estimates between regular solutions to bipolar and unipolar Euler-Poisson equations are considered. We assume $u_\nu^0$ to be independent of $\eps$, and consequently, $\bar{\hu}_\nu=\hu_\nu$. We derive error estimates for \eqref{sys-symm} and its unipolar limit with compactly supported initial densities, uniformly in $\mk$, and then let $\mk\to\infty$.

For initial data $n_\nu^0\in \Gamma\cap L^q$ and $\mk\geq 1$, we set
\[
	n_{\nu,0}^\mk=n_\nu^0 F^\mk.
\]
For each $\mk\geq 1$, the compactly supported initial data $n_{\nu,0}^\mk$ belong to $H^s$ and satisfy the uniform smallness condition established in \S5.1. The compact-support construction in the proof of Theorem~\ref{thm4.1} therefore gives a unique global regular solution $(n_\nu^{\mk,\eps},\eps^{-\beta_\nu}w_\nu^{\mk,\eps},\vp^{\mk,\eps})$ to \eqref{sys-symm} with
\begin{equation}\label{initial7.1}
		(n_\nu^{\mk,\eps},\eps^{-\beta_\nu}w_\nu^{\mk,\eps})(0,x)=(n_{\nu,0}^\mk(x),0).
\end{equation}
By this uniform smallness and Lemma~\ref{global-smallness}, these solutions satisfy the estimates \eqref{uni1}, \eqref{univp2} and \eqref{nablaw} uniformly with respect to $\eps$ and $\mk$.

Throughout this section, $C\geq 1$ denotes a generic constant depending only on $T,s,d,q,\kappa$, $A_\nu,\gamma_\nu$, $\|u_\nu^0\|_\Xi$, $|u_\nu^0(0)|$, $\|\bar n_\nu^0\|_{\dot H^{s+1}}$ $(\nu=i,e)$, and $\|\D u_i^0\|_{\dot H^{s+1}}$, but independent of $\eps$ and $\mk$, which may differ from line to line.

\subsection{Error estimates for compactly supported initial data}
We first pass to the limit $\eps\to0$ in \eqref{sys-symm} with the initial data \eqref{initial7.1}. Let $\bar n_{\nu,0}^\mk=\bar n_\nu^0F^\mk$. The initial error assumption in Theorem \ref{thm2.3} and the cutoff estimates in \S5.1 give
\[
	|n_{\nu,0}^\mk-\bar n_{\nu,0}^\mk|_\infty +\|\D(n_{\nu,0}^\mk-\bar n_{\nu,0}^\mk)\|_{s-1}\leq C\eps^r,
\]
uniformly in $\mk$. By \eqref{uni1}, \eqref{univp2} and \eqref{nablaw}, one has, up to subsequences as $\eps\to0$,
\begin{equation}\label{conv000}
\begin{alignedat}{2}
			n_\nu^{\mk,\eps}&\rightharpoonup\bar n_\nu^\mk &&\quad\text{weakly-* in }L^\infty([0,T];\Gamma\cap L^q),\\
			(w_e^{\mk,\eps},\D\vp^{\mk,\eps})&\rightharpoonup(\bar w_e^\mk,\D\bar\vp^\mk) &&\quad\text{weakly-* in }L^\infty([0,T];\Gamma),\\
			w_i^{\mk,\eps}&\to0&&\quad\text{strongly in }L^\infty([0,T];\Gamma).
\end{alignedat}
\end{equation}
For every $R>0$, the equations and the same estimates imply
\begin{equation*}
		\sup_{0\leq t\leq T}\Big(\|\pt n_\nu^{\mk,\eps}(t)\|_{H^{s-1}(B_R)} +\|\pt w_e^{\mk,\eps}(t)\|_{H^{s-1}(B_R)}\Big)\leq C_R,
\end{equation*}
where $C_R$ depends only on $C$ and $R$, and is independent of $\eps$ and $\mk$. The local $H^s$ bounds and Lemma \ref{aubin} therefore give, along a common subsequence,
\begin{equation*}
		n_i^{\mk,\eps}\to\bar n_i^\mk \quad\text{strongly in }C([0,T];H_{\rm loc}^{s-1}).
\end{equation*}
Since $u_i^{\mk,\eps}=w_i^{\mk,\eps}+\hu_i\to\hu_i$, one obtains
\begin{equation}\label{ion-limit-g}
\begin{cases} \pt\bar n_i^\mk+\hu_i\cdot\D\bar n_i^\mk+\dfrac{\gamma_i-1}{2}\bar n_i^\mk\dive\hu_i=0,\\
			\pt\hu_i+(\hu_i\cdot\D)\hu_i=0,\\
			(\bar n_i^\mk,\hu_i)(0,x)=(\bar n_{i,0}^\mk,u_i^0)(x). \end{cases}
\end{equation}
For the electron variables, the same compactness argument yields
\begin{equation*}
		(n_e^{\mk,\eps},w_e^{\mk,\eps})\to(\bar n_e^\mk,\bar w_e^\mk) \quad\text{strongly in }C([0,T];H_{\rm loc}^{s-1}).
\end{equation*}
Together with \eqref{conv000}, this is sufficient to pass to the limit in the electron and Poisson equations in \eqref{sys-symm}, which gives
\begin{equation}\label{TEuler}
\begin{cases} \pt\bar n_e^\mk+(\bar w_e^\mk+\hu_e)\cdot\D\bar n_e^\mk +\dfrac{\gamma_e-1}{2}\bar n_e^\mk\dive\bar w_e^\mk =-\dfrac{\gamma_e-1}{2}\bar n_e^\mk\dive\hu_e,\\[2mm]
			\pt\bar w_e^\mk+((\bar w_e^\mk+\hu_e)\cdot\D)\bar w_e^\mk +\dfrac{\gamma_e-1}{2}\bar n_e^\mk\D\bar n_e^\mk =-\D\bar\vp^\mk-(\bar w_e^\mk\cdot\D)\hu_e,\\[2mm]
			\Delta\bar\vp^\mk=C_i(\bar n_i^\mk)^{2/(\gamma_i-1)}-C_e(\bar n_e^\mk)^{2/(\gamma_e-1)},\\
			(\bar n_e^\mk,\bar w_e^\mk)(0,x)=(\bar n_{e,0}^\mk(x),0). \end{cases}
\end{equation}
The ion density is determined by \eqref{ion-limit-g}$_1$, so \eqref{TEuler} is a decoupled electron system. By weak lower semicontinuity, the limit satisfies \eqref{uni1}, \eqref{univp2} and \eqref{nablaw}, uniformly in $\mk$.

Under the additional assumptions in Theorem \ref{thm2.3}, one obtains the following improved estimates of solutions to the limiting system \eqref{ion-limit-g}--\eqref{TEuler}:
\begin{equation}\label{limit-high-regularity}
		\sup_{0\leq t\leq T}\left(\sum_{\nu=i,e}\|\D\bar n_\nu^\mk(t)\|_s +\|\D^2\bar u_e^\mk(t)\|_{s-1}\right)\leq C, \qquad \bar u_e^\mk=\bar w_e^\mk+\hu_e.
\end{equation}
Indeed, the product estimates used in the initial approximation yield
\[
	\|\bar n_\nu^0F^\mk\|_{\dot H^{s+1}} \leq C\big(\|\bar n_\nu^0\|_{\dot H^1} +\|\bar n_\nu^0\|_{\dot H^{s+1}}\big),\qquad \mk\geq1.
\]
Lemma \ref{hu}, applied with $s+1$ in place of $s$, bounds $\D^2\hu_i$ in $L^\infty([0,T];H^s)$, so the differentiated ion transport equation bounds $\bar n_i^\mk$ in $\dot H^{s+1}$. For the electron variables, we apply $\D^{s+1}$ to the equations for $(\bar n_e^\mk,\bar u_e^\mk)$ and add the two energy identities. The highest order terms containing $\bar n_e^\mk\dive\bar u_e^\mk$ and $\bar n_e^\mk\D\bar n_e^\mk$ cancel after integration by parts. Moreover, the Poisson equation and the parameter assumptions of Theorem \ref{globalthm} give
\[
	\|\bar\vp^\mk\|_{\dot H^{s+2}} \leq C\sum_{\nu=i,e} \| (\bar n_\nu^\mk)^{\frac{2}{\gamma_\nu-1}}\|_{\dot H^s} \leq C.
\]
For noninteger $2/(\gamma_\nu-1)$, the composition bound follows from Lemma \ref{comm0} and \eqref{res1}; for integer powers, it follows from polynomial product estimates, which impose no upper bound on $s$. Spatial localization justifies the integrations by parts for the linearly growing velocity. The remaining commutators are controlled by the uniform Lipschitz bounds and the highest order energy. The resulting differential inequality is
\[
	\frac{\rmd}{\rmd t}\left( \|\bar n_e^\mk\|_{\dot H^{s+1}}^2 +\|\bar u_e^\mk\|_{\dot H^{s+1}}^2\right) \leq C\left(1+\|\bar n_e^\mk\|_{\dot H^{s+1}}^2 +\|\bar u_e^\mk\|_{\dot H^{s+1}}^2\right).
\]
Gr\"{o}nwall's inequality gives \eqref{limit-high-regularity}, with $C$ independent of $\mk$. These differentiated estimates can first be obtained for smooth approximations and then passed to the limit.

We define the error variables by
\[
	\pi_\nu=n_\nu^{\mk,\eps}-\bar n_\nu^\mk,\qquad v_e=w_e^{\mk,\eps}-\bar w_e^\mk,\qquad \chi=\vp^{\mk,\eps}-\bar\vp^\mk,
\]
where the initial density errors are $\pi_\nu^0=n_{\nu,0}^\mk-\bar n_{\nu,0}^\mk$.

\subsubsection{Errors for ion variables}
We estimate the ion density error together with the scaled velocity. Subtracting \eqref{ion-limit-g}$_1$ from the ion density equation in \eqref{sys-symm}, and writing the ion momentum equation in terms of $\pi_i$, one obtains
\begin{equation}\label{error-pi-i}
\begin{cases} \pt\pi_i+(w_i^{\mk,\eps}+\hu_i)\cdot\D\pi_i +\dfrac{\gamma_i-1}{2}n_i^{\mk,\eps}\dive w_i^{\mk,\eps} =-w_i^{\mk,\eps}\cdot\D\bar n_i^\mk -\dfrac{\gamma_i-1}{2}\pi_i\dive\hu_i,\\[2mm]
			\eps^{-2}\big(\pt w_i^{\mk,\eps} +((w_i^{\mk,\eps}+\hu_i)\cdot\D)w_i^{\mk,\eps} +(w_i^{\mk,\eps}\cdot\D)\hu_i\big) +\dfrac{\gamma_i-1}{2}n_i^{\mk,\eps}\D\pi_i\\
			\qquad=\D\vp^{\mk,\eps} -\dfrac{\gamma_i-1}{2}n_i^{\mk,\eps}\D\bar n_i^\mk,\\
			(\pi_i,w_i^{\mk,\eps})(0,x)=(\pi_i^0(x),0). \end{cases}
\end{equation}

\begin{lemma} One has
\begin{equation}\label{pi-i-est}
			|\pi_i(t)|_\infty+\|\D\pi_i(t)\|_{s-1} \leq C\eps^{\min\{r,1\}},\qquad t\in[0,T].
\end{equation}
\end{lemma}
\begin{proof}
We begin with the $L^\infty$ estimate for $\pi_i$. By \eqref{uni1}, \eqref{univp2} and \eqref{nablaw} and \eqref{limit-high-regularity}, integration along the characteristics of $w_i^{\mk,\eps}+\hu_i$ in \eqref{error-pi-i}$_1$ yields
\begin{align*}
			|\pi_i(t)|_\infty &\leq |\pi_i^0|_\infty+C\int_0^t\Big(|\pi_i|_\infty +|n_i^{\mk,\eps}|_\infty|\D w_i^{\mk,\eps}|_\infty +|w_i^{\mk,\eps}|_\infty|\D\bar n_i^\mk|_\infty\Big)\rmd\tau\\
			&\leq |\pi_i^0|_\infty+C\int_0^t|\pi_i(\tau)|_\infty\rmd\tau+C\eps.
\end{align*}
Gr\"{o}nwall's inequality then gives
\[
		|\pi_i(t)|_\infty\leq C\big(|\pi_i^0|_\infty+\eps\big) \leq C\eps^{\min\{r,1\}}.
\]

For $1\leq l\leq s$, applying $\D^l$ to \eqref{error-pi-i}$_1$ and testing against $\D^l\pi_i$, one obtains
\begin{align*}
			\frac12\frac{\rmd}{\rmd t}|\D^l\pi_i|_2^2 ={}&\frac12\int_{\R^d}\dive(w_i^{\mk,\eps}+\hu_i)|\D^l\pi_i|^2\rmd x -\int_{\R^d}\D^l(w_i^{\mk,\eps}\cdot\D\bar n_i^\mk)\D^l\pi_i\rmd x\\
			&-\int_{\R^d}[\D^l,w_i^{\mk,\eps}\cdot\D]\pi_i\,\D^l\pi_i\rmd x -\int_{\R^d}[\D^l,\hu_i\cdot\D]\pi_i\,\D^l\pi_i\rmd x\\
			&-\frac{\gamma_i-1}{2}\int_{\R^d}\D^l(\pi_i\dive\hu_i)\D^l\pi_i\rmd x -\frac{\gamma_i-1}{2}\int_{\R^d}\D^l(n_i^{\mk,\eps}\dive w_i^{\mk,\eps})\D^l\pi_i\rmd x.
\end{align*}
We denote the six terms on the right by $J_{i,j}^l$, $1\leq j\leq6$, in order. Spatial localization justifies these integrations by parts. By Lemmas \ref{comm1}--\ref{comm2} and \ref{hu}, Sobolev embedding, the uniform estimates and integration by parts, one has
\begin{align*}
			|J_{i,1}^l| &\leq C\big(|\D w_i^{\mk,\eps}|_\infty+|\D\hu_i|_\infty\big)|\D^l\pi_i|_2^2 \leq C\|\D\pi_i\|_{s-1}^2,\\
			|J_{i,2}^l| &\leq C|\D^l\pi_i|_2\Big(|w_i^{\mk,\eps}|_\infty|\D^{l+1}\bar n_i^\mk|_2 +|\D^lw_i^{\mk,\eps}|_2|\D\bar n_i^\mk|_\infty\Big)\\
			&\leq C\eps\|\D\pi_i\|_{s-1} \leq C\|\D\pi_i\|_{s-1}^2+C\eps^2,\\
			|J_{i,3}^l| &\leq C|\D^l\pi_i|_2\Big(|\D^lw_i^{\mk,\eps}|_2|\D\pi_i|_\infty +|\D w_i^{\mk,\eps}|_\infty|\D^l\pi_i|_2\Big) \leq C\|\D\pi_i\|_{s-1}^2,\\
			|J_{i,4}^l|&\leq C\|\D\pi_i\|_{s-1}^2,\\
			|J_{i,5}^l| &\leq C|\D^l\pi_i|_2\Big(|\pi_i|_\infty|\D^l\dive\hu_i|_2 +|\D^l\pi_i|_2|\dive\hu_i|_\infty\Big)\\
			&\leq C\|\D\pi_i\|_{s-1}^2+C\eps^{2\min\{r,1\}},\\
			J_{i,6}^l &\leq \frac{\gamma_i-1}{2}\int_{\R^d} \D^l(n_i^{\mk,\eps}\D\pi_i)\cdot\D^lw_i^{\mk,\eps}\rmd x +C\|\D\pi_i\|_{s-1}\|\D w_i^{\mk,\eps}\|_{s-1}.
\end{align*}
Here \eqref{limit-high-regularity} is used for $J_{i,2}^l$, and the preceding $L^\infty$ estimate is used for $J_{i,5}^l$. The background term $J_{i,4}^l$ is treated as in \eqref{F2} and \eqref{Gs}.

Homogeneous interpolation and estimates \eqref{uni1}, \eqref{univp2} and \eqref{nablaw} and \eqref{limit-high-regularity} give
\begin{align*}
			|w_i^{\mk,\eps}|_\infty+|\D w_i^{\mk,\eps}|_\infty &\leq C\|\D w_i^{\mk,\eps}\|_{s-1},\\
			\left\|\D\left(\D\vp^{\mk,\eps} -\frac{\gamma_i-1}{2}n_i^{\mk,\eps}\D\bar n_i^\mk\right)\right\|_{s-1} &\leq C.
\end{align*}
Applying $\D^l$ to \eqref{error-pi-i}$_2$, multiplying by $\D^lw_i^{\mk,\eps}$ and integrating by parts yield
\begin{align*}
			\frac{1}{2\eps^2}\frac{\rmd}{\rmd t}|\D^lw_i^{\mk,\eps}|_2^2 \leq{}-\frac{\gamma_i-1}{2}\int_{\R^d} \D^l(n_i^{\mk,\eps}\D\pi_i)\cdot\D^lw_i^{\mk,\eps}\rmd x +C\eps^{-2}\|\D w_i^{\mk,\eps}\|_{s-1}^2+C\eps^2.
\end{align*}
Adding the two energy estimates and summing over $1\leq l\leq s$, one obtains
\[
		\frac{\rmd}{\rmd t}\Big(\|\D\pi_i\|_{s-1}^2+\eps^{-2}\|\D w_i^{\mk,\eps}\|_{s-1}^2\Big) \leq C\Big(\|\D\pi_i\|_{s-1}^2+\eps^{-2}\|\D w_i^{\mk,\eps}\|_{s-1}^2\Big) +C\eps^{2\min\{r,1\}}.
\]
Since $w_i^{\mk,\eps}(0)=0$, Gr\"{o}nwall's inequality yields
\[
		\|\D\pi_i(t)\|_{s-1}^2+\eps^{-2}\|\D w_i^{\mk,\eps}(t)\|_{s-1}^2 \leq C\Big(\|\D\pi_i^0\|_{s-1}^2+\eps^{2\min\{r,1\}}\Big) \leq C\eps^{2\min\{r,1\}}.\qedhere
\]
\end{proof}

\subsubsection{Errors for electron variables}

Subtracting \eqref{TEuler} from \eqref{sys-symm} gives
\begin{equation}\label{cha}
\begin{cases} \pt\pi_e+(w_e^{\mk,\eps}+\hu_e)\cdot\D\pi_e+\dfrac{\gamma_e-1}{2}n_e^{\mk,\eps}\dive v_e =-v_e\cdot\D\bar n_e^\mk-\dfrac{\gamma_e-1}{2}\pi_e\dive\bar u_e^\mk,\\[2mm]
			\pt v_e+((w_e^{\mk,\eps}+\hu_e)\cdot\D)v_e+\dfrac{\gamma_e-1}{2}n_e^{\mk,\eps}\D\pi_e =-(v_e\cdot\D)\bar u_e^\mk-\dfrac{\gamma_e-1}{2}\pi_e\D\bar n_e^\mk-\D\chi,\\[2mm]
			\Delta\chi=C_i\big((n_i^{\mk,\eps})^{\frac{2}{\gamma_i-1}}-(\bar n_i^\mk)^{\frac{2}{\gamma_i-1}}\big) -C_e\big((n_e^{\mk,\eps})^{\frac{2}{\gamma_e-1}}-(\bar n_e^\mk)^{\frac{2}{\gamma_e-1}}\big),\\[2mm]
			(\pi_e,v_e)(0,x)=(\pi_e^0(x),0), \end{cases}
\end{equation}
where $\bar u_e^\mk=\bar w_e^\mk+\hu_e$. We set $U_e=(\pi_e,v_e^\top)^\top$ and $\mathcal X_e=|U_e|_\infty+\|\D U_e\|_{s-1}$.

The following lemma gives the $L^\infty$ estimates for the electron errors.
\begin{lemma} For $t\in[0,T]$, one has
\begin{equation}\label{e-infty}
			|U_e(t)|_\infty\leq C\eps^{\min\{1,r\}}+C\int_0^t\mathcal X_e(\tau)\rmd\tau.
\end{equation}
\end{lemma}
\begin{proof}
The uniform bounds for $n_e^{\mk,\eps}$, $\D\bar n_e^\mk$ and $\D\bar u_e^\mk$, together with Sobolev embedding, imply that the absolute value of the right-hand side of the first equation in \eqref{cha}, after the divergence term is moved to that side, is bounded by $C\mathcal X_e$. Integration along the characteristics of $w_e^{\mk,\eps}+\hu_e$ therefore yields
\begin{equation}\label{pi-e-infty}
			|\pi_e(t)|_\infty\leq C\eps^r+C\int_0^t\mathcal X_e(\tau)\rmd\tau.
\end{equation}
The second equation similarly gives
\begin{equation}\label{mid--0}
			|v_e(t)|_\infty\leq C\int_0^t\big(\mathcal X_e(\tau)+|\D\chi(\tau)|_\infty\big)\rmd\tau.
\end{equation}

For the electric field, the mean value formula, H\"older's inequality and \eqref{res1} give
\[
		\sum_{\nu=i,e}\big|(n_\nu^{\mk,\eps})^{\frac{2}{\gamma_\nu-1}}-(\bar n_\nu^\mk)^{\frac{2}{\gamma_\nu-1}}\big|_2 \leq C\sum_{\nu=i,e}|\pi_\nu|_{\frac{2d}{d-2}} \leq C\sum_{\nu=i,e}|\D\pi_\nu|_2.
\]
The mean value coefficient is uniformly bounded in $L^d$, since $d(\frac{2}{\gamma_\nu-1}-1)>q$ and the densities are uniformly bounded in $L^q\cap L^\infty$. For $2\leq l\leq s$, Lemmas \ref{comm0}--\ref{comm1}, with polynomial product estimates used for integer powers, imply
\[
		\sum_{\nu=i,e}\big\|(n_\nu^{\mk,\eps})^{\frac{2}{\gamma_\nu-1}}-(\bar n_\nu^\mk)^{\frac{2}{\gamma_\nu-1}}\big\|_{\dot H^{l-1}} \leq C\sum_{\nu=i,e}\big(|\pi_\nu|_\infty+\|\pi_\nu\|_{\dot H^{l-1}}\big).
\]
The Poisson equation and \eqref{pi-i-est} consequently yield, for $1\leq l\leq s$,
\begin{equation}\label{vpvpvp}
			|\D^{l+1}\chi|_2\leq C\eps^{\min\{1,r\}}+C\mathcal X_e.
\end{equation}
By homogeneous Sobolev interpolation, one then has
\begin{equation}\label{vpvpvp0}
			|\D\chi|_\infty \leq C\|\D\chi\|_{\dot H^1}^{\frac{2s-d}{2s-2}}\|\D\chi\|_{\dot H^s}^{\frac{d-2}{2s-2}} \leq C\eps^{\min\{1,r\}}+C\mathcal X_e.
\end{equation}
Substitution into \eqref{mid--0} gives
\begin{equation*}
			|v_e(t)|_\infty\leq C\eps^{\min\{1,r\}}+C\int_0^t\mathcal X_e(\tau)\rmd\tau.
\end{equation*}
Together with \eqref{pi-e-infty}, this proves the lemma.
\end{proof}

The next lemma concerns the higher derivatives of the electron errors.
\begin{lemma} For $t\in[0,T]$, one has
\begin{equation}\label{e-high}
			\|\D U_e(t)\|_{s-1}\leq C\eps^{\min\{1,r\}}+C\int_0^t\mathcal X_e(\tau)\rmd\tau.
\end{equation}
\end{lemma}
\begin{proof}
We set $W_e^{\mk,\eps}=(n_e^{\mk,\eps},(w_e^{\mk,\eps})^\top)^\top$. The first two equations in \eqref{cha} have the symmetric form
\begin{equation*}
			\pt U_e+\sum_{j=1}^d A_e^j(W_e^{\mk,\eps};\hu_e,\eps)\pa_{x_j}U_e =Q_e^1(\D\chi)-Q_e^2(U_e;\D\hu_e)-Q_e^6(U_e),
\end{equation*}
where the remaining term is defined by
\[
		Q_e^6(U_e)=\begin{pmatrix} v_e\cdot\D\bar n_e^\mk+\dfrac{\gamma_e-1}{2}\pi_e\dive\bar w_e^\mk\\[2mm]
			(v_e\cdot\D)\bar w_e^\mk+\dfrac{\gamma_e-1}{2}\pi_e\D\bar n_e^\mk \end{pmatrix}.
\]
For $1\leq l\leq s$, differentiation and integration by parts give
\begin{align}\label{estart1}
			\frac12\frac{\rmd}{\rmd t}|\D^lU_e|_2^2 ={}&\frac12\sum_{j=1}^d\int_{\R^d}(\D^lU_e)^\top\pa_{x_j}A_e^j(W_e^{\mk,\eps};\hu_e,\eps)\D^lU_e\rmd x -\int_{\R^d}(\D^lv_e)^\top\D^{l+1}\chi\rmd x\nonumber\\
			&-\int_{\R^d}(\D^lU_e)^\top\sum_{j=1}^d[\D^l,A_e^j(W_e^{\mk,\eps};\hu_e,\eps)]\pa_{x_j}U_e\rmd x \\
			&-\int_{\R^d}(\D^lU_e)^\top\D^lQ_e^2(U_e;\D\hu_e)\rmd x -\int_{\R^d}(\D^lU_e)^\top\D^lQ_e^6(U_e)\rmd x =:\sum_{j=1}^5J_{e,j}^l.\notag
\end{align}
The uniform gradient bounds, \eqref{vpvpvp}, and Lemmas \ref{comm2} and \ref{hu} yield the following estimates, with the background term treated as in \eqref{F2} and \eqref{Gs}:
\begin{align*}
			|J_{e,1}^l|&\leq C|\D^lU_e|_2\|\D U_e\|_{s-1},\\
			|J_{e,2}^l|&\leq C|\D^lU_e|_2\big(\eps^{\min\{1,r\}}+\mathcal X_e\big),\\
			|J_{e,3}^l|&\leq C|\D^lU_e|_2\|\D U_e\|_{s-1}.
\end{align*}
To estimate the remaining terms, we combine their coefficients before differentiating:
\[
		Q_e^2(U_e;\D\hu_e)+Q_e^6(U_e)=\begin{pmatrix} v_e\cdot\D\bar n_e^\mk+\dfrac{\gamma_e-1}{2}\pi_e\dive\bar u_e^\mk\\[2mm]
			(v_e\cdot\D)\bar u_e^\mk+\dfrac{\gamma_e-1}{2}\pi_e\D\bar n_e^\mk \end{pmatrix}.
\]
Estimate \eqref{limit-high-regularity} and Lemma \ref{comm1} yield
\begin{align*}
			|J_{e,4}^l+J_{e,5}^l| \leq{}& C|\D^lU_e|_2\Big(|\D(\bar n_e^\mk,\bar u_e^\mk)|_\infty|\D^lU_e|_2 +|U_e|_\infty|\D^{l+1}(\bar n_e^\mk,\bar u_e^\mk)|_2\Big)\nonumber\\
			\leq{}& C|\D^lU_e|_2\mathcal X_e.
\end{align*}
Summing \eqref{estart1} for $1\leq l\leq s$, one obtains
\begin{equation*}
			\frac{\rmd}{\rmd t}\|\D U_e\|_{s-1}^2 \leq C\|\D U_e\|_{s-1}\big(\mathcal X_e+\eps^{\min\{1,r\}}\big).
\end{equation*}
Applying this inequality to $(\|\D U_e\|_{s-1}^2+\delta^2)^{1/2}$ and then letting $\delta\downarrow0$ gives
\[
		\|\D U_e(t)\|_{s-1}\leq\|\D U_e(0)\|_{s-1} +C\int_0^t\big(\mathcal X_e(\tau)+\eps^{\min\{1,r\}}\big)\rmd\tau.
\]
The initial error bound therefore proves \eqref{e-high}.
\end{proof}

Combining \eqref{e-infty} and \eqref{e-high}, one obtains
\[
	\mathcal X_e(t)\leq C\eps^{\min\{1,r\}}+C\int_0^t\mathcal X_e(\tau)\rmd\tau.
\]
Gr\"{o}nwall's inequality then yields
\begin{equation}\label{e-err}
		\mathcal X_e(t)\leq C\eps^{\min\{1,r\}},\qquad 0\leq t\leq T,
\end{equation}
with $C$ independent of $\eps$ and $\mk$. This completes the error estimates for compactly supported initial densities.

\subsection{Error estimates between original systems} From \eqref{pi-i-est}, \eqref{vpvpvp0},\eqref{vpvpvp} and \eqref{e-err}, one obtains the following estimate:
\begin{equation}\label{errnaka}
		|(\pi_i,\pi_e,v_e,\D\chi)|_{\infty}+\|\D(\pi_i,\pi_e,v_e,\D\chi)\|_{s-1}\leq  C\eps^{\min\{1,r\}},
\end{equation}
uniformly in $\eps$ and $\mk$. We now let $\mk\to\infty$ to obtain the error estimates between regular solutions to \eqref{sys-symm} and \eqref{ion-limit-g}--\eqref{TEuler} without the initial density cutoff.

For each fixed $\eps>0$, the convergences \eqref{convkwe}--\eqref{convkst} give the unique regular solution $(n_\nu^\eps,\eps^{-\beta_\nu}w_\nu^\eps,\vp^\eps)$ to \eqref{sys-symm}. The limit solutions $(\bar n_\nu^\mk,\bar w_e^\mk,\bar\vp^\mk)$ satisfy \eqref{uni1}, \eqref{univp2} and \eqref{nablaw} uniformly in $\mk$. Thus, up to a common subsequence,
\begin{equation}\label{convewe}
\begin{alignedat}{2}
			\bar n_\nu^\mk&\rightharpoonup\bar n_\nu^* &&\quad\text{weakly-* in }L^\infty([0,T];\Gamma\cap L^q),\\
			(\bar w_e^\mk,\D\bar\vp^\mk)&\rightharpoonup(\bar w_e^*,\D\bar\vp^*) &&\quad\text{weakly-* in }L^\infty([0,T];\Gamma).
\end{alignedat}
\end{equation}
By \eqref{ion-limit-g}--\eqref{TEuler}, the sequence $(\pt\bar n_i^\mk,\pt\bar n_e^\mk,\pt\bar w_e^\mk)$ is bounded in $L^\infty([0,T];H^{s-1}(B_R))$ for every $R>0$, uniformly in $\mk$. The local $H^s$ bounds and Lemma \ref{aubin} then yield
\begin{equation}\label{convest2}
		(\bar n_i^\mk,\bar n_e^\mk,\bar w_e^\mk)\to(\bar n_i^*,\bar n_e^*,\bar w_e^*) \quad\text{strongly in }C([0,T];H_{\rm loc}^{s-1}).
\end{equation}
The convergences \eqref{convewe}--\eqref{convest2} suffice to pass to the limit in \eqref{ion-limit-g}--\eqref{TEuler}, giving regular solutions to \eqref{ion-n-limit} and \eqref{electron-limit} with initial densities $\bar n_\nu^0$. Uniqueness for the ion transport equation and the localized energy argument in \S5.2.2 for the electron system identify these limits with $(\bar n_\nu,\bar w_e,\D\bar\vp)$ from \S5.3.

By lower semicontinuity and local time continuity, one obtains from \eqref{errnaka}
\begin{equation*}
		\|(n_\nu^\eps-\bar n_\nu,w_e^\eps-\bar w_e,\D\vp^\eps-\D\bar\vp)(t)\|_\Gamma \leq C\eps^{\min\{1,r\}},\qquad 0\leq t\leq T.
\end{equation*}
Finally, $u_i^\eps-\bar u_i=w_i^\eps$ satisfies $\|w_i^\eps(t)\|_\Gamma\leq C\eps$ by \eqref{uni1}, while $u_e^\eps-\bar u_e=w_e^\eps-\bar w_e$. Recalling the definition of $n_\nu$ in \S2.1 completes the proof of Theorem \ref{thm2.3}. \hfill $\square$

\bigskip

\section*{Appendix: some basic lemmas}
\setcounter{section}{7} \setcounter{lemma}{0} \setcounter{equation}{0}

In this appendix,  we list some lemmas which were used frequently in the previous sections.
%The first one is the well-known Sobolev embedding theorem.
%\begin{lemma}  \cite{Ladyzenskaja1968} Let $p>3/2$ and $f\in H^p(\R^3)$. Then
%\[
%|f|_\infty\leq C|f|_2^{1-\frac{3}{2p}}|\D^p f|_2^{\frac{3}{2p}}.
%\]
%\end{lemma}
The first one is the  well-known Gagliardo-Nirenberg inequality.
\begin{lemma}\cite{Nirenberg1959}\label{ga-ni}\ Let $1 \leq q_1, q_2, r \leq \infty$ and  $f\in L^{q_1}(\mathbb{R}^d)$ satisfying  $\D^i f\in L^r(\mathbb{R}^d)$.     Suppose also that  real numbers $\theta$ and $q_2$,  and  natural numbers $i$ and $j$ satisfy
$$\frac{1}{{q_2}} = \frac{j}{d} + \Big( \frac{1}{r} - \frac{i}{d} \Big) \theta + \frac{1 - \theta}{q_1} \quad \text{and} \quad
		\frac{j}{i} \leq \theta \leq 1.
$$
Then $\D^jf\in L^{q_2}(\mathbb{R}^d)$, and  there exists a constant $C$ depending only on $(i, d, j, q_1, r,\theta)$ such that the following estimates hold:
\begin{equation*}
\begin{split}
				\| \nabla^{j} f \|_{L^{{q_2}}} \leq C \| \nabla^{i} f \|_{L^{r}}^{\theta} \| f \|_{L^{q_1}}^{1 - \theta}.
\end{split}
\end{equation*}
Moreover, if $j = 0$, $ir < d$ and $q_1 = \infty$, then it is necessary to make the additional assumption that either $f$ tends to zero at infinity or that $f$ lies in $L^{\tilde{q}}(\mathbb{R}^d)$ for some finite $\tilde{q} > 0$; if $1 < r < \infty$ and $i -j -d/r$ is a non-negative integer, then it is necessary to assume also that $\theta \neq 1$. In addition, if $j=0$, $ir=d$, and $q_1=q_2=\infty$, we require $\theta=0$.
\end{lemma}

%\begin{lemma}\label{ga-ni}\cite{Nirenberg1959} Let $1\leq q\leq +\infty$ be a positive extended real quantity. Let $j$ and $m$ be non-negative integers such that $j<m$. Furthermore, let $1\leq r\leq +\infty$ be a positive extended real quantity, $p\geq 1$ be real and $\theta\in[0,1]$ such that the relation holds:
%\[
%\frac{1}{p}=\frac{j}{d}+\theta\Big(\frac{1}{r}-\frac{m}{d}\Big)+\frac{1-\theta}{q}, \qquad \frac{j}{m}\leq \theta\leq 1.
%\]
%Then
%\[
%|\D^j u|_p\leq C|\D^m u|_r^\theta |u|_q^{1-\theta},
%\]
%for any $u\in L^q(\R^d)$ such that $\D^m u\in L^r(\R^d)$.
%\end{lemma}

%Some useful inequalities that can be obtained from the above lemma are given as follows. We assume $s>\frac{d}{2}+1$ in the following and either $u$ tends to $0$ at infinity or $u\in L^{\tilde{q}}$ for some finite $\tilde{q}>1$.
%\begin{equation}
%   \begin{split}
%      |u|_\infty\leq& C\|u\|_{\dot{H}^1}^{\frac{2s-d}{2s-2}}\|u\|_{\dot{H}^s}^{\frac{d-2}{2s-s}},\\
%     |\D u|_\infty+|\D^j u|_{\frac{d}{j-1}} \leq& C|\D u|_2^{\frac{2s-d-2}{2s-2}}|\D^s u|_2^{\frac{d}{2s-2}}, \quad 1<j<1+d,\\
%    |\D^j u|_2\leq& C|\D u|_2^{\frac{s-j}{s-1}}|\D^s u|_2^\frac{j-1}{s-1},\quad 1\leq j\leq s,\\
%   |u|_{\frac{2d}{d-2}}\leq &C|\D u|_{2}.
%\end{split}
%\end{equation}

%\begin{lemma}\label{ga-ni}\cite{Ladyzenskaja1968} Let $r\geq 0, i\in[0,r]$ and  $f\in L^\infty\cap H^r$. Then $\D^i f\in L^{2r/i}$, and there exist some constants $C_{i,r}>0$, such that
%	\[
%	|\D^i f|_{2r/i}\leq C_{i,r}|f|_\infty^{1-i/r}|\D^r f|_2^{i/r}.
%	\]
%\end{lemma}

As a consequence of the Aubin-Lions lemma, one has the following conclusion (cf. \cite{Simon1987}). 

\begin{lemma} \cite{Simon1987}\label{aubin} Let $X_0,X$ and $X_1$ be three Banach spaces satisfying $X_0\subset X\subset X_1$. Suppose that $X_0$ is compactly embedded in $X$ and $X$ is continuously embedded in $X_1$.

\noindent  {\rm{(1)}} Let $\hat{F}$ be bounded in $L^p([0,T];X_0)$ with $1\leq p <\infty$, and $\frac{\pa \hat{F}}{\pa t}$ be bounded in $L^1([0,T];X_1)$. Then $\hat{F}$ is relatively compact in $L^p([0,T];X)$.

\noindent  {\rm{(2)}} Let $\tilde{F}$ be bounded in $L^{\infty}([0,T];X_0)$ and $\frac{\pa \tilde{F}}{\pa t}$ be bounded in $L^q([0,T];X_1)$ with $q>1$. Then $\tilde{F}$ is relatively compact in $C([0,T];X)$.

\end{lemma}

The next  one is used to improve a weak convergence to a strong one.

\begin{lemma}\cite{Majda1984} If the function sequence $\{w_n\}_{n=1}^\infty$ converges weakly to $w$ in a Hilbert space $X$, then it converges strongly to $w$ in $X$ if and only if
\[
			\|w\|_X\geq \limsup_{n\to\infty}\|w_n\|_X.
\]

\end{lemma}

The fourth  lemma concerns the $L^p$ estimates of the Riesz potential.

\begin{lemma}\label{HLS}\cite{Aubin1982}  Let $0<\al<d$ and $1<p<q<\infty$. The Riesz potential $I_\al f$ on $\R^d$ is  defined as
\[
			(I_\al f)(x)=\int_{\R^d}\dfrac{f(y)}{|x-y|^{d-\al}}\rmd y.
\]
Then for $(p,q)$ satisfying $1/q=1/p-\alpha/d$, there exists a constant $C$ depending only on $(p,\al,d)$, such that the following estimate holds:
\[
			|I_\al f|_q\leq C|f|_p.
\]
\end{lemma}

The fifth lemma concerns regularity  estimates of the Poisson equation.

\begin{lemma}\label{Poisson} \cite{Gilbarg1983} Let $d\geq 3$ and $f \in H^{s-1}\cap L^1$. Define the Newtonian potential of $f$ by
\begin{equation}\label{eq8-1}
				\Phi(x)=-\dfrac{1}{d(d-2)\omega_d}\int_{\R^d}\dfrac{f(y)}{|x-y|^{d-2}}\rmd y,
\end{equation}
where $\omega_d$ is the volume of the unit ball in $\R^d$. Then $\Delta\Phi=f$, and for $q>\max\{1,\frac{d}{d-1}\}$, there exists a constant $C>0$ depending only on $q$ and $d$, such that the following estimate holds:
\[
			|\D\Phi|_q\leq C|f|_{\frac{qd}{q+d}}.
\]
In addition, if $f$ has compact support, i.e., supp$_x f\subset B_{\tilde{R}}$, then it holds
\[
			\|\D\Phi\|_s\leq  C(\tilde{R}^{\frac{d+2}{2}}+1)\|f\|_{s-1},
\]
where the constant $C>0$ is independent of $\tilde{R}$.
\end{lemma}
\begin{proof} Since $f\in H^{s-1}\cap L^1$, the integral in \eqref{eq8-1} is well-defined. By classical potential theory (cf. \cite{Gilbarg1983}), one has
\[
			\D\Phi=\dfrac{1}{d\omega_d}\int_{\R^d}\dfrac{(x-y)f(y)}{|x-y|^{d}}\rmd y.
\]
%which yields that 
%\[
%|\D\Phi|\leq C\int_{\R^d}\dfrac{|f(y)|}{|x-y|^{d-1}}\rmdy.
%\]
By Lemma \ref{HLS}, for $q>\max\{1,\frac{d}{d-1}\}$,
\[
			|\D\Phi|_q\leq C|f|_{\frac{qd}{q+d}}.
\]
In addition, if $f$ has compact support, i.e., supp$_x f\subset B_{\tilde{R}}$ for some positive $\tilde{R}$, then
\[
			|\D\Phi|_2\leq C|f|_{\frac{2d}{d+2}}\leq C|f|_\infty\tilde{R}^{\frac{d+2}{2}}.
\]
For high order derivatives, it holds
\begin{align*}
				|\D^l \D\Phi|_2\leq C|\D^{l-1} f|_2, \quad 1\leq l\leq s,
\end{align*}
which  implies  that $\D \Phi\in H^s$.
\end{proof}

The sixth to eighth lemmas concern the product estimates and composite function estimates, which can be found in a series of papers (See \cites{Danchin2021,Kateb2003,Kato1988}).

\begin{lemma} \cite{Kateb2003}\label{comm0} Let $\al\geq 1$ and $0\leq \sigma<\al+1/2$. Then there exists a constant $C$ depending only on $d$, $\al$ and $\sigma$, such that
\[
			\||z|^\al\|_{\dot{H}^\sigma}\leq C|z|_\infty^{\al-1}\|z\|_{\dot{H}^\sigma}.
\]
\end{lemma}

\begin{lemma} \cites{Kato1988,Majda1984} \label{comm1} Let $m\geq1$ be an integer and $|\alpha|=m$. If $f,g\in L^\infty$ and $\D^m f,\D^m g\in L^2$, then
\[
|\pa^\alpha(fg)|_2\leq C_{d,m}\bigl(|f|_\infty|\D^m g|_2+|g|_\infty|\D^m f|_2\bigr).
\]
The same bound holds for each Leibniz product $\pa^\beta f\,\pa^{\alpha-\beta}g$, $\beta\leq\alpha$. No $L^2$ assumption on $f$ or $g$ themselves is required.
\end{lemma}
\begin{lemma} \cites{Blanc2021,Danchin2021,Kato1988,Majda1984}\label{comm2} Let $m\geq1$ be an integer and $|\alpha|=m$. The following homogeneous commutator estimates hold whenever the norms on the corresponding right-hand side are finite:
\[
\bigl|[\pa^\alpha,f]g\bigr|_2
\leq C_{d,m}\bigl(|\D^m f|_2|g|_\infty+|\D f|_\infty|\D^{m-1}g|_2\bigr),
\]
where $[\pa^\alpha,f]g=\pa^\alpha(fg)-f\pa^\alpha g$. Moreover, if $m\geq2$, then
\[
\left|[\pa^\alpha,f]g-\sum_{j:\alpha_j>0}\alpha_j(\pa_j f)\pa^{\alpha-e_j}g\right|_2
\leq C_{d,m}\bigl(|\D^m f|_2|g|_\infty+|\D^2 f|_\infty|\D^{m-2}g|_2\bigr),
\]
where $e_j$ is the $j$-th coordinate multi-index. For weakly differentiable functions, the commutators are understood as their Leibniz remainders. In particular, neither $f,g\in H^m$ nor $\D^m g\in L^2$ is required.
\end{lemma}
These integer-order estimates follow from the Leibniz formula, homogeneous Gagliardo--Nirenberg interpolation and H\"older's inequality, first for smooth functions and then by regularization. They also hold componentwise for vectors and matrices. In particular, Lemma~\ref{comm1} applies to $K_\nu W_\nu$ in \eqref{I4l}, since Lemma~\ref{hu} controls $|K_\nu|_\infty$ and $\|K_\nu\|_{\dot H^m}$, without requiring $K_\nu\in L^2$.

The ninth lemma concerns a convergence argument.

\begin{lemma}\label{lem:two-step-convergence}
Let $\{v^m\}_{m\geq0}$ be a sequence in a Banach space $X$. Suppose that, for some constant $0\leq \bar\delta<1/4$,
\[
			\|v^{m+1}-v^m\|_X^2\leq \bar\delta\big(\|v^m-v^{m-1}\|_X^2+\|v^{m-1}-v^{m-2}\|_X^2\big),\qquad m\geq2.
\]
Then
\[
			\sum_{m=1}^\infty\|v^m-v^{m-1}\|_X<\infty,
\]
and the whole sequence converges in $X$.
\end{lemma}
\begin{proof}
Taking square roots gives
\[
			\|v^{m+1}-v^m\|_X\leq\sqrt{\bar\delta}\big(\|v^m-v^{m-1}\|_X+\|v^{m-1}-v^{m-2}\|_X\big).
\]
Summing over $m=2,\ldots,M-1$, with $M\geq3$, yields
\[
			(1-2\sqrt{\bar\delta})\sum_{m=1}^{M}\|v^m-v^{m-1}\|_X \leq\|v^1-v^0\|_X+\|v^2-v^1\|_X.
\]
Since $2\sqrt{\bar\delta}<1$, letting $M\to\infty$ proves the summability. For $p>q\geq0$,
\[
			\|v^p-v^q\|_X\leq\sum_{m=q+1}^{p}\|v^m-v^{m-1}\|_X\longrightarrow0\qquad\text{as }q\to\infty.
\]
The conclusion follows from the completeness of $X$.
\end{proof}

Finally,  we give  one lemma on the  decay estimates of solutions of Burgers equations. Let $\hu$ be the solution in $[0,T]\times \R^d$ of  \eqref{burgers}. Without loss of generality, we drop the subscript $\nu$. Then $\hu$ is a constant in $t$  along  the particle path $\Omega(t;x_0)$ defined as
\begin{equation*}
			\dfrac{\rmd}{\rmd t}\Omega(t;x_0)=\hu(t,\Omega(t;x_0)), \quad \Omega(0,x_0)=x_0.
\end{equation*}
More specifically, it holds that
\begin{equation*}
			\hu(t, \Omega(t;x_0))=u^0(x_0), \quad \D\hu(t,\Omega(t;x_0))=(\Id+t\D u^0(x_0))^{-1}\D u^0(x_0),
\end{equation*}
where $\Id$ is the $d\times d$ unit matrix. Based on these, one has
\begin{lemma}\cite{Grassin1998} \label{hu}Let $s> \frac{d}{2}+1$ be an integer. Assume that
\[
			\nabla u^0\in L^\infty, \quad \D^2 u^0\in H^{s-1},
\]
and there exists a constant $\kappa>0$ such that for all $x\in\R^d$,
\[
			\text{dist}\,(\text{Sp}(\nabla u^0(x)),\R^{-})\geq \kappa.
\]
Then there exists a unique global classical solution $\hu$ to the problem \eqref{burgers} satisfying
\begin{enumerate}
\item For all $x\in\R^d$ and $t\geq 0$,
\begin{equation}\label{nablau}
					\D\hu(t,x)=\dfrac{1}{1+t}\Id+\dfrac{1}{(1+t)^2}K(t,x),
\end{equation}
where the matrix $K(t,x)$ satisfies
\[
				\|K\|_{L^\infty(\R^+\times\R^d)}\leq C(1+\kappa^{-d}|\D u^0|_\infty^{d-1}):=M;
\]
\item $
|\D^\sigma \hu(t,\cdot)|_2\leq C(1+t)^{\frac{d}{2}-\sigma-1}, \qquad 2\leq \sigma\leq s+1; $
\item $|\D^2 \hu(t,\cdot)|_\infty\leq C(1+t)^{-3}|\D^2 u^0|_\infty$,
\end{enumerate}
where $C>0$ is a constant depending only on $s,d,\kappa$, $u^0$ and $\sigma$. Moreover, it holds
\begin{equation}\label{galilean}
				|\hu(t,x)|\leq C(1+|\D \hu|_\infty)(1+|x|){\rm{exp}}\Big(\int_0^t|\D\hu|_\infty\rmd s\Big), \quad \text{for any}\,\, (t,x)\in[0,\infty)\times\R^d,
\end{equation}
where the constant $C>0$ depends only on $u_0(0)$.
\end{lemma}
\begin{proof} Throughout the proof, set $m=s$. We divide the proof into 4 steps \textbf{(I)}--\textbf{(IV)}:
\begin{itemize}
\item  \textbf{(I):} giving the global existence of a $C^1$ solution    to  \eqref{burgers}, which satisfies  (1);
\item \textbf{(II):}  proving  (2) and showing that  this  $C^1$ solution  indeed exists in the space  $\Xi_m$ globally in time, where
$$
				\|f\|_{\Xi_m}=|\nabla f|_\infty+\|\nabla^2 f\|_{m-1} \ \  \text{for \ any} \ m>1+d/2;
$$
\item    \textbf{(III)-(IV):} establishing the estimates in   (3) and \eqref{galilean}.
\end{itemize}

\textbf{(I)} The  local existence of $\widehat{u}(t,x)$ in the space $C^1$  can be established  from   the  characteristic  method and the implicit function theorem  (see \S 3 of  Evans \cite{Evans2010}). For extending $\widehat{u}$ to be a global $C^1$ solution, since $\widehat{u}(t, \Omega(t;x_0))$ remains invariant   along the particle path $\Omega(t;x_0)$,  one only  needs to give the uniform estimates on $|\nabla_x \widehat{u}|_\infty$ with respect to the time.

For the first order derivative, use the convention $(\D v)_{ab}=\pa_{x_a}v^{(b)}$ and let $\mathbb{G}(t,x)=\nabla \widehat{u}(t,x)$ and $\mathbb{G}_0(x_0)=\nabla u^0(x_0)$. Then
\begin{equation}\label{fujiajjj}
\begin{split}
					\mathbb{G}(t,\Omega(t;x_0))=&\big(\mathbb{I}_d+t\mathbb{G}_0(x_0)\big)^{-1}\mathbb{G}_0(x_0)\\
					=&\mathbb{G}_0(x_0)\big(\mathbb{I}_d+t\mathbb{G}_0(x_0)\big)^{-1} = \frac{1}{t} \Big( \mathbb{I}_d- \big(\mathbb{I}_d+t\mathbb{G}_0(x_0)\big)^{-1}\Big),
\end{split}
\end{equation}
where one has used the fact $\nabla_{x_0} \Omega= \mathbb{I}_d+t\mathbb{G}_0(x_0)$. Note  that
\begin{equation}\label{fujia1}
				\big(\mathbb{I}_d+t\mathbb{G}_0(x_0)\big)^{-1}=\big(\text{det}(\mathbb{I}_d+t\mathbb{G}_0)\big)^{-1}(\text{adj} (\mathbb{I}_d+t\mathbb{G}_0)),
\end{equation}
where $\text{adj} (\mathbb{I}_d+t\mathbb{G}_0)$ stands for the adjugate  of $(\mathbb{I}_d+t\mathbb{G}_0)$. Then it holds that
\begin{equation}\label{wuqiong}
				|\nabla_x \widehat{u}|_\infty\leq \frac{C(1+t|\nabla u^0|_\infty)^{d-1}}{(1+t\kappa)^d}|\nabla_{x_0} u^0|_\infty,
\end{equation}
which implies that $\widehat{u}$ is a global $C^1$  solution to the problem  \eqref{burgers}.  Rewrite
\[
	\mathbb{G}(t,\Omega(t;x_0))= \frac{1}{1+t}\mathbb{I}_d+\frac{1}{(1+t)^2}K(t,x_0),
\]
where $K(t,x_0)=(1+t)^2(\mathbb{I}_d+t\mathbb{G}_0)^{-1} \mathbb{G}_0-(1+t)\mathbb{I}_d$.

Since $ \mathbb{G}^{-1}_0=\big(\text{det}(\mathbb{G}_0)\big)^{-1}(\text{adj} \mathbb{G}_0) $,  one has
\[
	|\mathbb{G}^{-1}_0|_\infty\leq C\kappa^{-d}|\mathbb{G}_0|^{d-1}_\infty.
\]
Then for $t$ large enough, one has $|t^{-1}\mathbb{G}^{-1} _0(x_0)|<1$ for all $x_0$, and
\begin{equation*}\begin{split}
	K(t,x_0)=& \frac{(t+1)^2}{t} (\mathbb{I}_d+t^{-1}\mathbb{G}^{-1}_0)^{-1}-(1+t)\mathbb{I}_d\\
	=& \frac{(t+1)^2}{t} \Big(\mathbb{I}_d-\frac{\mathbb{G}^{-1}_0}{t}+O\Big(\frac{1}{t^2}\Big)\Big)-(1+t)\mathbb{I}_d\\
	= & \frac{(t+1)}{t} \mathbb{I}_d- \frac{(t+1)^2}{t^2} \mathbb{G}^{-1} _0+O\Big(\frac{1}{t}\Big).
\end{split}
\end{equation*}

\begin{remark}
Let $\lambda$ be  an eigenvalue of $\mathbb{G}_0$.  Then identity \eqref{fujiajjj} implies the corresponding eigenvalue $\lambda_{\mathbb{G}}$ of $\mathbb{G}$  satisfies
\[
	\lambda_{\mathbb{G}}=\frac{1}{t} \Big( 1- \big(1+t\lambda\big)^{-1}\Big)=\frac{\lambda}{1+t\lambda}.
\]
\end{remark}

\textbf{(II)} Now we show  that the  $C^1$ solution $\widehat{u}(t,x)$ obtained above  indeed exists in  $\Xi_m$ globally in time. To this end, one needs to show that $\nabla^k_x \mathbb{G}(t,x)\in L^2(\mathbb{R}^d)$ ($1\leq k\leq m$) for any $t\in [0,\infty)$. Actually, it follows from  $\widehat{u}(t,x)=u^0(x-t\widehat{u}(t,x))$ that
$$\mathbb{G}(t,x)=\big(\mathbb{I}_d+t\mathbb{G}_0(x-t\widehat{u}(t,x))\big)^{-1}\mathbb{G}_0(x-t\widehat{u}(t,x)).$$
The following calculations are first performed for smooth initial data and then justified under the stated assumptions by approximation. What we only need to do is to give  the estimates (2) in Lemma \ref{hu}. However, the information  of $\nabla^k_x \mathbb{G}(t,x)$  cannot be given directly from the problem \eqref{burgers}. Thus  we  consider $\nabla^k_x \mathbb{G}(t,\Omega(t;x_0))$ first, and then   obtain the desired estimates via a change of variables.

For the high order derivatives, let $\mathbb{H}(t,x_0)=\mathbb{G}(t,\Omega(t;x_0))$. Starting from the right hand side of the first line in  \eqref{fujiajjj}, by induction, one can get    that, for $1\leq k\leq m$:
\begin{equation}\label{juzhen1}
				\nabla^k_{x_0} \mathbb{H}(t,x_0)=(\mathbb{I}_d+t\mathbb{G}_0)^{-1}\Lambda_k (\mathbb{I}_d+t\mathbb{G}_0)^{-1},
\end{equation}
where $\Lambda_k$ is a sum of products of $t(\mathbb{I}_d+t\mathbb{G}_0)^{-1}$ and $\nabla^j \mathbb{G}_0, \ j\in \{0,1,...,k\}$, appearing $\beta_j$ times with $\sum_j j\beta_j=k$.

On the one hand,  starting from the left hand side of the first line in   \eqref{fujiajjj}, one gets by induction  that, for $1\leq k\leq m$:
\begin{equation}\label{juzhen2}
				\nabla^k_{x_0} \mathbb{H}(t,x_0) =\sum_{j=1}^k \nabla^j_{x} \mathbb{G}(t,\Omega(t;x_0)) \Big(\sum_{1\leq k_i\leq k} c_{k_1,\cdots,k_j}\nabla^{k_1}_{x_0}\Omega\otimes ...\otimes \nabla^{k_j}_{x_0}\Omega\Big)
\end{equation}
where $\sum_{i=1}^j k_i=k$ and $c_{k_1,\cdots,k_j}$ denote the corresponding combinatorial coefficients in the chain-rule expansion. Moreover, one has
\begin{equation*}\begin{split}
					\nabla_{x_0} \Omega=& \mathbb{I}_d+t\mathbb{G}_0(x_0),  \quad \nabla^{l}_{x_0}\Omega=t\nabla^{l-1}\mathbb{G}_0(x_0), \quad \text{for} \quad l\geq 2.
\end{split}
\end{equation*}

On the other hand, one can also  show that for all $1\leq j \leq m$ by induction:   \textbf{IH(j)}.

\begin{itemize}
\item[$(\rm  \textbf{IH(j)}_1)$] $\nabla^j_x \mathbb{G}(t,\Omega(t;x_0))$ is a sum of terms which are products in a certain order of $(\mathbb{I}_d+t\mathbb{G}_0)^{-1}$, $t\mathbb{I}_d$,  $\mathbb{I}_d+t\mathbb{G}_0$ or  $\nabla^l \mathbb{G}_0$ appearing $\beta_l$ times, with $\sum_l l\beta_l=j$;
\item[$(\rm\textbf{IH(j)}_2)$] the $L^\infty$-norms of the terms with $t$ are bounded by a constant times $(1+t)^{-(j+2)}$, and
$$
				|\nabla^j_x \mathbb{G}(t,\Omega(t;\cdot))|_2\leq C_j(1+ t)^{-(j+2)},\quad \text{with} \quad C_j=C(\kappa, j, \|u^0\|_{\Xi_m}).
$$
\end{itemize}

Actually, when $j=1$, the chain rule gives, for every $h\in\mathbb{R}^d$,
\begin{equation}\label{erjieQ}
\begin{aligned}
					\mathrm{D}_x\mathbb{G}(t,\Omega(t;x_0))[h] ={}&(\mathbb{I}_d+t\mathbb{G}_0(x_0))^{-1}\\
					&\quad\cdot\mathrm{D}\mathbb{G}_0(x_0) \big[(\mathbb{I}_d+t\mathbb{G}_0(x_0))^{-\top}h\big] (\mathbb{I}_d+t\mathbb{G}_0(x_0))^{-1},
\end{aligned}
\end{equation}
where $\mathrm{D}F(x)[h]$ denotes the directional derivative of $F$ at $x$ in the direction $h$. This identity, together with the estimates obtained in \textbf{(I)}, yields \textbf{IH(1)}.

Now suppose that  \textbf{IH(k-1)} holds. For $\nabla^k_{x} \mathbb{G}(t,\Omega(t;x_0)) $, it   follows from \eqref{juzhen2} that
\begin{equation*}\begin{split}
					\nabla^k_{x} \mathbb{G}(t,\Omega(t;x_0)) =&\Big(\nabla^k_{x_0}\mathbb{H}(t,x_0)-\mathbb{J}\Big) \odot \Big((\mathbb{I}_d+t\mathbb{G}_0)^{-1}\Big)^{\otimes k},
\end{split}
\end{equation*}
where
\[
			\mathbb{J}=\sum_{j=1}^{k-1} \nabla^j_{x} \mathbb{G}(t,\Omega(t;x_0)) \Big(\sum_{1\leq k_i\leq k} c_{k_1,\cdots,k_j}\nabla^{k_1}_{x_0}\Omega\otimes ...\otimes \nabla^{k_j}_{x_0}\Omega\Big),
\]
with $\sum_{i=1}^j k_i=k$. In the right-hand side term,  one has the norm of the following terms to estimate:
\begin{itemize}
\item $O_1=(\mathbb{I}_d+t\mathbb{G}_0)^{-1}\Lambda_k (\mathbb{I}_d+t\mathbb{G}_0)^{-k-1}$;
\item  $O_2=\nabla^j_{x} \mathbb{G}(\mathbb{I}_d+t\mathbb{G}_0)^{j-r_0} \amalg_{k_i \neq 1}t\nabla^{k_i-1}\mathbb{G}_0(\mathbb{I}_d+t\mathbb{G}_0)^{-k}$,\ \  $\sum_{k_i\neq 1}(k_i-1)=k-j$,
\end{itemize}
where one has used  \eqref{juzhen1}, and $r_0$ is the number of indices $i$ for which $k_i\ne1$.

First,  $O_1$ has the desired form in $(\rm  \textbf{IH(k)}_1)$,  and the $L^\infty$-norm of the terms with $t$ is bounded by a constant times $(1+t)^{-(k+2)}$.

Second, it follows from  \textbf{IH(j)} for $j\leq k-1$  that  $O_2$ is a product of $(\nabla^j \mathbb{G}_0)^{\beta_j}$ with $\sum_j j\beta_j=k$ and of $t \mathbb{I}_d$, $\mathbb{I}_d+t\mathbb{G}_0$,  $(\mathbb{I}_d+t\mathbb{G}_0)^{-1}$, such that the $L^\infty$-norm of the terms with $t$ is bounded by a constant times $(1+ t)^{-(k+2)}$.

For the desired $L^2$ estimates,  it follows from the initial assumptions on $u^0$, $m>1+\frac{d}{2}$ and  the  Gagliardo-Nirenberg inequality that one  can  find an upper bound in $L^2$-norm for $\amalg_{1\leq j\leq k} \big(\nabla^j \mathbb{G}_0\big)^{\beta_j}$ with $\sum_j j\beta_j=k$. Similarly, with the help of \textbf{IH(j)} for $j\leq k-1$, one can also obtain an upper bound of $|\nabla^j_{x} \mathbb{G} \amalg_{k_i \neq 1}\nabla^{k_i-1}\mathbb{G}_0|_2$ with $\sum_{k_i\neq 1}(k_i-1)=k-j$.  Then
\begin{equation*}\begin{split}
					|O_1|_2\leq  C(1+ t)^{-k-2}|\Lambda_k|_2\leq C(1+t\kappa)^{-k-2}|\amalg \big(\nabla^j \mathbb{G}_0\big)^{\beta_j}|_2 \leq& C_k(1+t)^{-k-2},\\
					|O_2|_2\leq  C(1+t)^{-k+j}|\nabla^j_{x} \mathbb{G} \amalg_{k_i \neq 1}\nabla^{k_i-1}\mathbb{G}_0|_2 \leq & C_k(1+ t)^{-k-2},
\end{split}
\end{equation*}
where the positive constant $C_k= C(\kappa, k, \|u^0\|_{\Xi_m})$. To conclude, one obtains the upper bound in \textbf{IH(k)} which depends on $\kappa$, $|\mathbb{G}_0|_\infty$,  and $|\nabla^k \mathbb{G}_0|_2$ for $1\leq k \leq m$.

Finally, a change of variables yields:
\begin{equation}\label{fujiajia0}
\begin{split}
					|\nabla^j_x \mathbb{G}(t,\cdot)|_2=& \Big(\int_{\mathbb{R}^d} |\nabla^j_x \mathbb{G}(t,x)|^2 \rmd x\Big)^{\frac{1}{2}}\\
					=&\Big(\int_{\mathbb{R}^d} |\nabla^j_x \mathbb{G}(t,\Omega(t;x_0))|^2 |\text{det}(\mathbb{I}_d+t\mathbb{G}_0)|  \rmd x_0\Big)^{\frac{1}{2}},
\end{split}
\end{equation}
which, along with \textbf{IH(j)}, yields that
$$
			|\nabla^j_x \mathbb{G}(t,\cdot)|_2 \leq C_j(1+t)^{\frac{d}{2}-(j+2)}.
$$
Therefore,  the conclusion  (2) in Lemma \ref{hu} is true for $l\in \mathbb{N},\ 2\leq l\leq  m+1$ since $\nabla_x \mathbb{G}=\nabla^2 \widehat{u}$.  Then by interpolation one obtains the result for all $l\in \mathbb{R},\ 2\leq l\leq \lfloor m \rfloor+1$.

Until now, one has established the global-in-time well-posedness  of the solutions $\widehat{u}$ to Burgers equations \eqref{burgers} in  $\Xi$.

\textbf{(III)} Since $m-1>\frac{d}{2}$, one has $\nabla^2 u^0\in L^\infty$. Thus $\nabla^2 \widehat{u} \in L^\infty$. Then the estimates (3) in Lemma \ref{hu} follow directly from \eqref{erjieQ} and the estimates obtained in \textbf{(I)}, i.e.,
$$|\nabla_x \mathbb{G}(t,\Omega(t; x_0))|_\infty=O((1+ t)^{-3}).$$

%\textbf{(IV)} Considering the  point $x_0\in \mathbb{R}^d$ satisfying  $$(t,x=x_0+tu^0(x_0))\in [0,T]\times B_N(0).$$
%Due to
%$
% \widehat{u}(t,x)=u^0(x_0)=u^0(0)+\nabla u^0(a x_0)\cdot x_0
%$,
%where $a\in [0,1]$  is a constant,  we need to control the value of $x_0$.
%Noticing
%$$
%x=x_0+t u^0(x_0)=x_0+tu^0(0)+t\nabla u^0(a x_0)\cdot x_0,
%$$
%and  multiplying  the above equality by $x_0$ on both sides,  we have
%$$
%(1+t\kappa ) |x_0| \leq |x|+t|u^0(0)|=N+t|u^0(0)|,
%$$
%which means that
%$$
%|x_0|\leq \frac{N+t|u^0(0)|}{1+t\kappa}.
%$$

\textbf{(IV)} Integrating the Burgers equations with respect to $t$ yields
\begin{equation}\label{Taylor0000}
				\hu(t,x)=\hu(0,x)-\int_0^t (\hu\cdot\D)\hu\rmd s.
\end{equation}
Applying Taylor's formula to $\hu(0,\cdot)$ about $x=0$, one obtains
\begin{equation}\label{eq8-2}
				|\hu(0,x)|\leq |\hu(0,0)|+|\D\hu(0,\cdot)|_\infty|x|\leq C(1+|\D \hu^0|_\infty)(1+|x|).
\end{equation}
Substituting \eqref{eq8-2} into \eqref{Taylor0000} leads to 
\[
			|\hu(t,x)|\leq C(1+|\D \hu^0|_\infty)(1+|x|)+\int_0^t|\hu(s,x)||\D\hu(s,\cdot)|_\infty\rmd s.
\]
Applying Gr\"{o}nwall's inequality to \eqref{Taylor0000} at $x=0$ yields
\[
			|\hu(t,0)|\leq |u^0(0)|\exp\Big(\int_0^t|\D\hu(s,\cdot)|_\infty\rmd s\Big).
\]
Combining this with
\[
			|\hu(t,x)|\leq |\hu(t,0)|+|\D\hu(t,\cdot)|_\infty|x|
\]
proves \eqref{galilean}.
\end{proof}

\bigskip
{\bf Acknowledgement:} This research is partially supported by the National Key R$\&$D Program of China (No. 2022YFA1007300), the Fundamental Research Funds for the Central Universities, National Natural Science Foundation of China under Grants 12101395 and 12161141004, and  The Royal Society (UK) under Grants: Newton International Fellowships Alumni AL/201021 and AL/211005.

\bigskip
{\bf Conflict of Interest:} The authors declare that they have no conflict of interest.

\bigskip
{\bf Data availability:} Data sharing is  not applicable to this article as no data sets were generated or analysed during the current study.

\raggedbottom

%\bibliographystyle{abbrv}
%\bibliography{degenerate-EP}

\begin{thebibliography}{10}
% Slightly tighten reference leading so the author addresses fit on the last page.
\setlength{\baselineskip}{10.3pt}

\bibitem{Acheritogaray2011}
M.~Acheritogaray, P.~Degond, A.~Frouvelle, and J.-G. Liu.
\newblock Kinetic formulation and global existence for the {H}all-{M}agneto-hydrodynamics system.
\newblock {\em Kinet. Relat. Models}, 4(4):\,\,901--918, 2011.

\bibitem{Ali2003}
G.~Al\`i.
\newblock Global existence of smooth solutions of the {N}--dimensional {E}uler--{P}oisson model.
\newblock {\em SIAM J. Math. Anal.}, 35(2):\,\,389--422, 2003.

\bibitem{Ali2010}
G.~Al\`i, L.~Chen, A.~J\"{u}ngel, and Y.-J. Peng.
\newblock The zero--electron--mass limit in the hydrodynamic model for plasmas.
\newblock {\em Nonlinear Anal.}, 72(12):\,\,4415--4427, 2010.

\bibitem{Athanasiou2023}
N.~Athanasiou, T.~Bayles-Rea, and S.~Zhu.
\newblock Development of singularities in the relativistic {E}uler equations.
\newblock {\em Trans. Amer. Math. Soc.}, 376(4):\,\,2325--2372, 2023.

\bibitem{Aubin1982}
T.~Aubin.
\newblock {\em Nonlinear analysis on manifolds. {M}onge-{A}mp\`ere equations}, Vol. 252.
\newblock Springer-Verlag, New York, 1982.

\bibitem{Blanc2021}
X.~Blanc, R.~Danchin, B.~Ducomet, and {\v S}.~Ne\v{c}asov\'{a}.
\newblock The global existence issue for the compressible {E}uler system with {P}oisson or {H}elmholtz couplings.
\newblock {\em J. Hyperbolic Differ. Equ.}, 18(1):\,\,169--193, 2021.

\bibitem{Chen1984}
F.~Chen.
\newblock {\em Introduction to {P}lasma {P}hysics and {C}ontrolled {F}usion}, Vol. 1.
\newblock Plenum Press, 1984.

\bibitem{Chen2021}
G.~Chen, G.-Q.~G. Chen, and S.~Zhu.
\newblock Formation of singularities and existence of global continuous solutions for the compressible {E}uler equations.
\newblock {\em SIAM J. Math. Anal.}, 53(6):\,\,6280--6325, 2021.








\bibitem{Chen2017}
G.~Chen, R.~Pan, and S.~Zhu.
\newblock Singularity formation for the compressible {E}uler equations.
\newblock {\em SIAM J. Math. Anal.}, 49(4):\,\,2591--2614, 2017.

\bibitem{Chen2002}
G.-Q.~G. Chen, and D.~Wang.
\newblock \newblock {\em The Cauchy problem for the Euler equations for compressible fluids}. Handbook of mathematical fluid dynamics, Vol. I, 421–543, North-Holland, Amsterdam, 2002. 

\bibitem{Lin1}
M. Cheng, X. Cheng, and Z. Lin. Expanding solutions near unstable Lane--Emden stars. {\em Comm. Math. Phys.}, 406(7): \ \ Paper No. 165, 2025.

\bibitem{Danchin2021}
R.~Danchin and B.~Ducomet.
\newblock On the global existence for the compressible {E}uler-{P}oisson system, and the instability of static solutions.
\newblock {\em J. Evol. Equ.}, 21(3):\,\,3035--3054, 2021.

\bibitem{Evans2010}
L.~C. Evans.
\newblock {\em Partial Differential Equations}, Vol. 19.
\newblock American Mathematical Society, Providence, 2010.

\bibitem{GerardVaret2013}
D.~G\'{e}rard-Varet, D.~Han-Kwan, and F.~Rousset.
\newblock Quasineutral limit of the {E}uler--{P}oisson system for ions in a domain with boundaries.
\newblock {\em Indiana Univ. Math. J.}, 62(2):\,\,359--402, 2013.

\bibitem{Germain2013}
P.~Germain, N.~Masmoudi, and B.~Pausader.
\newblock Nonneutral global solutions for the electron {E}uler--{P}oisson system in three dimensions.
\newblock {\em SIAM J. Math. Anal.}, 45(1):\,\,267--278, 2013.

\bibitem{Gilbarg1983}
D.~Gilbarg and N.~S. Trudinger.
\newblock {\em Elliptic partial differential equations of second order}, Vol. 224.
\newblock Springer-Verlag, Berlin, second edition, 1983.


%\bibitem{Goudon1999}
%T.~Goudon, A.~J\"{u}ngel, and Y.-J. Peng.
%\newblock Zero--mass--electrons limits in hydrodynamic models for plasmas.
%\newblock {\em Appl. Math. Lett.}, 12(4):\,\,75--79, 1999.

\bibitem{Grassin1998}
M.~Grassin.
\newblock Global smooth solutions to {E}uler equations for a perfect gas.
\newblock {\em Indiana Univ. Math. J.}, 47(4):\,\,1397--1432, 1998.

\bibitem{Gu2016}
X.~Gu and Z.~Lei.
\newblock Local well-posedness of the three dimensional compressible {E}uler-{P}oisson equations with physical vacuum.
\newblock {\em J. Math. Pures Appl. (9)}, 105(5):\,\,662--723, 2016.

\bibitem{Guo2011}
Y.~Guo and B.~Pausader.
\newblock Global smooth ion dynamics in the {E}uler-{P}oisson system.
\newblock {\em Commun. Math. Phys.}, 303(1):\,\,89--125, 2011.

\bibitem{Guo2014a}
Y.~Guo and X.~Pu.
\newblock Kd{V} limit of the {E}uler-{P}oisson system.
\newblock {\em Arch. Ration. Mech. Anal.}, 211(2):\,\,673--710, 2014.

\bibitem{Strauss1}
Y. Guo and W. Strauss. Stability of semiconductor states with insulating and contact boundary conditions. {\em Arch. Ration. Mech. Anal.}, 179(1): \ \ 1--30, 2006.

\bibitem{Hadzic2019}
M.~Had\v{z}i\'{c} and J.~Jang.
\newblock A class of global solutions to the {E}uler-{P}oisson system.
\newblock {\em Comm. Math. Phys.}, 370(2):\,\,475--505, 2019.

\bibitem{Hsiao2000}
L.~Hsiao and K.~Zhang.
\newblock The global weak solution and relaxation limits of the initial-boundary value problem to the bipolar hydrodynamic model for semiconductors.
\newblock {\em Math. Models Methods Appl. Sci.}, 10(9):\,\,1333--1361, 2000.

\bibitem{Hsiao2003}
L.~Hsiao, P.~A. Markowich, and S.~Wang.
\newblock The asymptotic behavior of globally smooth solutions of the multidimensional isentropic hydrodynamic model for semiconductors.
\newblock {\em J. Differential Equations}, 192(1):\,\,111--133, 2003.

\bibitem{Huang2012}
F.~Huang, M.~Mei, Y.~Wang, and T.~Yang.
\newblock Long-time behavior of solutions to the bipolar hydrodynamic model of semiconductors with boundary effect.
\newblock {\em SIAM J. Math. Anal.}, 44(2):\,\,1134--1164, 2012.

\bibitem{Jiang1995}
X.~L. Jiang.
\newblock A streamline-upwinding/{P}etrov-{G}alerkin method for the hydrodynamic semiconductor device model.
\newblock {\em Math. Models Methods Appl. Sci.}, 5(5):\,\,659--681, 1995.

\bibitem{Kateb2003}
D.~Kateb.
\newblock On the boundedness of the mapping {$f\mapsto |f|^\mu,\,\mu>1$} on {B}esov spaces.
\newblock {\em Math. Nachr.}, 248/249:\,\,110--128, 2003.

\bibitem{Kato1975}
T.~Kato.
\newblock The {C}auchy problem for quasi--linear symmetric hyperbolic systems.
\newblock {\em Arch. Ration. Mech. Anal.}, 58(3):\,\,181--205, 1975.

\bibitem{Kato1988}
T.~Kato and G.~Ponce.
\newblock Commutator estimates and the {E}uler and {N}avier-{S}tokes equations.
\newblock {\em Commun. Pure Appl. Math.}, 41(7):\,\,891--907, 1988.

\bibitem{Lax1973}
P.~D. Lax.
\newblock {\em Hyperbolic systems of conservation laws and the mathematical theory of shock waves}, Vol. ~11.
\newblock SIAM Regional Conf. Lecture, Philadelphia, 1973.

%\bibitem{Li2018}
%Y.~Li, Y.-J. Peng, and S.~Xi.
%\newblock Rigorous derivation of a {B}oltzmann relation from isothermal
%  {E}uler-{P}oisson systems.
%\newblock {\em J. Math. Phys.}, 59(12):\,\,Paper No. 123501, 14, 2018.

\bibitem{Li2012}
Y.~Li and X.~Yang.
\newblock Global existence and asymptotic behavior of the solutions to the three--dimensional bipolar {E}uler--{P}oisson systems.
\newblock {\em J. Differential Equations}, 252(1):\,\,768--791, 2012.

\bibitem{Lin2}
Z. Lin and Y. Wang. Stability of rotating gaseous stars. {\em Comm. Math. Phys.}, 402(2):\ \ 1725--1763, 2023.

%\bibitem{Liu2019a}
%C.~Liu and Y.-J. Peng.
%\newblock Global convergence of the {E}uler--{P}oisson system for ion dynamics.
%\newblock {\em Math. Methods Appl. Sci.}, 42(4):\,\,1236--1248, 2019.

%\bibitem{Luo1999}
%T.~Luo, R.~Natalini, and Z.~Xin.
%\newblock Large time behavior of the solutions to a hydrodynamic model for semiconductors.
%\newblock {\em SIAM J. Appl. Math.}, 59(3):\,\,810--830, 1999.

%\bibitem{Luo2014}
%T.~Luo, Z.~Xin, and H.~Zeng.
%\newblock Well-posedness for the motion of physical vacuum of the
%  three-dimensional compressible {E}uler equations with or without
%  self-gravitation.
%\newblock {\em Arch. Ration. Mech. Anal.}, 213(3):\,\,763--831, 2014.

%\bibitem{Luo2019}
%T.~Luo, S.~Wang, and Y.-L. Wang.
%\newblock Initial layer and incompressible limit for {E}uler-{P}oisson equation
%  in nonthermal plasma.
%\newblock {\em Math. Models Methods Appl. Sci.}, 29(9):\,\,1733--1751, 2019.

\bibitem{Majda1984}
A.~Majda.
\newblock {\em Compressible fluid flow and systems of conservation laws in several space variables}, Vol. ~53.
\newblock Springer-Verlag, New York, 1984.

\bibitem{Makino1986}
T.~Makino, S.~Ukai, and S.~Kawashima.
\newblock Sur la solution \`a support compact de l'\'{e}quation d'{E}uler compressible.
\newblock {\em Japan J. Appl. Math.}, 3(2):\,\,249--257, 1986.

\bibitem{Markowich1990}
P.~A. Markowich, C.~A. Ringhofer, and C.~Schmeiser.
\newblock {\em Semiconductor equations}.
\newblock Springer-Verlag, Vienna, 1990.

\bibitem{Nirenberg1959}
L.~Nirenberg.
\newblock On elliptic partial differential equations.
\newblock {\em Ann. Scuola Norm. Sup. Pisa Cl. Sci.}, 13:\,\,115--162, 1959.

\bibitem{Nisar2016}
U.~A. Nisar, W.~Ashraf, and S.~Qamar.
\newblock A splitting scheme based on the space-time {CE}/{SE} method for solving multi-dimensional hydrodynamical models of semiconductor devices.
\newblock {\em Comput. Phys. Commun.}, 205:\,\,69--86, 2016.



%\bibitem{Peng2015}
%Y.-J. Peng.
%\newblock Uniformly global smooth solutions and convergence of
% {E}uler--{P}oisson systems with small parameters.
%\newblock {\em SIAM J. Math. Anal.}, 47(2):\,\,1355--1376, 2015.

\bibitem{Serre1997}
D.~Serre.
\newblock Solutions classiques globales des \'{e}quations d'{E}uler pour un fluide parfait compressible.
\newblock {\em Ann. Inst. Fourier (Grenoble)}, 47(1):\,\,139--153, 1997.

\bibitem{Simon1987}
J.~Simon.
\newblock Compact sets in the space {$L^p(0,T;B)$}.
\newblock {\em Ann. Mat. Pura Appl. (4)}, 146:\,\,65--96, 1987.

\bibitem{Sitenko1995}
A.~Sitenko and V.~Malnev.
\newblock {\em Plasma physics theory}, Vol. ~10 of {\em Applied Mathematics and Mathematical Computation}.
\newblock Chapman \& Hall, London, 1995.


\bibitem{Strauss2}
W. Strauss and Y. Wu. Rapidly rotating stars.
\newblock{\em Comm. Math. Phys.}, 368(2): 701--721, 2019.

\bibitem{Wang1998}
D.~Wang and G.-Q.~G. Chen.
\newblock Formation of singularities in compressible {E}uler-{P}oisson fluids with heat diffusion and damping relaxation.
\newblock {\em J. Differential Equations}, 144(1):\,\,44--65, 1998.

% \bibitem{Wang2014}
% Y.~Wang.
% \newblock Formation of singularities to the {E}uler-{P}oisson equations.
% \newblock {\em Nonlinear Anal.}, 109:\,\,136--147, 2014.

\bibitem{Xi2024}
S.~Xi and L.~Zhao.
\newblock From bipolar {E}uler-{P}oisson system to unipolar {E}uler-{P}oisson one in the perspective of mass.
\newblock {\em J. Math. Fluid Mech.}, 26(1):\,\,Paper no. 10, 2024.

\bibitem{Xu2015}
J.~Xu and S.~Kawashima.
\newblock The optimal decay estimates on the framework of {B}esov spaces for the {E}uler--{P}oisson two-fluid system.
\newblock {\em Math. Models Methods Appl. Sci.}, 25(10):\,\,1813--1844, 2015.

\bibitem{Xu2013}
J.~Xu and T.~Zhang.
\newblock Zero-electron-mass limit of {E}uler--{P}oisson equations.
\newblock {\em Discrete Contin. Dynam. Systems}, 33(10):\,\,4743--4768, 2013.


\bibitem{Yuen2011}
M.~Yuen.
\newblock Blowup for the {E}uler and {E}uler-{P}oisson equations with repulsive forces.
\newblock {\em Nonlinear Anal.}, 74(4):\,\,1465--1470, 2011.

\bibitem{Zhao2021}
L.~Zhao.
\newblock The rigorous derivation of unipolar {E}uler{\textendash}{M}axwell system for electrons from bipolar {E}uler{\textendash}{M}axwell system by infinity-ion-mass limit.
\newblock {\em Math. Methods Appl. Sci.}, 44:\,\,3418--3440, 2021.

\bibitem{Zheng2019}
F.~Zheng.
\newblock Long-term regularity of the periodic {E}uler-{P}oisson system for electrons in 2{D}.
\newblock {\em Commun. Math. Phys.}, 366(3):\,\,1135--1172, 2019.











\end{thebibliography}

\end{document}